\documentclass[reqno]{amsart}

\usepackage{amsmath,amssymb,amsthm,dsfont,xcolor,mathtools}
\usepackage[dvipsnames]{xcolor}
\usepackage{hyperref}
\hypersetup{
	linktocpage = true,
	colorlinks = true,
	linkcolor = {RoyalBlue},
	urlcolor = {RoyalBlue},
	citecolor = {Green}}
\usepackage{relsize}
\usepackage[new]{old-arrows}

\newtheorem{theorem}{Theorem}[section]
\newtheorem{lemma}[theorem]{Lemma}
\newtheorem*{theorem*}{Theorem}
\newtheorem{proposition}[theorem]{Proposition}
\newtheorem{corollary}[theorem]{Corollary}

\theoremstyle{definition}
\newtheorem{definition}[theorem]{Definition}

\theoremstyle{remark}
\newtheorem{remark}[theorem]{Remark}
\numberwithin{equation}{section}

\newtheorem{taggedhypx}{Hypotheses}
\newenvironment{taggedhyp}[1]
{\renewcommand\thetaggedhypx{#1}\taggedhypx}
{\endtaggedhypx}

\newtheorem{taggedpropx}{Properties of the admissible triples}
\newenvironment{taggedprop}[1]
{\renewcommand\thetaggedpropx{#1}\taggedpropx}
{\endtaggedpropx}

\usepackage{calrsfs}
\DeclareMathAlphabet{\pazocal}{OMS}{zplm}{m}{n}
\usepackage{centernot}

\usepackage[margin=3.2 cm]{geometry}

\newcommand{\B}{\mathbb{B}}

\newcommand{\R}{\mathbb{R}}

\newcommand{\N}{\mathbb{N}}

\newcommand{\Fpazo}{\pazocal{F}}
\newcommand{\Gpazo}{\pazocal{G}}

\newcommand{\Kpazo}{\pazocal{K}}

\newcommand{\Mpazo}{\pazocal{M}}

\newcommand{\Qpazo}{\pazocal{Q}}
\newcommand{\Hpazo}{\pazocal{H}}

\newcommand{\Rpazo}{\pazocal{R}}
\newcommand{\Spazo}{\pazocal{S}}
\newcommand{\Opazo}{\pazocal{O}}

\newcommand{\Mcal}{\mathcal{M}}

\newcommand{\Lcal}{\mathcal{L}}

\newcommand{\Pcal}{\mathcal{P}}

\newcommand{\Kcal}{\mathcal{K}}

\newcommand{\Id}{\textnormal{Id}}

\newcommand{\supp}{\textnormal{supp}}

\newcommand{\Lip}{\textnormal{Lip}}
\newcommand{\AC}{\textnormal{AC}}
\newcommand{\Graph}{\textnormal{Graph}}

\newcommand{\Div}{\textnormal{div}}

\newcommand{\graph}{\mathrm{graph}}

\newcommand{\dsf}{\textnormal{\textsf{d}}}

\newcommand{\Bnu}{\boldsymbol{\nu}}

\renewcommand{\epsilon}{\varepsilon}

\newcommand{\INTDom}[3]{\int_{#2} #1 \,\textnormal{d} #3}

\newcommand{\INTSeg}[4]{\int_{#3}^{#4} #1 \,\textnormal{d} #2}
\newcommand{\NormL}[3]{\parallel \hspace{-0.075cm} #1 \hspace{-0.075cm} \parallel_{L^{#2}(#3)}}
\newcommand{\NormC}[3]{\left\| #1  \right\| _ {C^{#2}(#3)}}
\newcommand{\Norm}[1]{\parallel \hspace{-0.1cm} #1 \hspace{-0.1cm} \parallel}
\newcommand{\derv}[2]{\frac{\textnormal{d} #1}{ \textnormal{d} #2}}

\newcommand{\Liminf}[1]{\underset{~ #1}{\textnormal{Liminf}}}
\newcommand{\argmin}[1]{\underset{~ #1}{\textnormal{argmin}}}

\DeclareMathOperator*{\esssup}{ess\,sup}

\DeclareMathOperator*{\dom}{dom}

\DeclareMathOperator*{\Ima}{Im}

\begin{document}

\title{Viability theory in the $1$-Wasserstein space}

\author{Benoît Bonnet-Weill}
\address{Laboratoire des Signaux et Systèmes, Université Paris-Saclay, CNRS, CentraleSupélec, 91190 Gif-sur-Yvette, France.}

\author{Alberto Dom\'inguez Corella}
\address{IMJ-PRG, UMR 7586, Sorbonne Université, 4 place Jussieu, 75252 Paris, France.}
\thanks{\textbf{Funding.} The second author was supported by the Alexander von Humboldt Foundation.}

\author{Hélène Frankowska}
\address{CNRS, IMJ-PRG, UMR 7586, Sorbonne Université, 4 place Jussieu, 75252 Paris, France.}

\begin{abstract}
In this article, we establish necessary and sufficient viability conditions for continuity inclusions over the 1-Wasserstein space. Depending on the regularity properties of the dynamics, we derive two results which are based on fairly different proof strategies. When the admissible velocities are Lipschitz in the measure variable,  we characterise viability by the nonemptiness of the intersections of the velocity sets with the graphical derivatives of the constraint map. On the other hand, when the admissible velocities are merely upper semicontinuous in the measure variable, we provide a sufficient condition for viability involving the infinitesimal behaviour of their Aumann integral in time over a neighbouring set of measures.
\end{abstract}

\maketitle

\tableofcontents


\section{Introduction}
In this work, we 
address a central question within both theories of differential inclusions and controlled systems, which is that of \textit{viability}. The latter, studied originally (although with a different terminology) in \cite{Bebernes1970,Nagumo1942}, aims at providing necessary and sufficient conditions ensuring that a solution of a differential inclusion -- or a controlled system -- remains  within a given feasible set at all times. In addition to its numerous applications in many different fields ranging from engineering to economics and life sciences, for which we point e.g. to the excellent surveys \cite{AubinSurvey1990,Aubin2011}, the theory of viability provides some of the most powerful tools to study the well-posedness of Hamilton-Jacobi-Bellman equations with fairly irregular data, as established in the reference works \cite{Frankowska1996,Frankowska1995}. As things stand, there are a few related contributions in mean-field control theory, notably in \cite{Averboukh2018,CavagnariMQ2021} for decentralised problems, and those of \cite{ViabPp} (see also \cite{ViabCDC}) which are centred around the notion of \textit{continuity inclusion} introduced in \cite{ContInc,ContIncPp}. Indeed, in \cite{ViabPp}, one can find general necessary and sufficient conditions for viability of  time-dependent constraint sets under  continuity inclusions in $p$-Wasserstein spaces $\Pcal_p(\R^d)$, whenever $p > 1$. In this context, the Wasserstein metrics enjoy powerful superdifferentiability estimates described, e.g., in \cite[Chapter 10]{AGS}, which allowed -- together with the Cauchy-Lipschitz regularity assumptions posited therein -- to characterise the viability of the constraints in terms of intersections between the admissible velocities of the continuity inclusion and adequate tangent directions to the constraints. Those  contributions, however, were not applicable to the $1$-Wasserstein space, whose geometry is not as smooth as that of its higher-order relatives, but also limited to differential inclusions with sufficiently nice right-hand sides. The aim of the present work is to fill both of these gaps.   

\medskip

Given a set-valued mapping $V : [0,T] \times \Pcal_1(\R^d) \rightrightarrows C^0(\R^d,\R^d)$ whose images are convex compact subsets of Lipschitz and sublinear vector fields, we say following \cite{ContInc,ContIncPp} that an absolutely continuous curve $\mu : [0,T] \to \Pcal_1(\R^d)$ is a solution of the \textit{continuity inclusion} 
\begin{equation}
\label{eq:IntroCI}
\partial_t \mu(t) \in - \Div_x \Big( V(t,\mu(t)) \, \mu(t) \Big) 
\end{equation}
whenever there exists a Lebesgue measurable velocity selection $t \in [0,T] \mapsto v(t) \in V(t,\mu(t))$ for which the standard continuity equation
\begin{equation*}
\partial_t \mu(t) + \Div_x (v(t)\mu(t)) = 0
\end{equation*}
holds in the sense of distributions. Throughout the paper, our goal is to provide necessary and sufficient conditions under which a collection of time-dependent constraint sets modelled by a set-valued mapping $\Qpazo : [0,T] \rightrightarrows \Pcal_1(\R^d)$ is \textit{viable} for \eqref{eq:IntroCI}, that is, such that for each $\tau \in [0,T]$ and every $\mu_{\tau} \in \Qpazo(\tau)$, there exists a solution $\mu : [\tau,T] \to \Pcal_1(\R^d)$ of the latter dynamics that satisfies
\begin{equation*}
\mu(\tau) = \mu_{\tau} \qquad \text{and} \qquad \mu(t) \in \Qpazo(t)    
\end{equation*}
for all times $t \in [\tau,T]$. The relevant geometric object appearing in these conditions is the so-called \textit{graphical derivative} of $\Qpazo : [0,T] \rightrightarrows \Pcal_1(\R^d)$ at some element $(\tau,\nu) \in \Graph(\Qpazo)$, defined by 
\begin{equation}
\label{eq:IntroGraphicalDer}
D\Qpazo(\tau\vert \nu) := \left\lbrace \xi\in L^1(\mathbb R^d, \mathbb R^d; \nu) ~\,\textnormal{s.t.}~ \, \liminf_{h\to 0^+}\frac{1}{h} W_1\Big((\Id + h\xi)_{\sharp}\nu \,; \Qpazo(\tau+h)\Big) = 0 \right\rbrace.
\end{equation}
Heuristically, this set contains all the admissible directions $\xi$ for which the pair $(1,\xi)$ is metrically tangent to the graph of the constraints. Note in particular that when the set $\Qpazo$ is time-independent, its graphical derivative at any $(\tau,\nu) \in [0,T] \times \Qpazo$ reduces to the usual metric Bouligand cone considered e.g. in \cite{Badreddine2022,ViabPp,Lanzetti2024bis}. 

\subsection*{Main contributions} 

In the aforedescribed context, we develop in the present paper two fairly different strategies ensuring the existence of viable curves, depending on the regularity of the admissible velocities, and inspired respectively by \cite{ViabPp} and \cite{Frankowska1995}.


\subsubsection*{Viability in the Lipschitz case}

We start our investigations with the situation in which the velocity map $\mu \in \Pcal_1(\R^d) \rightrightarrows V(t,\mu) \subseteq C^0(\R^d,\R^d)$ is Lipschitz continuous with respect to an adequate norm over the space of sublinear vector fields. In this context, we prove in Theorem \ref{Thm1} that an absolutely continuous constraint mapping $\Qpazo : [0,T] \rightrightarrows \Pcal_1(\R^d)$ with nonempty  proximal\footnote{Recall that a nonempty set $\Qpazo \subseteq \Pcal_1(\R^d)$ is said to be \textit{proximal} if for every $\mu \in \Pcal_1(\R^d)$, there exists $\nu \in \Qpazo$ such that $W_1(\mu,\nu) = W_1(\mu \, ; \Qpazo)$, i.e., if  metric projections always exist for $\Qpazo$.} values is viable for \eqref{eq:IntroCI} if and only if 
\begin{equation}
\label{eq:IntroViabLip}
D\Qpazo(\tau | \nu) \cap V(\tau,\nu) \neq \emptyset
\end{equation}
for $\Lcal^1$-almost every $\tau \in [0,T]$ and all $\nu \in \Qpazo(\tau)$. The main difference between condition \eqref{eq:IntroViabLip} and those derived in \cite{ViabPp} is that, in the latter work, it was possible to take the convex hull of the graphical derivatives (both in time and space), providing thus a larger set of admissible directions to ensure viability. As previously mentioned, this was made possible thanks to the strong superdifferentiability properties enjoyed by the $p$-Wasserstein distances for $p > 1$, which, as illustrated in Appendix \ref{App1} below, do not hold for the $1$-Wasserstein distance. 

The proof that viability is characterised by \eqref{eq:IntroViabLip} in the Lipschitz framework follows the methodology from \cite{ViabPp}, and relies on Gr\"onwall-type estimates between the reachable sets of the continuity inclusion \eqref{eq:IntroCI} and the constraints. While similar in spirit to the latter work, it incidentally required a couple of new ideas to handle the lack of smoothness of the $1$-Wasserstein distance.  


\subsubsection*{Viability in the upper semicontinuous case}	

In this second part, which is the core of the manuscript, we assume that the set-valued mapping $\mu \in \Pcal_1(\R^d) \rightrightarrows V(t,\mu) \subseteq C^0(\R^d,\R^d)$ is merely upper semicontinuous (see Definition \ref{def:Continuity} below) with respect to the topology of uniform convergence on compact sets. This regularity framework, which loosely corresponds to the closedness of the graph of $V(t,\cdot)$, is the weakest one in which one may hope to prove the existence of solutions to a differential inclusion, even in finite dimensional vector spaces, see e.g. \cite[Chapter 2]{Aubin1984}. In this context, which was investigated in the reference work \cite{Frankowska1995} for classical differential inclusions, we establish in Theorem \ref{Thm4} that a left-absolutely continuous constraint mapping $\Qpazo : [0,T] \rightrightarrows \Pcal_1(\R^d)$ with nonempty  proximal values is viable for \eqref{eq:IntroCI} whenever
\begin{equation}
\label{eq:IntroViabUSC}
D \Qpazo(\tau\vert \nu) \cap \Liminf{h \to 0^+} \, \frac{1}{h} \INTSeg{V \big(s,\B(\nu,r) \big)}{s}{\tau}{\tau+h} \neq \emptyset
\end{equation}
for $\Lcal^1$-almost every $\tau \in [0,T]$, all $\nu \in \Qpazo(\tau)$ and each $r > 0$, where the limit is understood in the sense of Kuratowski-Painlevé and the integral in that of Aumann, see e.g. \cite[Chapters 1 and 8]{Aubin1990}.
To provide some intuition, condition \eqref{eq:IntroViabUSC} means that, up to enlarging the set of the admissible velocities by looking at a small Wasserstein ball around a point inside the constraint set, there should always be some direction $\xi_h := \frac1h\INTSeg{v_{h}(s)}{s}{\tau}{\tau+h}$ given as the time-average of an admissible field, such that $\{\xi_h\}_{h > 0}$ converges to $\xi$ with  $(1,\xi)$ being tangent to the graph of the constraints.  Note that the use of the neighbourhood $\B(\nu,r)$ enlarges the set of admissible Aumann integrals, since $V(s,\nu) \subseteq V(s,\B(\nu,r))$. For this reason, condition \eqref{eq:IntroViabUSC} looks less stringent than the analogous time-averaged condition in which the measure argument is frozen at $\nu \in \Qpazo(\tau)$. A comparison between \eqref{eq:IntroViabUSC} and the pointwise condition \eqref{eq:IntroViabLip} is further discussed in Remark \ref{rmlic} below. We stress that the main progress made in this work with respect to \cite{ViabPp} concerns the regularity of the dynamics, as in \cite{ViabPp}, the results are confined to velocities which are Lipschitz continuous in the measure variable, while Theorem \ref{Thm4} only requires 
them to be upper-semicontinuous.

The strategy for proving Theorem \ref{Thm4} is inspired by the general scheme of \cite[Section 4]{Frankowska1995} -- although its implementation in Wasserstein spaces is substantially harder --, and is based on constructing a sequence of trajectories which approximately solve \eqref{eq:IntroCI}, while remaining close to the images of $\Qpazo : [0,T] \rightrightarrows \Pcal_1(\R^d)$. This construction, combined with a compactness argument, provides then a viable curve in the limit. The existence of such a sequence is obtained by applying Zorn's lemma to a carefully chosen collection of triples, comprised of a time horizon, a countable family of intervals and a curve, endowed with a suitable pre-order. We would like to stress that while these results are stated and proven for dynamics over $\Pcal_1(\R^d)$ for the sake of continuity in the exposition, our proof strategy can be transposed almost verbatim to any $p$-Wasserstein space with $p \in [1,+\infty)$. 

\medskip
The manuscript is structured as follows. In Section \ref{section:Preli}, we start by recollecting preliminary notions about abstract integration, optimal transport and dynamics in measure spaces. Then, in Section \ref{section:Lip}, we characterise the viability of absolutely continuous constraint sets when the admissible velocities are Lipschitz continuous in the measure variable, and subsequently treat the more challenging case of upper semicontinuous admissible velocities in Section \ref{section:USC}. Lastly, we provide a couple of appendices whose content aims at shedding light on different parts of the article, and improving its self-containedness. 


\section{Preliminaries}
\label{section:Preli}
Throughout this preliminary section, we recall basic results and concepts from functional analysis, optimal transport, and set-valued dynamics in measure spaces.

\subsection{Function spaces and integration}


\subsubsection*{Spaces of continuous and differentiable functions.}
In what follows, we consider \((\mathbb{R}^d, |\cdot|)\) equipped with the usual Euclidean structure. Given a nonempty compact set $K\subseteq\mathbb R^d$, we denote by $(C^0(K,\R^d),\NormC{\cdot}{0}{K,\R^d})$ the separable Banach space of continuous functions from $K$ into $\mathbb R^d$ endowed with the usual supremum norm.  Its strong dual $\mathcal M(K,\mathbb R^d):=C^0(K,\mathbb R^d)^*$ can be identified with the space of $d$-tuples of signed Radon measures on $K$, equipped with the duality pairing 
\begin{equation*}
\langle \mu, \varphi \rangle_{C^0(K,\mathbb R^d)} := \sum_{i=1}^d\int_{K}\varphi_i(x) \mathrm{d}\mu_{i}(x). 
\end{equation*}
More generally, we denote by $C^0(\mathbb{R}^d, \mathbb{R}^d)$ the space of continuous functions from \(\mathbb{R}^d\) into itself. The latter is classically endowed with the \textit{compact-open topology}, i.e., the coarsest topology that makes each mapping
\begin{equation*}
\iota_K: v \in C^0(\mathbb R^d,\mathbb R^d)  \mapsto  v_{|{K}} \in C^0(K,\mathbb R^d)
\end{equation*}
continuous for every nonempty compact set $K\subseteq \mathbb R^d$. It is a standard result that this topology coincides with the topology of uniform convergence on compact sets (see, e.g., \cite{Warner1958} for more details), which is separable and completely metrisable, e.g., by the translation invariant metric 
\begin{equation*}
\dsf_{cc}(v,w):=\sum_{k=1}^{+\infty}\frac{1}{2^k} \bigg( \frac{\left\| v - w\right\|_{C^0(B(0,k),\mathbb R^d)} }{1 + \left\| v - w\right\|_{C^0(B(0,k),\mathbb R^d)}} \bigg) 
\end{equation*}
which is defined for every $v,w \in C^0(\R^d,\R^d)$. This turns $(C^0(\R^d,\R^d),\dsf_{cc}(\cdot,\cdot))$ into a separable Fréchet space, whose continuous dual $\mathcal M_c(\mathbb R^d,\mathbb R^d) := C^0(\mathbb R^d,\mathbb R^d)^*$ can be identified with the space of $d$-tuples of signed Radon measures with compact support, endowed with the duality pairing 
\begin{equation*}
\langle \mu, \varphi \rangle_{C^0(\mathbb R^d,\mathbb R^d)} := \sum_{i=1}^d\int_{\mathbb R^d}\varphi_i(x) \mathrm{d} \mu_{i}(x).
\end{equation*}
We refer, e.g., to \cite[Proposition 14, Page 155]{Bourbaki2007} for further details on this duality relation. We also consider the total variation norm\footnote{On vector measures supported in a fixed compact set, convergence in total variation norm implies weak-$^*$ convergence. On the whole space $\Mcal_c(\R^d,\R^d)$, these two topologies are incomparable.} over $\Mcal_c(\R^d,\R^d)$, given by 
\[
\| {\mu} \|_{\mathrm{TV}} := \sum_{i=1}^d |\mu_i|(\mathbb{R}^d),
\]
where $|\mu_i|(\mathbb R^d)$ denotes the total variation of the signed measure $\mu_i \in \Mcal_c(\R^d,\R)$. For $\Bnu=(\nu_1,\ldots,\nu_d)$, we set $|\Bnu|:=\sum_{i=1}^d|\nu_i|$, so that $\|\Bnu\|_{\mathrm{TV}}=|\Bnu|(\R^d)$.  

In what follows, we shall also consider the Banach space $(C^0_b(\mathbb R^d,\mathbb R^d),\Norm{\cdot}_{C^0_b(\R^d,\R^d)})$ of continuous bounded functions from $\mathbb R^d$ into itself, endowed with the usual supremum norm. Similarly, we let $C^0_{sl}(\mathbb R^d,\mathbb R^d)$ be the space of sublinear continuous functions, comprised of those functions $\varphi\in C^0(\mathbb R^d,\mathbb R^d)$ for which 
\begin{equation*}
\left\|\varphi\right\|_{C_{sl}^0(\mathbb R^d,\mathbb R^d)}:=\sup_{x\in\mathbb R^d}\frac{|\varphi(x)|}{1+|x|} < +\infty.
\end{equation*}
It can be checked that $(C^0_{sl}(\R^d,\R^d),\Norm{\cdot}_{C^0_{sl}(\R^d,\R^d)})$ is a Banach space, and that the following two embeddings are continuous
\begin{equation*}
\big(C_b^0(\mathbb R^d,\mathbb R^d), \|\cdot\|_{C^0_b(\mathbb R^d,\mathbb R^d)}\big)\longhookrightarrow 	\big(C_{sl}^0(\mathbb R^d,\mathbb R^d), \|\cdot\|_{C_{sl}^0(\mathbb R^d,\mathbb R^d)}\big)\longhookrightarrow 	\big(C_{}^0(\mathbb R^d,\mathbb R^d), \dsf_{cc}(\cdot,\cdot)\big).
\end{equation*}
The set of smooth functions \(\varphi: \mathbb{R}^d \to \mathbb{R}\) with compact support -- commonly referred to as the space of test functions -- is denoted by $C_c^\infty(\mathbb{R}^d,\R)$, whereas $\Lip(\R^d,\R^d)$ stands for the space of Lipschitz continuous functions from $\mathbb R^d$ into itself. 


\subsubsection*{Bochner integration.}

Given a bounded closed subinterval of the real line $I \subseteq \R$ equipped with the standard Lebesgue measure $\Lcal^1$ and a real Banach space $(X,\Norm{\cdot}_X)$, we denote by $\Spazo(I,X)$ the space of so-called $\Lcal^1$-measurable \textit{simple functions}, i.e. those functions which only take a finite number of values. The integral of an element $s\in\Spazo(I,X)$ is then defined by
\begin{equation*}
\INTDom{s(t)}{I}{t} := \sum_{x \in \Ima(s)} \Lcal^1(s^{-1}(x)) x.
\end{equation*}
We recall that a sequence $\{f_n\}_{n\in\mathbb N}$ of functions $f_n:I\to X$ converges $\Lcal^1$-almost everywhere to some $f:I\to X$ provided that
\begin{equation*}
\Lcal^1 \Big(\Big\{ t \in I ~\, \textnormal{s.t.}~ \|f_n(t) - f(t) \|_X \underset{n \to +\infty}{\centernot\longrightarrow 0} \Big\} \Big) = 0. 
\end{equation*}
In this context, a function $f:I\to X$ is said to be \textit{strongly measurable} if it is the $\Lcal^1$-almost everywhere limit of a sequence $\{s_i\}_{i\in\mathbb N}\subseteq\Spazo(I,X)$. Similarly, we shall say that $f : I \to X$ is \textit{scalarly measurable} if for every $\Bnu \in X^*$, the mapping 
\begin{equation*}
t \in I \mapsto \langle \Bnu , f(t) \rangle_X \in \R
\end{equation*}
is $\Lcal^1$-measurable in the usual sense. It can be easily seen that strongly measurable functions are automatically scalarly measurable, and both notions coincide with the usual one involving preimages of Borel sets when the Banach space $(X,\Norm{\cdot}_X)$ is separable (see, e.g., \cite[Chapter II -- Theorem 2]{DiestelUhl}). In the next definition, we recollect some basic terminology and facts about Bochner integration, a detailed account of which can be found, e.g., in \cite[Chapter II]{DiestelUhl} or in the more modern reference \cite[Chapter 1]{AnalysisBanachSpaces}. 

\begin{definition}[Space of Bochner integrable functions]
We say that a map $f : I \to X$ is \textit{Bochner integrable} if it is strongly measurable and such that
\begin{equation*}
\|f\|_{L^1(I,X)} := \int_{I} \|f(t)\|_X\, \,\mathrm{d}t<+\infty.
\end{equation*}
We denote by {$(L^1(I,X),\NormL{\cdot}{1}{I,X})$} the corresponding Banach space of all Bochner integrable functions, and note that the latter is separable whenever $(X,\Norm{\cdot}_X)$ is itself separable. The Bochner integral of an element $f \in L^1(I,X)$ can be classically defined, e.g., by 
\begin{equation*}
\INTDom{f(t)}{I}{t} := \lim_{n \to +\infty} \INTDom{s_n(t)}{I}{t}
\end{equation*}
for any sequence $\{s_n\}_{n \in \N} \subseteq \Spazo(I,X)$ satisfying $\int_I\|f(t)-s_n(t)\|_X\,dt\rightarrow 0$ as $n\rightarrow+\infty$; it can be checked that the limit is independent of the chosen sequence of simple functions. 
\end{definition}

In keeping with what precedes, we shall say that a map $\Bnu: I \to X^*$ is \textit{scalarly-$^*$ measurable} provided that 
\begin{equation*}
t \in I \mapsto \langle \Bnu(t), x \rangle_{X} \in \R
\end{equation*}
is $\Lcal^1$-measurable for all $x\in X$. Below, we recall a simplified and adapted version of a deep compactness result {for the weak topology of $L^1(I,X)$,} borrowed from \cite[Corollary 2.6]{Diestel1993}.

\begin{theorem}[Weak compactness criterion in Bochner spaces]
\label{TheoremcompactnessBochner}
Let $\{f_n\}_{n\in\mathbb N} \subseteq L^1(I,X)$, and suppose that there exist $m(\cdot)\in L^1(I, \mathbb{R}_+)$ and a family $\{K_{t}\}_{t \in I}$ of compact subsets of $X$ such that 
\begin{equation*}
\sup_{n\in\mathbb N}\|f_n(t)\|_X \leq m(t) \qquad \text{and} \qquad \{f_n(t)\}_{n\in\mathbb N} \subseteq K_{t} 
\end{equation*}
for $\Lcal^1$-almost every $t \in I$. Then, there exists a subsequence $\{f_{n_k}\}_{k\in\mathbb N} \subseteq L^1(I,X)$ that converges weakly to some $f \in L^1(I,X)$, and in particular
\begin{equation*}
\int_{I} \langle \Bnu(t), f(t) - f_{n_k}(t) \rangle_X \, \,\mathrm{d}t \underset{k \to +\infty}{\longrightarrow} 0 
\end{equation*}
for every scalarly-$^*$ measurable  $\Bnu:I\to X^*$ satisfying  $\underset{t \in I}{\esssup} \|\Bnu(t)\|_{X^*} < +\infty$. 
\end{theorem}


\subsubsection*{Integration in the space of continuous functions}
For a given nonempty compact set $K\subseteq \mathbb R^d$, it follows from the previous discussions that the Bochner integral of an element $v\in L^1(I,C^0(K,\mathbb R^d))$ can be computed as 
\begin{equation*}
\bigg( \int_{I} v(t) \,\mathrm{d}t \bigg)(x) = \int_I v(t,x) \,\mathrm{d}t
\end{equation*}
for all $x\in K$. This serves as the starting point for introducing a concept of integrability for mappings valued in the space $C^0(\R^d,\R^d)$ of continuous functions defined over the whole of $\R^d$.

\begin{definition}[Bochner integral of $C^0(\R^d,\R^d)$-valued maps]
\label{def:IntegralC0}
We say that a function $v:I\to C^0\big(\mathbb R^d,\mathbb R^d\big)$ is integrable if its restrictions $v_{|K} : I \to C^0(K,\mathbb R^d)$ are Bochner integrable for every nonempty compact set $K\subseteq \mathbb R^d$. In that case, we define its integral by
\begin{equation*}
\bigg( \int_{I}v(t) \,\mathrm{d}t \bigg)(x) := \int_I v(t,x) \,\mathrm{d}t,
\end{equation*}
and observe that $\big( \int_I v(t) \,\mathrm{d}t \big)_{|K} = \int_I v(t)_{|K} \,\mathrm{d}t$ for every nonempty compact set $K \subseteq \R^d$.
\end{definition}

Following previously introduced terminology, a map $\Bnu: I \to \Mcal_c(\R^d,\R^d)$ is scalarly-$^*$ measurable provided that 
\begin{equation*}
t \in I \mapsto \langle \Bnu(t), \varphi \rangle_{C^0(\R^d,\R^d)} \in \R
\end{equation*}
is $\Lcal^1$-measurable for all $\varphi \in C^0(\R^d,\R^d)$.   Employing Theorem \ref{TheoremcompactnessBochner} along with a standard diagonal argument over an exhausting family of compact subsets of $\mathbb R^d$, one can obtain the following handy weak compactness criterion. 

\begin{lemma}[Weak $L^1$-compactness criterion for $C^0(\R^d,\R^d)$-valued maps]
\label{lem:comvf}
Let $\{v_n(\cdot)\}_{n\in\mathbb N}$ be a sequence of functions $v_n:I\to C^0\big(\mathbb R^d,\mathbb R^d\big)$ integrable in the sense of Definition \ref{def:IntegralC0} and suppose that the following holds. 
\begin{itemize}
\item[$(i)$] There exists a map $m(\cdot)\in L^1(I,\mathbb R_+)$ such that 
\begin{equation*}
\sup_{n \in \N} |v_{n}(t,x)|\,\leq m(t)(1+|x|)
\end{equation*}
for $\Lcal^1$-almost every $t\in I$ and all $x\in\mathbb R^d$.
\item[$(ii)$] The family $\{v_n(t)\}_{n\in\mathbb N}$ is locally uniformly equicontinuous for $\Lcal^1$-almost every $t\in I$. 
\end{itemize} 
Then, there exists a subsequence $\{v_{n_k}(\cdot)\}_{k\in\mathbb N}$ and a function $v:I\to C^0(\mathbb R^d,\mathbb R^d)$ integrable in the sense of Definition \ref{def:IntegralC0} such that
\begin{equation*}
\int_I  \langle \Bnu(t) , v(t) - v_{n_k}(t)\rangle_{C^0(\mathbb R^d,\mathbb R^d)}\,\mathrm{d}t \underset{k \to +\infty}{\longrightarrow} 0
\end{equation*}
for every scalarly-$^*$ measurable $\Bnu: I\to \mathcal M_c(\mathbb R^d,\mathbb R^d)$ with $\underset{t \in I}{\esssup} \int_{\mathbb R^d}\big(1+|x|\big)\, \mathrm{d}|\Bnu(t)|(x) < +\infty$. 
\end{lemma}

\begin{proof}
Let us begin by observing that for each radius $R >0$, the sequence of maps $\{f_n\}_{n\in\mathbb N}\subseteq L^1(I, C^0(B(0,R),\mathbb R^d))$ given by $f_n:=v_n(\cdot)_{|B(0,R)}$ satisfies
\begin{align*}
\sup_{n\in\mathbb N}\|f_n(t)\|_{C^0(B(0,R),\mathbb R^d)}\le m(t)(1+R) \qquad \text{and} \qquad \{f_n(t)\}_{n\in\mathbb N}\subseteq K_t:=\operatorname{cl}{\{v_n(t)|_{B(0,R)}\}_{n\in\mathbb N}}
\end{align*}
for $\Lcal^1$-almost every $t \in I$. By our local uniform equicontinuity assumption and the usual Ascoli-Arzelà theorem (see e.g. \cite[Theorem 11.28]{Rudin1987}), the set $K_t \subseteq C^0(B(0,R),\mathbb R^d)$ is compact for $\Lcal^1$-almost every $t \in I$. Thus, we see that the sequence $\{v_n(\cdot)_{|B(0,R)}\}_{n\in\mathbb N}$ satisfies the hypotheses of Theorem \ref{TheoremcompactnessBochner} for each $R > 0$.

Our goal now is to prove the result by means of a diagonal argument. Since $\{v_n|_{B(0,1)}\}_{n\in\mathbb N}$ satisfies the  hypotheses of Theorem \ref{TheoremcompactnessBochner}, one may find an increasing function $\sigma_1:\mathbb N\to\mathbb N$ and a map $v^1 \in L^1(I, C^0(B(0,1),\mathbb R^d)) $ such that  $\{v_{\sigma_1(n)}(\cdot)_{|B(0,1)}\}_{n\in\mathbb N}$ is a subsequence of $\{v_n(\cdot)_{|B(0,1)}\}_{n\in\mathbb N}$ and 
\begin{align}
\int_I \langle \Bnu(t), v_{\sigma_1(n)}(t)_{|B(0,1)} \rangle  \, \mathrm{d}t \underset{n \to +\infty}{\longrightarrow} \int_I \langle \Bnu(t), v^1(t) \rangle \, \mathrm{d}t 
\end{align}
for every scalarly-$^*$ measurable $\Bnu: I \to \mathcal{M}(B(0,1), \mathbb{R}^d)$ with $\esssup_{t \in I} \|\Bnu(t)\|_{\mathcal M(B(0,1),\mathbb R^d)} < +\infty$. Upon repeatedly using Theorem \ref{TheoremcompactnessBochner} for $k\in\mathbb N$, one can find a subsequence $\{v_{\sigma_1\circ\cdots\circ\sigma_k(n)}(\cdot)_{|B(0,k)}\}_{n\in\mathbb N}$ of $\{v_{\sigma_1\circ\cdots\circ\sigma_{k-1}(n)}(\cdot)_{|B(0,k)}\}_{n\in\mathbb N}$ and a map $v^k\in L^1(I,C^0(B(0,k),\mathbb R^d))$ such that 
\begin{align*}
\int_I \big\langle \Bnu(t), v_{\sigma_1 \circ \cdots \circ \sigma_k(n)}(t)_{|B(0,k)} \big\rangle \, \mathrm{d} t ~\underset{n \to +\infty}{\longrightarrow}~ \int_I \langle \Bnu(t), v^k(t)\rangle \,  \mathrm{d}t 
\end{align*}
for every scalarly-$^*$ measurable $\Bnu: I \to \mathcal{M}(B(0,k), \mathbb R^d)$ with $\esssup_{t \in I} \|\Bnu(t)\|_{\mathcal M(B(0,k),\mathbb R^d)} < +\infty$. Consider at present the diagonal extraction 
\begin{align*}
\tau_n := \sigma_1\circ\cdots\circ \sigma_n(n), 
\end{align*}
and observe that $\{v_{\tau_n}(\cdot)\}_{n\in\mathbb N}$ is a subsequence of $\{v_{n}(\cdot)\}_{n\in\mathbb N}$ by construction. Define now the mapping $v:I\to C^0(\mathbb R^d,\mathbb R^d)$ for $\mathcal L^1$-almost every $t\in I$ by setting
\begin{align*}
v(t,x):= v^k(t,x)  
\end{align*}
whenever $x \in B(0,k)$, and note that $v:I\to C^0(\mathbb R^d,\mathbb R^d)$ is well defined and integrable in the sense of Definition \ref{def:IntegralC0} by the above construction. Now, fix some scalarly-$^*$ measurable map $\Bnu: I \to \mathcal{M}_c(\mathbb R^d, \mathbb R^d)$ satisfying $\esssup_{t \in I} \int_{\mathbb R^d}(1+|x|)\,\mathrm{d}|\Bnu(t)|(x) < +\infty$. For each $k\in\mathbb N$, consider the set defined by 
\begin{align*}
E_k:=\Big\lbrace t\in I ~\,\textnormal{s.t.}~  \supp\,\Bnu(t)\subseteq B(0,k)\Big\rbrace. 
\end{align*}
To see that each $E_k$ is $\Lcal^1$-measurable, let
$\{\varphi_j^k\}_{j\in\N}\subset C^0(\R^d,\R^d)$
be a countable dense family for the compact-open topology satisfying $\supp(\varphi_j^k) \subset \R^d \backslash B(0,k)$. Then
\begin{align*}
E_k=\bigcap_{j\in\N}
\Big\lbrace t\in I ~\,\textnormal{s.t.}~
\langle \Bnu(t),\varphi_j^k\rangle_{C^0(\R^d,\R^d)}=0\Big\rbrace,
\end{align*}
and hence $E_k$ is $\Lcal^1$-measurable by the scalar-$^*$ measurability of
$\Bnu : I \to \Mcal_c(\R^d,\R^d)$. Moreover, $E_k\subseteq E_{k+1}$ and
$\bigcup_{k\in\N}E_k$ has full $\Lcal^1$-measure in $I$, since
$\Bnu(t)\in\Mcal_c(\R^d,\R^d)$ for $\Lcal^1$-almost every $t\in I$.
Consequently, $\mathcal L^1(I\setminus E_k)\to0$ as $k\to+\infty$, and it follows from the absolute continuity of the Lebesgue integral that
\begin{align*}
\sup_{n\in\mathbb N}\int_{I\setminus E_k}\bigl|\langle\Bnu(t),v_{\tau_n}(t)-v(t)\rangle\bigr|\,\mathrm{d}t \leq 2 \esssup_{t \in I}\int_{\mathbb R^d}(1+|x|)\, \mathrm{d}|\Bnu(t)|(x)  \int_{I\setminus E_k} m(t)\,\mathrm{d}t ~\underset{k \to +\infty}{\longrightarrow}~ 0. 
\end{align*}
Consequently, there exists for every $\varepsilon>0$ an integer $k_{\epsilon}\in\mathbb N$ such that
\begin{align*}
\sup_{n\in\mathbb N} \int_{I\setminus E_{k_{\epsilon}}}\bigl|\langle\Bnu(t),v_{\tau_n}(t)-v(t)\rangle\bigr|\,\mathrm{d}t<\varepsilon.
\end{align*}
Recalling that $\supp\,\Bnu(t)\subseteq B(0,{k_{\epsilon}})$ for $\mathcal L^1$-almost every $t\in E_{k_{\epsilon}}$ by construction, it also holds that
\begin{equation*}
\begin{aligned}
\int_{E_{k_{\epsilon}}}\langle\Bnu(t),v_{\tau_n}(t)\rangle\,\mathrm{d}t  & = \int_{E_{k_{\epsilon}}}\langle\Bnu(t),v_{\tau_n}(t)|_{B(0,k_{\epsilon})}\rangle\,\mathrm{d}t \\
& \hspace{-0.3cm} \underset{n \to+\infty}{\longrightarrow}~	\int_{E_{k_{\epsilon}}}\langle\Bnu(t),v(t)|_{B(0,k_{\epsilon})}\rangle\,\mathrm{d}t = \int_{E_{k_{\epsilon}}}\langle\Bnu(t),v(t)\rangle\,\mathrm{d}t. 
\end{aligned}
\end{equation*}
By combining both identities, we may easily infer that 
\begin{align*}
\limsup_{n \to +\infty}\bigg|\int_{I}\langle\Bnu(t),v_{\tau_n}(t)-v(t)\rangle\,\mathrm{d}t\bigg| \, \leq \varepsilon,
\end{align*}
which concludes the proof since $\varepsilon>0$ was arbitrary.
\end{proof}

\subsection{Optimal transport and the $1$-Wasserstein space}

In this second preliminary section, we recall some basic notations and results of optimal transport theory, borrowed mainly from \cite[Chapters 5 and 7]{AGS}. 


\subsubsection*{The space of probability measures}
We denote by \(\mathcal{P}(\mathbb{R}^d)\) the space of Borel probability measures on \(\mathbb{R}^d\) endowed with the usual \textit{narrow topology}, that is the coarsest topology for which the maps 
\begin{equation*}
\mu \mapsto  \int_{\mathbb{R}^d} \varphi(x) \, \mathrm{d} \mu(x)
\end{equation*}
are continuous for each $\varphi \in C^0_b(\mathbb{R}^d, \mathbb{R})$. It is well known (see e.g. \cite[Remark 5.1.2]{AGS}) that the latter coincides with the restriction of the weak-$^*$ topology of $C_b^0(\R^d)^*$ to the space of positive measures with unit mass. Given some \(\mu \in \mathcal{P}(\mathbb{R}^d)\), we denote by \(L^1(\mathbb{R}^d, \mathbb{R}^d; \mu)\) the separable Banach space of all {(equivalence classes of)} {$\mu$-measurable} maps \(\xi : \mathbb{R}^d \to \mathbb{R}^d\) satisfying
\begin{equation*}
\|\xi\|_{L^1(\mathbb{R}^d, \mathbb{R}^d;\, \mu)}:= \int_{\mathbb R^d} |\xi(x)| \,\mathrm{d}\mu(x) < +\infty. 
\end{equation*}
Its continuous dual \(L^\infty(\mathbb{R}^d, \mathbb{R}^d; \mu) := L^1(\mathbb{R}^d, \mathbb{R}^d; \mu)^*\) can be identified with the space of \(\mu\)-essentially bounded functions via the standard duality pairing
\begin{equation*}
\langle w, v \rangle_{L^1(\mathbb{R}^d, \mathbb{R}^d ; \, \mu)} = \int_{\mathbb{R}^d} \langle w(x), v(x) \rangle \, \mathrm{d}\mu(x).
\end{equation*}
The image of an element \(\mu \in \mathcal{P}(\mathbb{R}^d)\) through a Borel map  \(\xi : \mathbb{R}^d \to \mathbb{R}^d\) is canonically defined by
\begin{equation*}
\xi_{\sharp} \mu \, (B) := \mu(\xi^{-1}(B))
\end{equation*}
for each Borel set $B \subseteq \R^d$. Since every $\mu$-measurable function $\xi : \R^d \to \R^d$ coincides with a Borel map outside of a $\mu$-negligible set (see e.g. \cite[Corollary 6.5.6]{BogachevII}), one can talk more generally about the image of $\mu \in \mathcal{P}(\mathbb{R}^d)$ through any $\mu$-measurable map $\xi:\mathbb R^d\to\mathbb R^d$, see also the discussion in \cite[Section 3.6 -- Page 191]{BogachevI}. 

Lastly, an element $\gamma\in\mathcal P(\mathbb R^{2d})$ is said to be a \textit{transport plan} -- or \textit{coupling} -- between two measures $\mu,\nu\in\mathcal P(\mathbb R^d)$, provided that 
\begin{equation*}
\pi^1_{\sharp}\gamma = \mu \qquad \text{and} \qquad \pi^2_{\sharp}\gamma = \nu,
\end{equation*}
where \(\pi^1, \pi^2 : \mathbb{R}^d \times \mathbb{R}^d \to \mathbb{R}^d\) are the projections onto the first and second components, respectively. 
The set of all such plans is then denoted by $\Gamma(\mu,\nu)$.


\subsubsection*{The 1-Wasserstein distance}
Given a measure \(\mu \in \mathcal{P}(\mathbb{R}^d)\), we define its first moment by
\begin{equation*}
\Mpazo_1(\mu) := \int_{\mathbb{R}^d} |x| \mathrm{d}\mu(x)
\end{equation*}
and let \(\mathcal P_1(\mathbb{R}^d)\) be the set of all Borel probability measures whose first moment is finite. The \(1\)-Wasserstein distance between two measures \(\mu,\nu\in\mathcal P_1(\mathbb R^d)\) is given by
\begin{equation*}
W_1(\mu, \nu) := \inf \bigg\{ \int_{\mathbb{R}^{2d}} |x - y| \mathrm{d}\gamma(x, y) ~\, \textnormal{s.t.}~ \gamma \in \Gamma(\mu, \nu) \bigg\}, 
\end{equation*}
and in the sequel, we denote  
\begin{equation*}
\Gamma_o(\mu,\nu) := \argmin{\gamma \in \Gamma(\mu, \nu)} \INTDom{|x-y|}{\R^{2d}}{\gamma(x,y)}.
\end{equation*}
the (nonempty) set of so-called \textit{optimal transport plans}. It is a basic fact in optimal transport theory that $(\mathcal P_1(\mathbb{R}^d), W_1(\cdot,\cdot))$ is a complete separable metric space, whose relatively compact subsets $\Kcal \subseteq \Pcal_1(\R^d)$ are characterised by the uniform integrability property  
\begin{equation}
\label{eq:RelativeCompactness}
\sup_{\mu \in \Kcal} \INTDom{|x|}{\{ x \;\textnormal{s.t.}\, |x| \geq k\}}{\mu(x)} ~\underset{k \to +\infty}{\longrightarrow}~ 0.
\end{equation}
Besides, it follows from the very definition of the $1$-Wasserstein distance that 
\begin{equation}
\label{eq:WassEst}
W_1(\xi_{\sharp} \mu, \zeta_{\sharp} \mu) \leq \|\xi - \zeta\|_{L^1(\mathbb{R}^d, \mathbb{R}^d; \, \mu)} \qquad \text{and} \qquad W_1(\phi_{\sharp} \mu,\phi_{\sharp} \nu) \leq \Lip(\phi) W_1(\mu,\nu)
\end{equation}
for every $\mu,\nu \in\mathcal P_1(\mathbb R^d)$, any $\xi,\zeta \in L^1(\R^d,\R^d;\mu)$ and each $\phi \in \Lip(\R^d,\R^d)$. Similarly, it holds that 
\begin{equation}
\label{eq:WassEstBis}
\bigg| \INTDom{\phi(x)}{\R^d}{(\mu-\nu)(x)} \bigg| \leq \INTDom{|\phi(x) - \phi(y)|}{\R^{2d}}{\gamma(x,y)} \leq \Lip(\phi) W_1(\mu,\nu)
\end{equation}
for each $\gamma \in \Gamma_o(\mu,\nu)$. 


\subsection{Set-valued analysis}
In this third preliminary section, we review known results on the regularity, measurability, and geometric properties of set-valued maps, following from the reference monograph \cite{Aubin1990}.


\subsubsection*{Continuity of set-valued mappings}	
Let $(X, \dsf_X(\cdot,\cdot))$ and $(Y,\dsf_Y(\cdot,\cdot))$ be two metric spaces and consider a set–valued mapping \( \Fpazo : X \rightrightarrows Y \), whose domain and graph are defined respectively by  
\begin{equation*}
\operatorname{dom}(\Fpazo) := \Big\{ x \in X ~\,\textnormal{s.t.}~ \Fpazo(x) \neq \emptyset \Big\} \qquad \text{and} \qquad \graph(\Fpazo) := \Big\{(x,y) \in X \times Y ~\,\textnormal{s.t.}~ y\in\Fpazo(x) \Big\}.
\end{equation*}
For a nonempty set \( A \subseteq X \), we write
\begin{equation*}
\Fpazo(A) := \bigcup_{x\in A} \Fpazo(x).
\end{equation*}
Below, we recall a few fundamental definitions pertaining to the regularity of set-valued maps. Therein, and in what follows, we denote by $\B_X(x,r)$ the closed metric ball of radius $r > 0$ centred at $x \in X$. In the absence of potential confusions, we may simply write $\B(x,r)$, and given a set $A\subseteq X$ we simply let $\mathbb B_X(A,r):=\bigcup_{x\in A}\mathbb B_X(x,r)$.

\begin{definition}[Basic continuity notions for set-valued maps]
\label{def:Continuity}
Let  $\Fpazo : X \rightrightarrows Y$ and fix $\bar x \in \dom(\Fpazo)$.  
\begin{itemize}
\item[$(i)$] We say that $\Fpazo$ is \textit{lower semicontinuous} at $\bar{x}$ if for every $\varepsilon>0$ and any $\bar y\in\Fpazo(\bar x)$, there exists $\delta>0$ such that 
\begin{equation*}
\Fpazo(x) \cap \mathbb B_Y(\bar y,\varepsilon)\neq\emptyset
\end{equation*}
for all $x\in\mathbb B_X(\bar x,\delta)$. 
\item[$(ii)$] We say that $\Fpazo$ is \textit{upper semicontinuous} at $\bar x$ if for every $\varepsilon>0$, there exists $\delta>0$ such that 
\begin{equation*}
\Fpazo(x)\subseteq \mathbb B_Y(\Fpazo(\bar x),\varepsilon\big) 
\end{equation*}
for all $x\in \mathbb B_X(\bar x,\delta)$.
\item[$(iii)$] We say that $\Fpazo$ is \textit{continuous} at $\bar x$ if it is both lower and upper semicontinuous at that point.
\end{itemize}
\end{definition}

Next, we recall the concept of Lipschitz continuity for set-valued mappings.

\begin{definition}[Lipschitz continuity for set-valued maps]
We say that $\Fpazo : X \rightrightarrows Y$ is \textit{Lipschitz continuous} with constant $L > 0$ provided that 
\begin{equation*}
\Fpazo(x_1) \subseteq \mathbb B_Y\Big(\Fpazo(x_2), L \, \dsf_X(x_1, x_2)\Big) 
\end{equation*}
for all $x_1, x_2 \in X$. 
\end{definition}


\subsubsection*{Set-valued mappings defined over intervals}
Let $I \subseteq \R$ be a nonempty closed interval and consider two complete separable metric spaces $(X,\dsf_X(\cdot,\cdot))$, $(Y,\dsf_Y(\cdot,\cdot))$. In the following definition, we condense some terminology and basic results on measurable set-valued mappings, for which we refer to \cite[Chapter 8]{Aubin1990}.  

\begin{definition}[Measurability of set-valued maps and measurable selections]
\label{def:Measurability}
Let $\Fpazo: I \rightrightarrows Y$ be a set-valued mapping.
\begin{itemize}
\item[$(i)$] We say that $\Fpazo$ is \textit{$\Lcal^1$-measurable} if for every open set $\Opazo \subseteq Y$, the preimage
\begin{equation*}
\Fpazo^{-1}(\Opazo) := \Big\{t \in I : \Fpazo(t) \cap \Opazo \neq \emptyset \Big\}
\end{equation*}
is an $\Lcal^1$-measurable set.
\item[$(ii)$] We say that $f: I \to Y$ is a \textit{measurable selection} of $\Fpazo : I \rightrightarrows Y$ provided it is $\Lcal^1$-measurable and satisfies $f(t) \in \Fpazo(t)$ for almost every $t \in I$.
\end{itemize}
Moreover, we say that $\Gpazo : I \times X \rightrightarrows Y$ is \textit{Carathéodory} whenever $t \in I \rightrightarrows \Gpazo(t,x)$ is $\Lcal^1$-measurable for every $x\in X$, and $x \in X \rightrightarrows \Gpazo(t,x)$ is continuous for $\mathcal L^1$-almost every $t\in I$. 
\end{definition}
It is a known fact that if the metric space $(X,\dsf_X(\cdot,\cdot))$ is complete separable and $\Gpazo : I \times X \rightrightarrows Y$ is Carathéodory with closed images, then for any $\mathcal L^1$-measurable $x:I\to X$, the set-valued mapping $t \in I \rightrightarrows \Gpazo(t,x(t))$ is $\mathcal L^1$-measurable.

Recall that the Hausdorff distance between two nonempty compact sets $A,B\subseteq Y$ is given by
\begin{align*}
\Delta_H(A,B) := \max\left\{ \sup_{y\in A} \dsf_Y(y,B) \, , \, \sup_{y\in B} \dsf_Y(y,A) \right\}.
\end{align*}
In what follows, we recall the classical transposition of local absolute continuity to set-valued mappings, which plays an essential role in the characterisation of viable sets, see e.g. \cite{ViabPp,Frankowska1995}. In order to describe it, we consider a one-sided (extended) version of the Hausdorff distance, given by 
\begin{equation}
\label{eq:HausdorffSemiDist}
\Delta_{y_0,r}(A,B):=\displaystyle\sup \bigg\{ \dsf_Y(y \,;B) ~\,\textnormal{s.t.}~ y\in A\cap \mathbb B_Y(y_0,r)\bigg\} 
\end{equation}
for all $y_0\in Y$, each $r>0$ and any pair of nonempty closed sets $A,B\subseteq Y$ with $A\cap\mathbb B_Y(y_0,r)\neq\emptyset$. In the case in which $A\cap\mathbb B_Y(y_0,r) = \emptyset$, we simply set $\Delta_{y_0,r}(A,B)=0$. 

\begin{definition}[Absolute continuity of set-valued maps]
Let $\Fpazo: I\rightrightarrows Y$ be a set-valued mapping with nonempty images. 
\begin{itemize}
\item[$(i)$] We say that $\Fpazo$ is \textit{left-absolutely continuous} if for any $y_0\in Y$ and $r>0$ there exists a nonnegative function $m_{y_0,r}(\cdot)\in L^1(I, \mathbb R_+)$ such that 
\begin{equation*}
\Delta_{y_0,r}\big(\Fpazo(s),\Fpazo(t)\big)\le \int_{s}^t m_{y_0,r}(\theta)\,\mathrm{d}\theta
\end{equation*}
for all $s,t\in I$ satisfying $s\le t$.
\item[$(ii)$] We say that $\Fpazo$ is \textit{right-absolutely continuous} if for any $y_0\in Y$ and $r>0$ there exists a nonnegative function $m_{y_0,r}(\cdot)\in L^1(I, \mathbb R_+)$ such that
\begin{equation*}
\Delta_{y_0,r}\big(\Fpazo(t),\Fpazo(s)\big)\le \int_{s}^t m_{y_0,r}(\theta)\mathrm{d}\theta 
\end{equation*}
for all $s,t\in I$ satisfying $s\le t$.
\item[$(iii)$]  We say that $\Fpazo$ is \textit{absolutely continuous} if it is  both left and right absolutely continuous.
\end{itemize}
In the case in which $\Fpazo$ has compact images, we will also say that it is  \textit{Hausdorff absolutely continuous} if  there exists  $m(\cdot)\in L^1(I, \mathbb R_+)$ such that 
\begin{align*}
\Delta_H(\Fpazo(t),\Fpazo(s))\le \int_{s}^t m(\theta)\mathrm{d}\theta 
\end{align*}
for all $s,t\in I$ satisfying $s\le t$.
\end{definition} 	

These notions of absolute continuity yield useful regularity properties for the distance between set-valued mappings. We refer to \cite[Proposition 2.8]{ViabPp} for the proof of the following result. 

\begin{proposition}[Absolute continuity of the metric distance]
\label{Propabs}
Let \( \Kpazo : I \rightrightarrows Y \) be a set-valued mapping with nonempty compact images that is Hausdorff absolutely continuous, and let \( \Qpazo : I \rightrightarrows Y \) be an absolutely continuous set-valued mapping with nonempty closed images. Then, the function $g:I\to[0,+\infty)$ given by
\begin{equation*}
g(t):=\inf\Big\lbrace \dsf_Y(y_1,y_2) ~\, \textnormal{s.t.}~  y_1\in \Kpazo(t) ~~\text{and}~~ y_2 \in\Qpazo(t) \Big\rbrace 
\end{equation*}
for all $t \in I$ is absolutely continuous. 
\end{proposition}

\subsubsection*{Aumann integration for $C^0(\R^d,\R^d)$-valued multifunctions.}
Let $I\subseteq \R_+$ be a closed interval and consider  a set-valued mapping $V: I\rightrightarrows C^0(\mathbb R^d,\mathbb R^d)$. We say that a measurable map $v:I\to C^0(\mathbb R^d,\mathbb R^d)$ is an \textit{integrable selection} of $V : I\rightrightarrows C^0\big(\mathbb R^d,\mathbb R^d\big)$ if it is integrable in the sense of Definition \ref{def:IntegralC0} above, and such that $v(t)\in V(t)$ for $\Lcal^1$-almost every $t\in I$.

\begin{definition}[Aumann integral of $C^0(\R^d,\R^d)$-valued multifunctions]
We define the \textit{Aumann integral} of the set-valued mapping $V : I \rightrightarrows C^0(\R^d,\R^d)$ as
\begin{equation*}
\int_I V(t) \,\mathrm{d}t := \bigg\{\int_I v(t)\,\mathrm{d}t ~\,\textnormal{s.t.}~ \text{$v:I\to C^0\big(\mathbb R^d,\mathbb R^d\big)$ is an integrable selection of $V(\cdot)$}\bigg\}.
\end{equation*}
\end{definition}	
Given $\tau\in I$, we define the set
\begin{equation*}
\begin{aligned}
\Liminf{h\to 0^+} \,\frac{1}{h} \INTSeg{\hspace{-0.05cm}V(t)}{t}{\tau}{\tau+h} :=
\Bigg\{ & w\in C^0(\R^d,\R^d) ~\, \textnormal{s.t.}~ \text{for every sufficiently small $h>0$ with $\tau+h\in I$} \\
&~ \text{there is an integrable selection $t \in [\tau,\tau+h] \mapsto v_h(t)  \in V(t)$ \,s.t.} \\
&~ \dsf_{cc} \Big( \tfrac{1}{h} \mathsmaller{\INTSeg{v_h(s)}{s}{\tau}{\tau+h}},w \Big) \underset{h \to 0^+}{\longrightarrow} 0
\Bigg\}. 
\end{aligned}
\end{equation*}
We shall encounter these generalised Lebesgue points again in Section \ref{section:USC}, and point the reader, e.g., to \cite[Section 3]{Frankowska1995} for ampler details on set-valued lower limits. 


\subsection{Continuity equations and inclusions in Wasserstein spaces}

In this fourth preliminary section, we consider a closed subinterval $I \subseteq [0,T]$ together with a set-valued mapping $V : I \times \Pcal_1(\R^d) \rightrightarrows C^0(\R^d,\R^d)$ whose images are sets of continuous vector fields over $\R^d$, and we recall the definition of solutions to the following class of set-valued Cauchy problems
\begin{equation}\label{cie}
\left\{
\begin{aligned}
& \partial_t \mu(t) \in - \Div_x \Big( V(t,\mu(t)) \, \mu(t) \Big), \\
& \mu(\tau) = \mu_{\tau}, 
\end{aligned}
\right.
\end{equation}
starting from some $(\tau,\mu_{\tau}) \in I \times \Pcal_1(\R^d)$. Additionally, we discuss a few basic topological properties of their reachable and solution sets following \cite{ContIncPp,ViabPp}. 


\subsubsection*{Curves in Wasserstein spaces and Carathéodory fields}

We denote by $C^0(I,\mathcal P_1(\mathbb R^d))$ the space of all continuous curves $\mu:I \to \mathcal P_1(\mathbb R^d)$. This is a complete metric space endowed with the usual supremum metric, which induces the topology of uniform convergence. We further say that $\mu(\cdot) \in C^0(I,\mathcal P_1(\mathbb R^d))$ is \textit{absolutely continuous} if there exists $M(\cdot) \in L^1(I,\mathbb R_+)$ such that
\begin{equation*}
W_1(\mu(s),\mu(t)) \leq \int_{s}^t M(\theta)\,\mathrm{d}\theta, 
\end{equation*}
for all times $s,t \in I$ such that $s \leq t$, and denote by $\AC(I,\mathcal P_1(\mathbb R^d))$ the set of absolutely continuous curves from $I$ into $\mathcal P_1(\mathbb R^d)$. In keeping with Definition \ref{def:Measurability} above, we say that $v : I \times\mathbb R^d \to \mathbb R^d$ is a Carathéodory vector field if $t \in I \mapsto v(t,x)$ is $\Lcal^1$-measurable for all $x \in \R^d$ and $x \in \R^d \mapsto v(t,x) \in \R^d$ is continuous for $\Lcal^1$-almost every $t \in I$. We define its natural functional lift as
\begin{equation*}
t \in I \mapsto v(t) \in C^0(\R^d,\R^d). 
\end{equation*}
It is known (see e.g. \cite[Page 511]{Papageorgiou1986}) that a vector field $v:I \times \mathbb R^d\to \mathbb R^d$ is Carathéodory if and only if its functional lift  $v:I\to C^0(\mathbb R^d,\mathbb R^d)$ is $\Lcal^1$-measurable with respect to the compact-open topology, and we will henceforth refer to both objects interchangeably.  To quantify the discrepancy between elements of \( C^0(\mathbb{R}^d, \mathbb{R}^d) \), we shall sometimes use the following (extended) supremum distance
\[
\dsf_{\operatorname{sup}}(v, w) := \sup_{x \in \mathbb{R}^d} |v(x) - w(x)| \in \mathbb{R}_+ \cup \{+\infty\},
\]
which captures global mismatches between continuous vector fields.


\subsubsection*{Continuity inclusions}
We recall that a curve $\mu \in C^0(I,\Pcal_1(\R^d))$ solves the \textit{continuity equation} driven by a Carathéodory vector field $v:I\times\R^d\to\R^d$ satisfying
\begin{equation*}
\int_I \NormL{v(t)}{1}{\R^d,\R^d;\,\mu(t)}\,\mathrm{d}t < +\infty
\end{equation*}
provided that  
\begin{equation}\label{cee}
\partial_t \mu(t) + \Div_x(v(t)\mu(t))=0, 
\end{equation}
provided that 
\begin{equation*}
\int_{I} \int_{\mathbb{R}^d}\Big(\partial_t\varphi(t,x) + \langle \nabla\varphi(t,x) , v(t,x) \rangle\Big)\mathrm{d}\mu(t)(x)\,\mathrm{d}t = 0
\end{equation*} 
for each test function $\varphi\in C_c^\infty\big(\mathrm{int}(I)\times\mathbb R^d\big)$. {As a consequence, e.g., of \cite[Theorem 8.3.1]{AGS} (see also the method proposed in \cite[Section 5]{Fornasier2019}), any solution of \eqref{cee} complies with the following estimates 
\begin{equation}
\label{eq:BasicEstimates}
\left\{
\begin{aligned}
& W_1(\mu(s),\mu(t)) \leq \INTSeg{\NormL{v(\theta)}{1}{\R^d,\R^d;\,\mu(\theta)}}{\theta}{s}{t}, \\
& \Mpazo_1(\mu(t)) \leq \Mpazo_1(\mu(s)) + \INTSeg{\NormL{v(\theta)}{1}{\R^d,\R^d;\,\mu(\theta)}}{\theta}{s}{t}, 
\end{aligned}
\right.
\end{equation}
for all  $s,t \in I$ such that $s \leq t$}. It is well known that continuity equations are essentially ordinary differential equations over Wasserstein spaces -- which can be loosely seen as infinite dimensional manifolds --, with the role of tangent velocity being played by $t \in I \mapsto v(t) \in C^0(\R^d,\R^d)$. Building upon this insight, the following notion of solution to a differential inclusion in measure spaces was originally proposed in \cite{ContInc,ContIncPp}.

\begin{definition}[Solutions of continuity inclusions]
\label{def:ContInc}
We say that a curve $\mu \in C^0(I,\Pcal_1(\mathbb R^d))$ is a solution of the \textit{continuity inclusion}
\begin{equation*}
\partial_t \mu(t) \in - \Div_x \Big( V(t,\mu(t)) \, \mu(t) \Big)
\end{equation*}
if there exists an $\mathcal L^1$-measurable selection $t \in I \mapsto v(t) \in V(t,\mu(t))$ satisfying
\begin{equation*}
\int_I \NormL{v(t)}{1}{\R^d,\R^d;\,\mu(t)}\,\mathrm{d}t < +\infty
\end{equation*}
and for which $\mu(\cdot)$ solves \eqref{cee}.
\end{definition}


\subsubsection*{Properties of the reachable sets.} In what follows, we let $T > 0$ and consider a set-valued map $V:[0,T]\times \mathcal P_1(\mathbb R^d)\rightrightarrows C^0(\mathbb R^d,\mathbb R^d)$ with nonempty, closed and convex images. The latter is so that $t \in [0,T] \rightrightarrows V(t,\mu)$ is $\Lcal^1$-measurable with respect to the compact-open topology for all $\mu \in \Pcal_1(\R^d)$, and satisfies all or parts of the following assumptions.
\begin{enumerate}
\item[$(i)$] There exists a function $M(\cdot)\in L^1([0,T],\mathbb R_+)$ such that
\begin{equation}
\label{eq:Itemi}
|v(x)| \le M(t) \Big( 1 + |x| + \Mpazo_1(\mu) \Big) \qquad \text{and} \qquad \Lip(v) \leq M(t)
\end{equation}
for $\Lcal^1$-almost every $t \in [0,T]$, any $(\mu,v) \in \Graph(V(t))$ and all $x \in \R^d$.
\item[$(ii)$] There exists a function $L(\cdot)\in L^1([0,T],\mathbb R_+)$ such that for $\Lcal^1$-almost every $t\in[0,T]$, any $\mu,\nu \in \Pcal_1(\R^d)$, and each $v \in V(t,\mu)$, there exists $w \in V(t,\nu)$ with
\[
\dsf_{\operatorname{sup}}(v, w) \leq L(t)\, W_1(\mu, \nu).
\]

\end{enumerate}
Let us now introduce some basic terminology concerning the solutions of the set-valued Cauchy problem \eqref{cie}. Given a pair $(\tau,\mu_{\tau}) \in [0,T] \times \Pcal_1(\R^d)$, define its \textit{(forward) solution set} as
\begin{equation*}
\Spazo_{[\tau, T]}(\mu_\tau) := \Big\{ \mu(\cdot) \in C^0([\tau, T], \mathcal{P}_1(\mathbb{R}^d)) ~\, \textnormal{s.t.}~ \text{$\mu(\cdot)$ is a solution of \eqref{cie}} \Big\},
\end{equation*}
and consider the corresponding \textit{reachable set} at any given time $t\in[\tau,T]$, given by 
\begin{equation*}
\Rpazo_{(\tau, t)}(\mu_\tau) := \Big\{ \mu(t) ~\,\textnormal{s.t.}~ \mu \in \Spazo_{[\tau, t]}(\mu_{\tau}) \Big\}.
\end{equation*}
It can be shown that the latter satisfy the semigroup property
\begin{equation*}
\Rpazo_{(\tau, t)}(\mu_\tau) = \Rpazo_{(s, t)} \circ \Rpazo_{(\tau, s)}(\mu_\tau) 
\end{equation*}
for all times $\tau\le s\le t\le T$. Below, we recall some basic estimates along with some useful properties satisfied by these sets, which were proven in \cite{ContIncPp,ViabPp} under similar regularity assumptions. 

\begin{lemma}[Basic moment and regularity estimates]
\label{lem:Estimates}
Suppose that item $(i)$ above holds. Then
\begin{equation*}
\Mpazo_1(\mu(t)) \leq c_T \qquad \text{and} \qquad  W_1(\mu(s),\mu(t)) \leq (1+2c_T)\INTSeg{M(\theta)}{\theta}{s}{t}
\end{equation*}
for all times $\tau \leq s \leq t \leq T$ and every $\mu(\cdot) \in \Spazo_{[\tau,T]}(\mu_{\tau})$ with $(\tau,\mu_{\tau}) \in [0,T] \times \Pcal_1(\R^d)$, where $c_T > 0$ only depends on the magnitudes of $\Mpazo_1(\mu_\tau)$ and $\Norm{M(\cdot)}_{L^1([0,T],\R_+)}$.
\end{lemma}
\begin{proof}
See \cite[Proposition 2.20]{ViabPp}.
\end{proof}

\begin{proposition}[Topological properties of the solution and reachable sets]
\label{reachabscon}
Under items $(i)$ and $(ii)$ above, the following statements hold.
\begin{itemize}
\item[$(a)$] The solution set \(\Spazo_{[\tau,T]}(\mu_{\tau}) \subseteq C^0([\tau, T], \mathcal{P}_1(\mathbb{R}^d))\) is nonempty and compact for the topology of uniform convergence.
\item[$(b)$] The mapping $t \in [\tau,T] \rightrightarrows \Rpazo_{(\tau,t)}(\mu_{\tau}) \subseteq \Pcal_1(\R^d)$ has compact images and is Hausdorff absolutely continuous.
\end{itemize}

\end{proposition}
\begin{proof}
See \cite[Theorem 4.2 and Theorem 4.5]{ContIncPp} for the first statement and 
\cite[Proposition 2.20, Proposition 2.21 and Lemma 4.5]{ViabPp} for the second one. Although the latter lemma appears in a section devoted to $p>1$, its proof only uses the semigroup property of reachable sets and the preceding absolute-continuity estimate, which remain valid for $p=1$.
\end{proof}

\begin{theorem}[Infinitesimal behaviour of the reachable set]
\label{thm:Infinitesimal}
Suppose that items $(i)$ and $(ii)$ above hold. Then, there exists a subset \( E \subseteq (0, T) \) of full Lebesgue measure such that the following statements hold.
\begin{itemize}
\item[$(a)$] For every \( \tau \in E \), all \( \mu_{\tau} \in \mathcal{P}_1(\mathbb{R}^d) \), each \( v_{\tau}  \in V(\tau, \mu_{\tau}) \) and any \( \varepsilon > 0 \), there exist some \( h_{\varepsilon} > 0 \) along with a curve \( \mu_{\varepsilon}(\cdot) \in \Spazo_{[\tau,T]}(\mu_{\tau}) \) such that
\begin{equation*}
W_1 \Big( \mu_{\varepsilon}(\tau + h), (\Id + h v_{\tau})_\sharp \mu_{\tau} \Big) \leq \varepsilon h\quad 
\end{equation*}
for all $h \in [0, h_{\varepsilon}]$.
\item[$(b)$] For every \( \tau \in E \), all \( \mu_{\tau} \in \mathcal{P}_1(\mathbb{R}^d) \), each curve  $\mu(\cdot) \in \Spazo_{[\tau,T]}(\mu_{\tau})$, any \( \varepsilon > 0 \) and every sequence \( h_i \to 0 \), there exists an element \( v_{\tau}^{\epsilon} \in V(\tau, \mu_{\tau}) \) such that
\begin{equation*}
W_1\Big( \mu(\tau + h_{i_k}), (\Id + h_{i_k} v_{\tau}^{\epsilon})_\sharp \mu_{\tau} \Big) \leq \varepsilon |h_{i_k}|
\end{equation*}
along a subsequence \( h_{i_k} \to 0 \).
\end{itemize}
\end{theorem}

\begin{proof}
See \cite[Theorem 3.1]{ViabPp} for the first statement and \cite[Theorem 3.2]{ViabPp} for the second one. 
\end{proof}

The following lemma, which resembles item $(a)$ of the previous theorem, will play an important role in the arguments presented in Section \ref{section:USC} below. 

\begin{lemma}[An alternative infinitesimal behaviour]
\label{lemcc}
Let $M(\cdot) \in L^1([0,T],\R_+)$ and $v :[0,T] \times \R^d \to \R^d$ be a Carathéodory vector field such that
\begin{equation*}
|v(t,x)| \le M(t)(1 + |x|) \qquad \text{and} \qquad \Lip(v(t)) \leq M(t)
\end{equation*}
for $\Lcal^1$-almost every $t \in [0,T]$ and all $x \in \R^d$. Then, for every  $\tau \in [0,T]$ and each $\mu_\tau\in\mathcal P_1(\mathbb R^d)$, there exists a constant $c>0$ depending only on the magnitudes of $\Mpazo_1(\mu_{\tau})$ and $\Norm{M(\cdot)}_{L^1([0,T],\R_+)}$ such that the unique solution of the Cauchy problem
\begin{equation}
\label{eq:LemContinuity}
\left\{
\begin{aligned}
& \partial_t \mu(t) + \Div_x(v(t) \mu(t)) = 0, \\
& \mu(\tau) = \mu_{\tau}, 
\end{aligned}
\right.
\end{equation}
satisfies
\begin{equation*}
W_1\bigg(\mu(\tau+h) \,, \Big(\Id + \mathsmaller{\int_{\tau}^{\tau+h} v(s)\,\mathrm{d}s} \Big)_{\raisebox{2pt}{\hspace{-0.05cm}$\scriptstyle{\sharp}$}} \, \mu_\tau\bigg) \le   c \bigg(\int_{\tau}^{\tau+h}M(s)\,\mathrm{d}s\bigg)^2
\end{equation*}
for all $h\in[0,T-\tau]$.
\end{lemma}

\begin{proof}
Under our working assumptions, it stems from the standard theory of ODEs (see e.g. \cite[Chapter 16]{AmbrosioBS2021}) that $v : [0,T] \times \R^d \to \R^d$ generates a well-defined characteristic flow $(\Phi_{(\tau,t)}^v)_{t \in [0,T]} \subseteq C^0(\R^d,\R^d)$ solution of the Cauchy problems
\begin{equation*}
\Phi_{(\tau, t)}^{v}(x) = x + \INTSeg{v \Big( s , \Phi_{(\tau, s)}^{v}(x) \Big)}{s}{\tau}{t}
\end{equation*}
for all $(t,x) \in [\tau,T] \times \R^d$, and that the unique solution of \eqref{eq:LemContinuity} is given by 
\begin{equation}
\label{eq:FlowRepresentation}
\mu(t) = \big( \Phi_{(\tau,t)}^{v} \big)_{\sharp} \, \mu_{\tau}.
\end{equation}
Besides, one may easily verify that under our working assumptions, there exists $C_T > 0$ depending only on the magnitude of   $\Norm{M(\cdot)}_{L^1([0,T],\R_+)}$ such that
\begin{equation*}
|\Phi_{(\tau,t)}^{v}(x)| \leq C_T(1+|x|)
\end{equation*}
for all $(t,x) \in [0,T] \times \R^d$. This implies in particular that
\begin{equation*}
\begin{aligned}
\bigg|\Phi_{(\tau, t)}^{v}(x) - x - \int_{\tau}^{t}v(s,x)\,\mathrm{d}s \bigg| &\le \int_{\tau}^{t}\Big| v\Big(s,\Phi_{(\tau, s)}^{v}(x) \Big)-v(s,x) \Big|\,\mathrm{d}s \\
& \le \int_{\tau}^{t}M(s)\big|\Phi_{(\tau, s)}^{v}(x) -x\big|\,\mathrm{d}s \\
& \le (1+C_T) (1+|x|) \bigg(\int_{\tau}^{t}M(s)\,\mathrm{d}s\bigg)\bigg(\int_{\tau}^{t}M(s)\,\mathrm{d}s\bigg)
\end{aligned}
\end{equation*}
for all $(t,x) \in [\tau,T] \times \R^d$, which together with \eqref{eq:FlowRepresentation} further entails 
\begin{equation*}
\begin{aligned}
W_1\bigg(\mu(\tau+h) \,, \Big(\Id + \mathsmaller{\int_{\tau}^{\tau+h} v(s)\,\mathrm{d}s} \Big)_{\raisebox{2pt}{\hspace{-0.05cm}$\scriptstyle{\sharp}$}} \mu_\tau\bigg) & \leq ~ \Big\| \, \Phi_{(\tau,\tau+h)}^{v} - \Big(\Id + \mathsmaller{\int_{\tau}^{\tau+h} v(s)\,\mathrm{d}s} \Big) \, \Big\|_{L^1(\R^d,\R^d;\mu_{\tau})} \\
& \leq  (1+C_T) (1+\Mpazo_1(\mu_{\tau})) \bigg(\int_{\tau}^{\tau+h}M(s)\,\mathrm{d}s\bigg)^2.
\end{aligned}
\end{equation*}
It is then enough to set $c:=(1+C_T) (1+\Mpazo_1(\mu_{\tau}))$ to obtain the desired estimate. 
\end{proof}

We end these preliminaries on measure dynamics by showcasing a simplified version of the general closure principle \cite[Theorem VI-4]{Castaing1977}, tailored to set-valued mappings $V : I \times \Pcal_1(\R^d) \rightrightarrows C^0(\R^d,\R^d)$, where $I \subseteq [0,T]$ is a closed subinterval. 

\begin{proposition}[The Castaing-Valadier closure principle for  $C^0(\R^d,\R^d)$-valued multifunctions]
\label{prop:CastaingValadier}
Let $V : I \times \Pcal_1(\R^d) \rightrightarrows C^0(\R^d,\R^d)$ be a set-valued mapping with compact convex images, such that $t \in I \rightrightarrows V(t,\mu)$ is $\Lcal^1$-measurable for all $\mu \in \Pcal_1(\R^d)$ and $\mu \in \Pcal_1(\R^d) \rightrightarrows V(t,\mu)$ is upper semicontinuous for $\Lcal^1$-almost every $t \in I$. Let further $\{\mu_n(\cdot)\}_{n \in \N}$ and $\{v_n(\cdot)\}_{n \in \N}$ be two sequences of maps valued in $\Pcal_1(\R^d)$ and $C^0(\R^d,\R^d)$ respectively, which satisfy the following. 
\begin{enumerate}
\item[$(i)$] There exists a map $\mu : I \to \Pcal_1(\R^d)$ such that 
\begin{equation*}
W_1\big(\mu_n(t),\mu(t)\big) ~\underset{n \to +\infty}{\longrightarrow}~ 0
\end{equation*}
for $\Lcal^1$-almost every $t \in I$. 
\item[$(ii)$] The maps $\{v_n(\cdot)\}_{n \in \N}$ are $\Lcal^1$-measurable, and there exists $v : I \to C^0(\R^d,\R^d)$ such that 
\begin{equation*}
\INTDom{\xi(t) \langle \Bnu , v(t) - v_n(t) \rangle_{C^0(\R^d,\R^d)}}{I}{t} ~\underset{n \to +\infty}{\longrightarrow}~ 0
\end{equation*}
for every $\Bnu \in \Mcal_c(\R^d,\R^d)$ and each $\xi(\cdot) \in L^{\infty}(I,\R)$.
\item[$(iii)$] It holds that $v_n(t) \in V(t,\mu_n(t))$ for $\Lcal^1$-almost every $t \in I$ and each $n\in \N$. 
\end{enumerate}
Then necessarily $v(t) \in V(t,\mu(t))$ for $\Lcal^1$-almost every $t \in I$. 
\end{proposition}

\begin{proof}
For the convenience of the reader and self-containedness, we detail in Appendix \ref{AppCastaingValadier} below how this result can be derived from its abstract counterpart \cite[Theorem VI-4]{Castaing1977}. 
\end{proof}

\begin{corollary}[On the existence of integrable velocity selections]
\label{cors}
Let $V : I \times \Pcal_1(\R^d) \rightrightarrows C^0(\R^d,\R^d)$ be a set-valued mapping with nonempty compact convex images satisfying \eqref{eq:Itemi}. Suppose also that $t \in I \rightrightarrows V(t,\mu)$ is $\Lcal^1$-measurable for all $\mu \in \Pcal_1(\R^d)$, whereas $\mu \in \Pcal_1(\R^d) \rightrightarrows V(t,\mu)$ is upper semicontinuous for $\Lcal^1$-almost every $t \in I$. Then for every $\mu(\cdot) \in C^0(I,\mathcal P_1(\mathbb R^d))$, there exists an integrable selection $t \in I \mapsto v(t) \in V(t,\mu(t))$. 
\end{corollary}
\begin{proof}
Since $\mu(\cdot) \in C^0(I,\Pcal_1(\R^d))$ is continuous, there exists a sequence $\{\mu_n(\cdot)\}_{n \in \N}$ of simple $\Lcal^1$-measurable functions $\mu_n:I\to\mu(I)$ such that
\begin{equation*}
\sup_{t\in I}W_1(\mu_n(t),\mu(t)) ~\underset{n \to +\infty}{\longrightarrow}~ 0    
\end{equation*}
Then, for each $n\in\mathbb N$, the set-valued map $t\rightrightarrows V(t,\mu_n(t))$ is $\mathcal L^1$-measurable, and by the Kuratowski- Ryll-Nardzewski measurable selection theorem (see e.g. \cite[Theorem 8.1.3]{Aubin1990}), we can find a sequence of $\mathcal L^1$-measurable functions $\{v_n(\cdot)\}_{n \in \N}$ such that  
\begin{align*}
v_n(t)\in V(t,\mu_n(t))
\end{align*}
for $\mathcal L^1$-almost every $t\in I$ and each $n\in\mathbb N$. Since $C:=\sup_{t\in I}\Mpazo_1(\mu(t))<+\infty$ and $\mu_n(t)\in\mu(I)$, the sublinearity estimate in \eqref{eq:Itemi} gives $|v_n(t,x)|\leq M(t)(1+|x|+C)\leq M(t)(1+C)(1+|x|)$ for $\Lcal^1$-almost every $t\in I$ and all $x\in\R^d$, uniformly in $n \in \N$. Similarly, the Lipschitz estimate in \eqref{eq:Itemi} entails local uniform equicontinuity of the sequence. It follows in particular from that each mapping $v_n : I \to C^0(\R^d,\R^d)$ is integrable in the sense of Definition \ref{def:IntegralC0}, and by Lemma \ref{lem:comvf}, we may find an extracted subsequence $\{v_{n_k}(\cdot)\}_{k\in\mathbb N}$ along with an integrable map $ v :I \to C^0(\mathbb R^d,\mathbb R^d)$  such that 
\begin{equation*}
\int_I  \langle \Bnu(t) , v(t) - v_{n_k}(t)\rangle_{C^0(\mathbb R^d,\mathbb R^d)}\,\mathrm{d}t \underset{k \to +\infty}{\longrightarrow} 0
\end{equation*}
for every scalarly-$^*$ measurable $\Bnu: I\to \mathcal M_c(\mathbb R^d,\mathbb R^d)$ with $\esssup_{t \in I} \int_{\mathbb R^d}(1+|x|)\, \mathrm{d} |\Bnu(t)|(x) < +\infty$.  We then see that the assumptions of Proposition \ref{prop:CastaingValadier} are satisfied, from whence we deduce that $v(t) \in V(t,\mu(t))$ for $\Lcal^1$-almost every $t \in I$, which concludes the proof.
\end{proof}

\section{Viability in the Lipschitz case}
\label{section:Lip}

In this section we prove our first main result, which extends the viability theorems of \cite{ViabPp} obtained  under Lipschitz regularity assumptions for dynamics in $p$-Wasserstein spaces with $p \in (1,+\infty)$ to the 1-Wasserstein space. In this context, we consider a set-valued mapping $V:[0,T]\times \mathcal P_1(\mathbb R^d)\rightrightarrows C^0(\mathbb R^d,\mathbb R^d)$ complying with the following hypotheses. 

\begin{taggedhyp}{\textnormal{(CL)}} 
\label{hyp:CICL} \hfill
\begin{itemize}
\item[$(i)$] The set-valued map $V : [0,T] \times \Pcal_1(\R^d) \rightrightarrows C^0(\R^d,\R^d)$ has nonempty, closed and convex images, and $t \in [0,T] \rightrightarrows V(t,\mu)$ is $\Lcal^1$-measurable for every $\mu \in \Pcal_1(\R^d)$. 
\item[$(ii)$] There exists a function $M(\cdot)\in L^1([0,T],\mathbb R_+)$ such that 
\begin{equation*}
|v(x)| \le M(t) \Big( 1 + |x| + \Mpazo_1(\mu)  \Big) \qquad \text{and} \qquad \Lip(v ) \leq M(t)
\end{equation*}
for $\Lcal^1$-almost every $t \in [0,T]$, any $(\mu,v) \in \Graph(V(t))$ and all $x \in \R^d$.
\item[$(iii)$]There exists a function $L(\cdot)\in L^1([0,T],\mathbb R_+)$ such that for $\Lcal^1$-almost every $t\in[0,T]$, any $\mu,\nu \in \Pcal_1(\R^d)$, and each $v \in V(t,\mu)$, there exists $w \in V(t,\nu)$ for which
\[
\dsf_{\operatorname{sup}}(v, w) \leq L(t)\, W_1(\mu, \nu).
\]
\end{itemize}
\end{taggedhyp}

Let us point out that, although the images of $V : [0,T] \times \Pcal_1(\R^d) \rightrightarrows C^0(\R^d,\R^d)$ are assumed to be merely closed in Hypothesis \ref{hyp:CICL}-$(i)$, they are actually compact for $\Lcal^1$-almost every $t \in [0,T]$, and $V$ may thus be replaced by a compact-valued representative. Indeed, for $\Lcal^1$-almost every $t \in [0,T]$, every $\mu \in \Pcal_1(\R^d)$ and each $R > 0$, Hypothesis \ref{hyp:CICL}-$(ii)$ yields
\begin{equation*}
\sup_{v \in V(t,\mu)} \NormC{v}{0}{B(0,R),\R^d} \leq M(t) \Big( 1 + R + \Mpazo_1(\mu) \Big) \qquad \text{and} \qquad \sup_{v \in V(t,\mu)} \Lip(v) \leq M(t),
\end{equation*}
so that $V(t,\mu) \subset C^0(\R^d,\R^d)$ is relatively compact for the compact-open topology, as a consequence of the Ascoli-Arzel\`a theorem applied over the balls $B(0,N)$ for every $N \in \N$, combined with a diagonal argument (see also \cite[Theorem 11.28]{Rudin1987}). Being closed, this set is therefore compact, and since the exceptional set of times is common to all $\mu \in \Pcal_1(\R^d)$, the images of $V : [0,T] \times \Pcal_1(\R^d) \rightrightarrows C^0(\R^d,\R^d)$ may be redefined thereon without modifying the solution sets of the continuity inclusion or the corresponding Aumann integrals. In what follows, we shall thus apply results which are stated for compact-valued multifunctions -- such as Proposition \ref{prop:CastaingValadier} and Corollary \ref{cors} -- to this compact-valued representative.

Given a pair $(\tau,\mu_{\tau}) \in[0,T] \times \Pcal_1(\R^d)$, we shall consider the set-valued Cauchy problem
\begin{equation}\label{CILip}
\left\{
\begin{aligned}
& \partial_t \mu(t) \in -\Div_x \Big( V\big(t,\mu(t)\big)\, \mu(t) \Big)\\
& \mu(\tau) = \mu_{\tau}, 
\end{aligned}
\right.
\end{equation}
and say that a set-valued map $\Qpazo:[0,T]\rightrightarrows \mathcal P_1(\mathbb R^d)$ is \textit{viable} for \eqref{CILip} if for any  $\tau\in[0,T]$ and $\mu_\tau\in \Qpazo(\tau)$, there exists a solution $\mu(\cdot) \in \Spazo_{[\tau,T]}(\mu_{\tau})$ of the latter dynamics such that 
\begin{equation*}
\mu(t)\in \Qpazo(t)
\end{equation*}
for all times $t \in [\tau,T]$. We recall that the relevant geometric object allowing to characterise said viability is the so-called \textit{graphical derivative} at an element $(\tau,\nu)\in \graph(\Qpazo)$, defined by
\begin{equation}
\label{eq:GraphicalDef}
D \Qpazo(\tau|\nu):= \left\lbrace \xi\in L^1(\mathbb R^d, \mathbb R^d; \nu) ~\,\textnormal{s.t.}~ \, \liminf_{h\to 0^+}\frac{1}{h} W_1\Big((\Id + h\xi)_{\sharp}\nu \,; \Qpazo(\tau+h)\Big) = 0 \right\rbrace,
\end{equation}	
where here and in all that follows, we adopt the notation $W_1(\mu \, ;\Qpazo) := \inf_{\nu \in \Qpazo} W_1(\mu,\nu)$ to lighten the exposition. We also recall that a nonempty set $\Qpazo \subseteq \Pcal_1(\R^d)$ is \textit{proximal} if for every $\mu \in \Pcal_1(\R^d)$, there exists $\nu \in \Qpazo$ such that $W_1(\mu,\nu) = W_1(\mu \, ; \Qpazo)$, namely if the metric projection onto $\Qpazo$ is attained for every $\mu \in \Pcal_1(\R^d)$. It is well known that proximal subsets of metric spaces are closed.
%
We are ready to state the main result of this section.

{
\begin{theorem}[Viability in the Lipschitz framework]
\label{Thm1}
Suppose that Hypotheses \ref{hyp:CICL} hold and let $\Qpazo:[0,T]\rightrightarrows\mathcal P_1(\mathbb R^d)$ be absolutely continuous with nonempty proximal images. Then, the following statements are equivalent. 
\begin{itemize}
\item[$(a)$] The set-valued map $\Qpazo : [0,T] \rightrightarrows \Pcal_1(\R^d)$ is viable for \eqref{CILip}.
\item[$(b)$] It holds that 
\begin{equation}
\label{eq:PointwiseViabilityCond}
V(\tau,\nu)\cap D\Qpazo(\tau|\nu)\neq \emptyset
\end{equation} 
for $\Lcal^1$-almost every $\tau \in [0,T]$ and all $\nu\in \Qpazo(\tau)$.
\end{itemize}
\end{theorem}}

We split the proof of this theorem in two parts, starting with the easier direct implication. 


\subsection*{Proof of Theorem \ref{Thm1} $(a) \implies (b)$}
Let \( E \subseteq (0, T) \) be the set of full measure over which the statement of Theorem \ref{thm:Infinitesimal}-$(b)$ as well as Hypotheses \ref{hyp:CICL}-$(ii)$,$(iii)$ hold. Furthermore, let \( \tau \in E \) and \( \nu \in \Qpazo(\tau) \) be arbitrary, and fix a vanishing sequence $\varepsilon_n \to 0^+$ of positive numbers. Since we assumed that $\Qpazo : [0,T] \rightrightarrows \Pcal_1(\R^d)$ is viable for \eqref{CILip}, there exists a curve $\mu(\cdot) \in \Spazo_{[\tau,T]}(\nu)$ such that 
\begin{equation*}
\mu(t)\in\Qpazo(t)
\end{equation*}
for all times $t \in [\tau,T]$. By Theorem \ref{thm:Infinitesimal}-$(b)$, there exists for any $n\in\mathbb N$ a sequence $\{h^n_i\}_{i\in\mathbb N}$ of positive numbers  converging  to zero along with admissible velocities \( \{v_n\}_{n\in \N} \subseteq V(\tau, \nu)\) such that
\begin{equation*}
W_1 \Big((\Id + h^{n}_{i} v^{}_{n})_{\sharp} \nu \,; \Qpazo(\tau+h_i^n) \Big)\le W_1 \Big((\Id + h^{n}_{i} v_{n})_{\sharp} \nu ,  \mu(\tau + h^{n}_{i}) \Big) \leq \varepsilon_n h^{n}_{i}
\end{equation*}
for each $i \in \N$. Since $\tau\in E$ is such that Hypotheses \ref{hyp:CICL}-$(ii)$ hold, the sequence $\{v_n\} \subseteq C^0(\R^d,\R^d)$ is pointwise uniformly equi-bounded as well as locally  uniformly equi-Lipschitz, and hence it can be deduced  from the Ascoli-Arzel\`a theorem (see e.g. \cite[Theorem 11.28]{Rudin1987})  that there exists some \( v_{\tau} \in V(\tau, \nu) \) such that
\begin{equation*}
\dsf_{cc}(v_{n},v_{\tau}) ~\underset{n \to +\infty}{\longrightarrow}~ 0
\end{equation*}
up to a subsequence that we do not relabel. We next show that 
\begin{equation}\label{snawtd}
\left\| v_{n} - v_{\tau} \right\|_{L^1(\mathbb{R}^d, \mathbb{R}^d ;\, \nu)} \underset{n \to +\infty}{\longrightarrow} 0.
\end{equation}
Using arguments from \cite[Lemma 2.15]{ViabPp}, splitting the integral over a ball and its complement and using Hypothesis \ref{hyp:CICL}-$(ii)$, one has that
\begin{equation*}
\NormL{v_n - v_{\tau}}{1}{\R^d,\R^d ; \, \nu} \leq \NormC{v_n - v_{\tau}}{0}{B(0,R),\R^d} + 2M(\tau) \INTDom{\Big( 1 + |x| + \Mpazo_1(\nu) \Big)}{\{ x \;\textnormal{s.t.}\, |x| \geq R\}}{\nu(x)}
\end{equation*}
for every $R > 0$. For a fixed radius, the first term vanishes as $n \to +\infty$ because $v_n \to v_{\tau}$ uniformly over compact sets, whereas the second one converges to $0$ as $R \to +\infty$ due to the characterization \eqref{eq:RelativeCompactness} of relative compactness in $\Pcal_1(\R^d)$. These facts put together then yield \eqref{snawtd}. For each $n\in\mathbb N$, let now $m_n\in\mathbb N$ be such that $\delta_n:=h^n_{m_n}\le 1/n$. Observe then that $\delta_n \to 0$ as $n \to +\infty$, and 
\begin{equation*}
\begin{aligned}
W_1\Big((\Id+\delta_n v_{\tau})_{\sharp}\nu \, ; \Qpazo(\tau+\delta_n)\Big)& \le W_1\Big((\Id+\delta_n v_{\tau})_{\sharp}\nu,(\Id+\delta_n v_{n})_{\sharp}\nu \Big) \\
& \hspace{0.45cm} + W_1\Big((\Id+\delta_n v_{n})_{\sharp}\nu \, ; \Qpazo(\tau+\delta_n) \Big) \\
&\le \delta_n \Big(	\left\| v_{n} - v_{\tau} \right\|_{L^1(\mathbb{R}^d, \mathbb{R}^d; \, \nu)} + \varepsilon_{n} \Big).
\end{aligned}
\end{equation*}
In particular, one gets that
\begin{equation*}
\liminf_{\delta\to 0^+} \frac{1}{\delta}	W_1\Big((\Id+\delta v_\tau)_{\sharp}\nu\, ; \Qpazo(\tau+\delta)\Big) \leq \lim_{n \to +\infty} \frac{1}{\delta_n}	W_1\Big((\Id+\delta_n v_{\tau})_{\sharp}\nu \, ; \Qpazo(\tau+\delta_n)\Big) = 0.
\end{equation*}
We thus conclude that $v_{\tau} \in D\Qpazo(\tau|\nu)$, and the thesis follows since $v_{\tau} \in V(\tau,\nu)$ by construction. \hfill $\square$


\subsection*{Proof of Theorem \ref{Thm1} $(b) \implies (a)$}
Let $\tau\in[0,T]$ and $\mu_\tau\in\Qpazo(\tau)$ be given, and recall that $\Rpazo_{(\tau, t)}(\mu_\tau)$ denotes the reachable set of \eqref{CILip} at time $t\in[\tau,T]$. Then, define $g:[\tau,T]\to \mathbb R_+$ by 
\begin{equation*}
g(t):=\inf\Big\lbrace W_1(\mu,\nu) ~\,\textnormal{s.t.}~ \mu\in\Rpazo_{(\tau, t)}(\mu_\tau) \text{ and } \nu \in\Qpazo(t) \Big\rbrace.
\end{equation*}
It follows from Propositions \ref{Propabs} and \ref{reachabscon} that $g(\cdot)$ is absolutely continuous, and therefore differentiable $\Lcal^1$-almost everywhere in $[\tau,T]$. We will divide the proof into three steps. In Step 1, we build a specific full measure set $E\subseteq[\tau,T]$ and recall basic estimates on the reachable sets, which are then leveraged in Step 2 to show that whenever $t\in E$ is such that $g(t)>0$, then 
\begin{align}\label{pregronwall}
g'(t)\le \Big(1 + M(t) +  L(t)\Big)g(t).
\end{align}
Clearly, if $t\in (\tau,T)$ is such that $g'(t)$ exists and $g(t)=0$, then $t$ is a minimizer of $g(\cdot)$, which necessarily implies that $g'(t)=0$ and the previous inequality automatically holds. Consequently \eqref{pregronwall} holds $\Lcal^1$-almost everywhere in $[\tau,T]$, and a simple application of Gr\"onwall's lemma entails that $g(\cdot)$ must be identically zero since $g(\tau) = 0$. In Step 3, we prove via a compactness argument that this in turn implies the existence of a viable solution.

\medskip

\noindent \textbf{Step 1.} (Construction of a full measure set and preliminary estimates). Let $E_1\subseteq(\tau,T)$ be the set of differentiability points of $g:[\tau,T]\to\mathbb R$, which has full Lebesgue measure in $[\tau,T]$. Now, let  $E_2\subseteq[\tau,T]$ be the set of points $t\in[\tau,T]$ satisfying $D\Qpazo(t|\nu)\cap V(t,\nu)\neq\emptyset$ for all $\nu\in\Qpazo(t)$, which by assumption also has full Lebesgue measure in $[\tau,T]$. Let $E_3\subseteq[\tau,T]$ be the set of points $t\in[\tau,T]$ such that for every $\mu_t\in\mathcal P_1(\mathbb R^d)$, each $v_t \in V(t,\mu_t)$ and any $\varepsilon>0$, there exists a solution $\mu(\cdot) \in \Spazo_{[t,T]}(\mu_t)$ of \eqref{CILip} satisfying 
\begin{equation*}
W_1\Big(\mu(t+h), (\Id+h v_t)_{\sharp}\mu_t\Big) \le \varepsilon h
\end{equation*}
for all $h>0$ sufficiently small. By Theorem \ref{thm:Infinitesimal}-$(a)$, the set $E_3$ has full Lebesgue measure in $[\tau,T]$. Let finally $E_4 \subseteq [0,T]$ be the set of all $t\in[\tau,T]$ such that hypotheses \ref{hyp:CICL}-$(ii)$,$(iii)$ hold. 
Define $E:=E_1\cap E_2\cap E_3\cap E_4$ and observe that, by construction, $E$ has full Lebesgue measure in $[\tau,T]$. 
\medskip

\noindent \textbf{Step 2.} (A pre-Gr\"onwall inequality).
Let $t\in E$ be such that $g(t)>0$, and choose $\mu_t \in \Rpazo_{(\tau, t)}(\mu_\tau)$ along with $\nu_t \in \Qpazo(t)$ in such a way that
\begin{equation}
\label{eq:LipViabDist}
g(t) = W_1(\mu_t,\nu_t).
\end{equation}
This is always possible since the map $\mu \in \Pcal_1(\R^d) \mapsto W_1(\mu \, ; \Qpazo(t))$ is $1$-Lipschitz whereas $\Rpazo_{(\tau,t)}(\mu_{\tau})$ is compact, so that there exists $\mu_t \in \Rpazo_{(\tau,t)}(\mu_{\tau})$ for which $g(t) = W_1(\mu_t \, ; \Qpazo(t))$, and the proximality of $\Qpazo(t)$ then provides $\nu_t \in \Qpazo(t)$ satisfying $W_1(\mu_t,\nu_t) = W_1(\mu_t \, ; \Qpazo(t))$. Besides, under the geometric assumption \eqref{eq:PointwiseViabilityCond} posited in item $(b)$, one may choose $w_t \in V(t,\nu_t)\cap D\Qpazo(t|\nu_t)$. Then, owing to Hypotheses \ref{hyp:CICL}-$(iii)$, there exists  $v_t\in V(t,\mu_t)$ such that
\begin{equation}
\label{eq:LipViabSel}
\dsf_{\operatorname{sup}}(v_t,w_t) \le L(t) W_1(\mu_t,\nu_t).
\end{equation}
Moreover, because $w_t\in D\Qpazo(t|\nu_t)$, there exists a sequence $h_i \to 0^+$ of positive numbers such that 
\begin{equation*}
W_1\Big((\Id+h_{i}w_t)_{\sharp} \nu_t \, ; \Qpazo(t + h_{i})\Big) = o(h_{i}). 
\end{equation*}
Fix $\epsilon \in (0,g(t))$, and recall that by definition of $E \subseteq [\tau,T]$, there exists $\mu(\cdot) \in \Spazo_{[t,T]}(\mu_t)$ such that 
\begin{equation*}
W_1\Big(\mu(t + h_i), (\Id + h_i v_t)_{\sharp}\mu_t \Big)\le \varepsilon h_i\le h_i g(t) 
\end{equation*} 
for all sufficiently large $i \in \N$. Combining these two facts along with the basic estimates \eqref{eq:WassEst}, we get that
\begin{equation}
\label{eq:LipViabEst}
\begin{aligned}
g(t+h_i)&\le  W_1\Big(\mu(t+h_i),(\Id + h_i v_t)_{\sharp}\mu_t \Big) + W_1\Big((\Id + h_i v_t)_{\sharp}\mu_t \, ; \Qpazo(t+h_i)\Big)\\
&\le h_i g(t) + W_1\Big((\Id + h_i v_t)_{\sharp}\mu_t, (\Id + h_i w_t)_{\sharp} \nu_t \Big) + W_1\Big((\Id + h_i w_t)_{\sharp} \nu_t \, ; \Qpazo(t+h_i) \Big)\\
&\le h_i g(t) + W_1\Big((\Id + h_iv_t)_{\sharp}\mu_t, (\Id + h_iw_t)_{\sharp}\mu_t \big) \\
& \hspace{1.38cm}+ W_1\Big((\Id + h_i w_t)_{\sharp}\mu_t, (\Id + h_i w_t)_{\sharp} \nu_t \Big) + o(h_i) \\
&\le  h_i g(t) + h_i \NormL{v_t - w_t}{1}{\R^d,\R^d; \, \mu_t} + ~ \Lip(\Id + h_i w_t) W_1(\mu_t,\nu_t) + o(h_i) \\
&\le h_i g(t) + h_i \NormL{v_t - w_t}{1}{\R^d,\R^d; \, \mu_t} + ~ \big(1+ h_i M(t) \big) g(t) + o(h_i), 
\end{aligned}
\end{equation}
where we also leveraged Hypothesis \ref{hyp:CICL}-$(ii)$ and \eqref{eq:LipViabDist}. At this stage, observe in addition that 
\begin{equation*}
\NormL{v_t - w_t}{1}{\R^d,\R^d; \, \mu_t}  ~= \INTDom{|v_t(x) - w_t(x)|}{\R^d}{\mu_t(x)} \leq  \dsf_{\operatorname{sup}}(v_t,w_t).
\end{equation*}
The latter inequality combined with \eqref{eq:LipViabSel}, \eqref{eq:LipViabEst}  yield altogether that  
\begin{equation*}
g(t+h_i) \leq g(t) + h_i \Big( 1 + M(t) +  L(t) \Big) g(t) + o(h_i).
\end{equation*}
It then remains to divide by $h_i > 0$ and let $i \to +\infty$ to recover differential inequality (\ref{pregronwall}).

\medskip

\noindent \textbf{Step 3.} (Existence of a viable solution). Following Step 2 and the discussion at the beginning of the proof, we have established that 
\begin{equation*}
g'(t)\le  \Big(1 + M(t) +  L(t) \Big)g(t)
\end{equation*}
for $\Lcal^1$-almost every $t \in [\tau, T]$. By Gr\"onwall's lemma, this implies in particular that $g(t) = 0$ for all times $t \in [\tau,T]$. Moreover, since the pair $(\tau,\mu_{\tau}) \in \Graph(\Qpazo)$ that we initially fixed was arbitrary, what we have actually proven is that
\begin{equation}
\label{eq:LipNullDist}
\inf\bigg\{ W_1(\mu_{t_2},\nu_{t_2}) ~\,\textnormal{s.t.}~ \mu_{t_2} \in \Rpazo_{(t_1,t_2)}(\mu_{t_1}) ~\text{and}~ \nu_{t_2} \in \Qpazo(t_2) \bigg\} = 0
\end{equation}
for all times $0 \leq t_1 \leq t_2 \leq T$ and every $(t_1,\mu_{t_1}) \in \Graph(\Qpazo)$.  For each $n\in\mathbb N$, consider the partition $\{t_k\}_{k=0}^n$ of the interval $[\tau,T]$ given by 
\begin{equation*}
t_k:=\tau +\displaystyle\frac{T-\tau}{n}k.
\end{equation*}
By \eqref{eq:LipNullDist}, the compactness of the reachable set and the closedness of $\Qpazo(t_{k+1})$, which follows from its proximality, imply that these sets intersect. Using this observation inductively together with the semigroup property of the reachable sets, one can construct for each $n\in\mathbb N$ a curve $\mu_n(\cdot) \in \Spazo_{[\tau,T]}(\mu_{\tau})$ such that 
\begin{align}\label{muuniformmesh}
\mu_n(t_k)\in \Qpazo(t_k)
\end{align} 
for every $k \in \{0,\dots,n\}$. Then, thanks to the compactness result of Proposition \ref{reachabscon}-$(a)$, one may extract a subsequence $\{\mu_{n_i}(\cdot)\}_{i\in\mathbb N}$ that converges uniformly to a limit curve $\mu(\cdot)\in\Spazo_{[\tau,T]}(\mu_{\tau})$. To see that $\mu(t) \in \Qpazo(t)$ for all times $t \in [\tau,T]$, fix such a time and for each $i \in \N$, denote by $s_i \in [\tau,t]$ the largest point of the partition associated with $n_i$ which is smaller than $t$, so that $t - s_i \leq (T-\tau)/n_i$. Then, one has that
\begin{equation*}
W_1(\mu(t) \, ; \Qpazo(t)) \leq W_1(\mu(t),\mu_{n_i}(t)) + W_1(\mu_{n_i}(t),\mu_{n_i}(s_i)) + W_1(\mu_{n_i}(s_i) \, ; \Qpazo(t)),
\end{equation*}
where the first term becomes small as $i \to +\infty$ by uniform convergence, and the second one by the absolute continuity estimate of Lemma \ref{lem:Estimates}. Concerning the third term, by \eqref{muuniformmesh} we know that $\mu_{n_i}(s_i) \in \Qpazo(s_i)$, and by Lemma \ref{lem:Estimates}, all the curves $\{\mu_n(\cdot)\}_{n \in \N}$ take their values in the ball $\B(\mu_{\tau},r)$ with $r := 1+(1+2c_T) \INTSeg{M(\theta)}{\theta}{\tau}{T}$. Hence, the left absolute continuity of $\Qpazo : [0,T] \rightrightarrows \Pcal_1(\R^d)$ combined with the definition \eqref{eq:HausdorffSemiDist} of the one-sided Hausdorff semidistance yields
\begin{equation*}
W_1(\mu_{n_i}(s_i) \, ; \Qpazo(t)) \leq \Delta_{\mu_{\tau},r} \big( \Qpazo(s_i),\Qpazo(t) \big) \leq \INTSeg{m_{\mu_{\tau},r}(\theta)}{\theta}{s_i}{t} ~\underset{i \to +\infty}{\longrightarrow}~ 0.
\end{equation*}
Therefore $W_1(\mu(t) \, ; \Qpazo(t)) = 0$, and since $\Qpazo(t)$ is proximal -- in particular it is closed --, we obtain that $\mu(t) \in \Qpazo(t)$. We may then conclude that $\Qpazo : [0,T] \rightrightarrows \Pcal_1(\R^d)$ is viable for \eqref{CILip} since the pair $(\tau,\mu_{\tau}) \in \Graph(\Qpazo)$ was arbitrary. \hfill $\square$


\section{Viability in the upper semicontinuous case} 
\label{section:USC}

This section is devoted to a  viability result for continuity inclusions whose right-hand side is not Lipschitz continuous in the measure variable, but merely upper semicontinuous. In what follows, we posit that $V:[0,T]\times \mathcal P_1(\mathbb R^d)\rightrightarrows C^0(\mathbb R^d,\mathbb R^d)$ complies with the following assumptions. 

\begin{taggedhyp}{\textnormal{(USC)}} 
\label{hyp:USC} \hfill
\begin{itemize}
\item[$(i)$] The set-valued map $V : [0,T] \times \Pcal_1(\R^d) \rightrightarrows C^0(\R^d,\R^d)$ has nonempty, closed and convex images, and is such that the map $t \in [0,T] \rightrightarrows V(t,\mu) \subseteq C^0(\R^d,\R^d)$ is $\Lcal^1$-measurable for each $\mu \in \Pcal_1(\R^d)$. 
\item[$(ii)$] There exists  $M(\cdot)\in L^1([0,T],\mathbb R_+)$ such that 
\begin{equation*}
|v(x)| \le M(t) \Big( 1 + |x| + \Mpazo_1(\mu) \Big) \qquad \text{and} \qquad \Lip(v ) \leq M(t)
\end{equation*}
for $\Lcal^1$-almost every $t \in [0,T]$, any $(\mu,v) \in \Graph(V(t))$ and all $x \in \R^d$.
\item[$(iii)$]  The set-valued mapping $\mu \in \mathcal P_1(\mathbb R^d) \rightrightarrows V(t,\mu) \subseteq C^0\big(\mathbb R^d,\mathbb R^d\big)$ is upper semicontinuous for $\Lcal^1$-almost every $t \in [0,T]$. 
\end{itemize}
\end{taggedhyp}

As already observed in the previous section, Hypotheses \ref{hyp:USC}-$(i)$,$(ii)$ imply that the images of $V : [0,T] \times \Pcal_1(\R^d) \rightrightarrows C^0(\R^d,\R^d)$ are compact for $\Lcal^1$-almost every $t \in [0,T]$. Since the corresponding exceptional set of times is independent of $\mu \in \Pcal_1(\R^d)$, we may redefine $V$ thereon and henceforth work with a compact-valued representative.

\begin{remark}[Viability for upper semicontinuous dynamics in general Wasserstein spaces]
Although in this manuscript we chose to solely work with dynamics over the $1$-Wasserstein space, the results and methods developed in the present section can be transposed almost verbatim to any $p$-Wasserstein space with $p \in [1,+\infty)$.   
\end{remark}

For a given pair $(\tau,\mu_{\tau}) \in [0,T] \times \mathcal P_1(\mathbb R^d)$, we study as before the set-valued Cauchy problem
\begin{equation}\label{CIUC}
\left\{
\begin{aligned}
& \partial_t \mu(t) \in - \Div_x \Big( V\big(t,\mu(t)\big)\, \mu(t) \Big),\\
& \mu(\tau) = \mu_{\tau}, 
\end{aligned}
\right.
\end{equation}
and recall that a set-valued map $\Qpazo:[0,T]\rightrightarrows \mathcal P_1(\mathbb R^d)$ is viable for \eqref{CIUC} if for any  $(\tau,\mu_{\tau}) \in \Graph(\Qpazo)$, there exists a solution $\mu(\cdot) \in \Spazo_{[\tau,T]}(\mu_{\tau})$ such that 
\begin{equation*}
\mu(t)\in \Qpazo(t)
\end{equation*}
for all times $t \in [\tau,T]$. 

\begin{theorem}[Viability in the upper semicontinuous framework]
\label{Thm4}
Suppose that Hypotheses \ref{hyp:USC} hold, and let $\Qpazo:[0,T]\rightrightarrows\mathcal P_1(\mathbb R^d)$ be left absolutely continuous with nonempty proximal images. Suppose in addition that 
\begin{align}
\label{eq:IntegralViabilityCond}
D\Qpazo\big(\tau|\nu\big)\cap \Liminf{h\to 0^+} \,\frac{1}{h}\int_{\tau}^{\tau+h} V\big(s,\B(\nu,r)\big) \,\mathrm{d}s \neq \emptyset 
\end{align}
for $\Lcal^1$-almost every $\tau \in [0,T]$, every $\nu \in \Qpazo(\tau)$ and each $r > 0$. 
Then, the set-valued mapping $\Qpazo : [0,T] \rightrightarrows \Pcal_1(\R^d)$ is viable for \eqref{CIUC}. 
\end{theorem}

\begin{remark}[Concerning the integral viability condition \eqref{eq:IntegralViabilityCond}]\label{rmlic}
Since Hypotheses \ref{hyp:CICL} are strictly stronger and imply Hypotheses \ref{hyp:USC}, the two viability conditions can be compared whenever the remaining assumptions of both theorems are satisfied. In that case, condition \eqref{eq:IntegralViabilityCond} implies viability by Theorem \ref{Thm4}, and hence \eqref{eq:PointwiseViabilityCond} by the necessity part of Theorem \ref{Thm1}. Thus, under all the assumptions of Theorem \ref{Thm1}, condition \eqref{eq:IntegralViabilityCond} implies \eqref{eq:PointwiseViabilityCond}. Note that we do not establish the converse implication here. In the context of Euclidean spaces, it was  shown in \cite{Frankowska1995} that \eqref{eq:IntegralViabilityCond} is actually equivalent to \eqref{eq:PointwiseViabilityCond} when the dynamics is Carathéodory, but the underlying arguments do not seem to carry easily to Wasserstein spaces, see also the discussion in Remark \ref{rmk:Hyp} below.     
\end{remark}

The proof of Theorem \ref{Thm4} is postponed to Section \ref{ssusc} and split into two main parts. In the first one, we construct a suitable sequence of approximate solutions of the dynamics whose distance to the constraints is properly controlled, while in the second one, we identify a converging subsequence whose limit satisfies both the desired continuity inclusion and state constraints. The construction of these approximate solutions is subtle and quite intricate, relying on Zorn's lemma, and is detailed across several points in Section \ref{fas}. 


\subsection{A family of approximate solutions}\label{fas}
 
Throughout this subsection, the assumptions of Theorem \ref{Thm4} are in force. Since the initial time may always be translated, it is enough to work with an initial datum $\mu_0 \in \Qpazo(0)$, which we fix once and for all. Our aim here is to construct, for every sufficiently small approximation parameter, a globally defined curve which follows an enlarged dynamics outside a set of times of small measure, and which remains quantitatively close to the images of $\Qpazo : [0,T] \rightrightarrows \Pcal_1(\R^d)$. The family of curves obtained in this way forms the approximating sequence used in the compactness argument of Section \ref{ssusc} below. To carry out this construction, we begin by defining the constant
\begin{equation}
\label{eq:cTDef}
c_T := 2\bigg( 1 + \Mpazo_1(\mu_0) + \INTSeg{M(s)}{s}{0}{T} \bigg) \exp \bigg( 2 \INTSeg{M(s)}{s}{0}{T} \bigg), 
\end{equation}
and let
\begin{align}\label{r}
r_T:= {4(1 + c_T)} \bigg( 1 + \int_{0}^{T} M(s)\,\mathrm{d}s \bigg) + T \exp\bigg(\int_{0}^T M(s)\,\mathrm{d}s\bigg).
\end{align}
Set also $B_T:=\int_0^T M(s)\,\mathrm{d}s$, $A_T:=3+c_T$ and $c:=(2+c_T)A_T^2\big(1+(1+A_TB_T)e^{A_TB_T}\big)$.
We also denote by $M_{\Qpazo}(\cdot)\in L^1([0,T],\mathbb R_+)$ any function such that $M_{\Qpazo}(t) \geq 1$ for $\Lcal^1$-almost every $t \in [0,T]$, for which
\begin{align}\label{eq:MT}
\Delta_{\mu_0,r_T}(\Qpazo(s),\Qpazo(t)) \le \int_{s}^t M_{\Qpazo}(\theta)\,\mathrm{d}\theta
\end{align}
for all times $0 \leq s\leq t \leq T$. We fix in addition some $\varepsilon_0\in(0,1)$ for which $r_T \epsilon_0 \leq 1$, and
\begin{align}\label{absint}
\int_{A} M_{\Qpazo}(s)\,\mathrm{d}s\le 1 + \Mpazo_1(\mu_0) \qquad \text{whenever} \qquad \Lcal^1(A) \le \varepsilon_0,
\end{align}
which is always possible by the absolute continuity of the Lebesgue integral, see e.g. \cite[Theorem 2.5.7]{BogachevI}. Throughout this subsection, we suppose that $\{C_{\varepsilon}\}_{\varepsilon\in(0,\varepsilon_0)}\subseteq[0,T]$ is a family of compact sets whose elements are Lebesgue points of $M(\cdot) \in L^1([0,T],\R_+)$ at which \eqref{eq:IntegralViabilityCond} holds, and which comply with the conditions
\begin{equation}
\label{eq:Ceps}
C_{\varepsilon_2}\subseteq C_{\varepsilon_1} \quad \text{for }\varepsilon_1\le\varepsilon_2\qquad\text{and}\qquad \Lcal^1([0,T]\setminus C_\varepsilon) \le \varepsilon 
\end{equation}
for all $\varepsilon \in(0,\varepsilon_0)$.


\subsubsection*{An ordered family of curves}
\label{sub_sec_aof}
In what follows, we work with triples of the form $(\tau, \{[a_i,b_i)\}_{i\in\Lambda}, \mu(\cdot))$ where $\tau\in(0,T]$ is given, $\{[a_i,b_i)\}_{i\in\Lambda}$ is a collection of intervals with $a_i<b_i$ and $\mu : [0,\tau] \to \Pcal_1(\mathbb R^d)$. The second component is the set of its intervals, the indices serving only to enumerate it. For each $\varepsilon\in(0,\varepsilon_0)$, we denote by $\Fpazo_{\varepsilon}$ the set of all such triples which additionally satisfy the following properties. 

\begin{taggedprop}{\textnormal{(PT)}} \hfill
\label{hyp:PT}
\begin{itemize}
\item[$(i)$] The index set $\Lambda$ is at most  countable  and the family $\{[a_i,b_i)\}_{i\in\Lambda}$ consists of nonempty pairwise disjoint intervals such that 
\begin{equation*}
b_i - a_i \le \varepsilon \quad \text{for each $i \in \Lambda$} \qquad \text{and}\qquad \bigcup_{i\in \Lambda}[a_i,b_i) = [0,\tau).
\end{equation*}
\item[$(ii)$] The curve $\mu:[0,\tau]\to\mathcal P_1(\mathbb R^d)$ is absolutely continuous with $\mu(0)=\mu_0$, {and satisfies 
\begin{equation*}
\Mpazo_1(\mu(t)) \leq 2 \bigg( 1 + \Mpazo_1(\mu_0) + \INTSeg{M(s)}{s}{0}{t} \bigg) \exp \bigg( 2 \INTSeg{M(s)}{s}{0}{t} \bigg)
\end{equation*}}
as well as 
\begin{equation*}
\left\{
\begin{aligned}
& W_1(\mu_0,\mu(t))\le {2(1 + c_T)} \bigg( 1 + \int_{0}^{t}M(s)\,\mathrm{d}s \bigg), \\
& W_1\big(\mu(b_i) \, ;\Qpazo(b_i)\big) \le b_i\exp\bigg(\int_{0}^{b_i}M(s)\,\mathrm{d}s\bigg) \varepsilon,
\end{aligned}
\right.
\end{equation*}
for all times $t \in [0,\tau]$ and each $i \in \Lambda$. 
\item[$(iii)$]  For each $i\in\Lambda$ such that $a_i\in C_\varepsilon$, the curve $\mu : [a_i,b_i) \to \Pcal_1(\R^d)$ solves the continuity inclusion
\begin{equation*}
\partial_t \mu(t) \in - \Div_x \Big( V \big(t,\mathbb B (\mu(t), r_T\varepsilon)\big) \mu(t)\Big), 
\end{equation*}
namely there exists an integrable selection $t \in [a_i,b_i] \mapsto v_i(t) \in V\big(t,\B(\mu(t),r_T \epsilon)\big)$ such that
\begin{equation*}
\int_{a_i}^{b_i} \int_{\mathbb R^d}\Big( \partial_t\varphi_i(t,x) + \langle \nabla_x \varphi_i(t,x) , v_i(t,x) \rangle\Big) \mathrm{d}\mu(t)(x)\,\mathrm{d}t = 0
\end{equation*}
for every $\varphi_i \in C_c^\infty((a_i,b_i) \times\mathbb R^d)$. {In addition, it holds that 
\begin{equation*}
W_1(\mu(s),\mu(t)) \leq 2(1+c_T)\int_s^t M(\theta)\,\mathrm{d}\theta
\end{equation*}
for all times $a_i \leq s \leq t < b_i$.}
\item[$(iv)$] For each $i\in\Lambda$ such that $a_i\notin C_\varepsilon$, the interval $[a_i,b_i)$ is contained in $ [0,T]\setminus C_\varepsilon$ and {the curve $\mu : [a_i,b_i) \to \Pcal_1(\R^d)$ satisfies}
\begin{equation*}
W_1(\mu(s),\mu(t)) \le \int_s^t M_{\Qpazo}(\theta)\,\mathrm{d}\theta
\end{equation*}
for all times $a_i\le s\le t\le b_i$.
\end{itemize}
\end{taggedprop}

At this stage, we consider over $\Fpazo_\varepsilon$ the binary relation $\preceq$ given by \smallskip
\begin{equation}
\label{eq:PartialOrder1}
\big(\tau_1, \{[a^1_i,b^1_i)\}_{i\in\Lambda_1}, \mu_1(\cdot)\big)\preceq 	\big(\tau_2, \{[a^2_i,b^2_i)\}_{i\in\Lambda_2}, \mu_2(\cdot)\big)
\end{equation}
if and only if 
\begin{equation}
\label{eq:PartialOrder2}
\tau_1\le \tau_2,\quad \{[a^1_i,b^1_i)\}_{i\in\Lambda_1}\subseteq\{[a^2_i,b^2_i)\}_{i\in\Lambda_2} \quad \text{and} \quad \mu_2(t) = \mu_1(t) ~~ \text{for all $t\in[0,\tau_1]$}.
\end{equation}
In what follows, we show that the family $\Fpazo_{\epsilon}$ is nonempty under our working assumptions, and that $\preceq$ defines a \textit{partial order}, namely a reflexive, anti-symmetric and transitive binary relation.

\begin{proposition}[Nontriviality and partial ordering]
\label{prop:nonemptiness}
Suppose that  $\Qpazo : [0,T] \rightrightarrows \Pcal_1(\R^d)$ is a left absolutely continuous mapping with nonempty proximal images, which satisfies the viability condition \eqref{eq:IntegralViabilityCond}. 
Let $\varepsilon\in(0,\varepsilon_0)$ be given and $C_{\epsilon} \subseteq [0,T]$ be as above. Then, the family $\Fpazo_\varepsilon$ is nonempty and the relation $\preceq$ defines a partial order.
\end{proposition}

\begin{proof}
We only prove that $\Fpazo_\varepsilon$ is nonempty, since checking that $\preceq$ defines a partial order is quite straightforward. The rest of the proof is split into two cases, depending on whether $0$ belongs to the set $C_\varepsilon \subseteq[0,T]$ or not.

\medskip

\noindent \textbf{Case 1 $(0\notin C_\varepsilon)$.} 
Since $C_\varepsilon \subseteq[0,T]$ is closed, we can find  $\tau\in(0,T]$ such that $\tau\leq\varepsilon$ and $[0, \tau)\subseteq [0,T]\setminus C_\varepsilon$. Let then $\mu_{\tau}\in \Qpazo(\tau)$ be such that $W_1(\mu_0 , \mu_\tau) = W_1(\mu_0 \, ; \Qpazo(\tau))$, which is licit since $\Qpazo(\tau)$ is proximal, {and observe that
\begin{equation*}
\begin{aligned}
W_1(\mu_0,\mu_\tau) & \le \Delta_{\mu_0,r_T}(\Qpazo(0),\Qpazo(\tau)) \\
& \leq \INTSeg{M_{\Qpazo}(s)}{s}{0}{\tau} \\
& \leq 2(1+c_T) \bigg( 1 + \INTSeg{M(s)}{s}{0}{\tau} \bigg)
\end{aligned}
\end{equation*}
by \eqref{eq:cTDef}, \eqref{eq:MT} and \eqref{absint} combined with the left absolute continuity of $\Qpazo : [0,T] \rightrightarrows \Pcal_1(\R^d)$ and the definition \eqref{eq:HausdorffSemiDist} of the one-sided Hausdorff semidistance}. Fix now some $\gamma\in\Gamma_{o}\big(\mu_0,\mu_\tau\big)$, and let
\begin{equation*}
\mu(t):= \bigg(\pi^1 + \displaystyle\frac{\int_{0}^t M_{\Qpazo}(\theta)\,\mathrm{d}\theta}{\int_{0}^\tau M_{\Qpazo}(\theta)\,\mathrm{d}\theta} \big(\pi^2 - \pi^1 \big) \bigg)_{\raisebox{4pt}{\hspace{-0.05cm}$\scriptstyle{\sharp}$}} \, \gamma
\end{equation*}
for all times $t \in [0,\tau]$. {It is then straightforward to check that the triple $(\tau, \{[0,\tau)\}, \mu(\cdot))$ satisfies \ref{hyp:PT}-$(i)$ with $[a_1,b_1) = [0,\tau)$ and $\Lambda=\{1\}$}. {Regarding \ref{hyp:PT}-$(iv)$, one may simply note that 
\begin{equation*}
\begin{aligned}
W_{1}(\mu(s),\mu(t)) & \le \bigg( \frac{\int_{s}^tM_{\Qpazo}(\theta)\,\mathrm{d}\theta}{\int_{0}^\tau M_{\Qpazo}(\theta)\,\mathrm{d}\theta} \bigg) \| \pi^2-\pi^1\|_{L^1(\R^{2d},\R^d; \gamma)} \\
& = \bigg( \frac{\int_{s}^tM_{\Qpazo}(\theta)\,\mathrm{d}\theta}{\int_{0}^\tau M_{\Qpazo}(\theta)\,\mathrm{d}\theta} \bigg) W_{1}(\mu_0,\mu_\tau)\\
& \le \bigg( \frac{\int_{s}^tM_{\Qpazo}(\theta)\,\mathrm{d}\theta}{\int_{0}^\tau M_{\Qpazo}(\theta)\,\mathrm{d}\theta} \bigg) \Delta_{\mu_0,r_T}(\Qpazo(0),\Qpazo(\tau)) \\
& \le  \int_{s}^tM_{\Qpazo}(\theta)\,\mathrm{d}\theta
\end{aligned}
\end{equation*}
for all times $0 \leq s\le t \leq \tau$. Lastly, to check \ref{hyp:PT}-$(ii)$, observe first that the curve $\mu : [0,\tau] \to \Pcal_1(\R^d)$ is continuous and such that $\mu(0)=\mu_0$ and $\mu(\tau)=\mu_\tau \in \Qpazo(\tau)$ by construction. Moreover, by what precedes, it holds that
\begin{equation*}
\begin{aligned}
W_1(\mu_0,\mu(t)) & \leq \INTSeg{M_{\Qpazo}(s)}{s}{0}{t} \leq 1 + \Mpazo_1(\mu_0),
\end{aligned}
\end{equation*}
for all times $t \in [0,\tau]$, from which we may infer all the remaining estimates from \ref{hyp:PT}-$(ii)$}. Recalling that $0 \notin C_{\epsilon}$, we have shown that the triple $(\tau,  \{[0,\tau)\}, \mu(\cdot))$ belongs to $\Fpazo_\varepsilon$. 

\medskip

\noindent \textbf{Case 2 $(0\in C_\varepsilon)$.}  Let $\delta \in (0,1)$ with $\delta\leq T$ be such that
\begin{equation*}
{\INTSeg{M(s)}{s}{0}{\delta} \leq \frac{\epsilon}{2 \exp \Big( \INTSeg{M(t)}{t}{0}{T} \Big)}}.
\end{equation*}
Let $c>0$ be as defined above, and take some $h_1\in(0,\delta)$ so that 
\begin{align}\label{kem}
\bigg(\frac{1}{h} \int_{0}^{h}M(s)\,\mathrm{d}s\bigg) \bigg( \int_{0}^{h}M(s)\,\mathrm{d}s \bigg) \leq \frac{\varepsilon}{3 c}
\end{align}
for all $h \in (0,h_1)$. Note that this is possible because $0$ is a Lebesgue point of $M(\cdot) \in L^1([0,T],\R_+)$, according to the definition of $C_\varepsilon \subseteq [0,T]$. By the viability condition \eqref{eq:IntegralViabilityCond}, there exists $w\in C^0(\mathbb R^d,\mathbb R^d)$ such that
\begin{equation}
\label{eq:ViabCondProof} 
w\in D\Qpazo(0|\mu_0) \cap \Liminf{h\to 0^+} \,\frac{1}{h}\int_{0}^{h} V\big(s,\B(\mu_0,\varepsilon)\big) \,\mathrm{d}s.
\end{equation}
In particular, owing to the definition \eqref{eq:GraphicalDef} of graphical derivative, there exists a vanishing sequence $h_i\to0^+$ as $i \to +\infty$, such that 
\begin{equation}
\label{eq:ViabCondProofBis}
\lim_{i \to +\infty} \frac{1}{h_i} W_1 \Big( (\Id+h_i w)_{\sharp}\mu_0 \,; \Qpazo(h_i) \Big) = 0.
\end{equation}
Besides, it also stems from \eqref{eq:ViabCondProof} that there exists a sequence $\{v_i(\cdot)\}_{i\in\mathbb N}$ of integrable selections $t \in [0,h_i] \mapsto v_i(t) \in V(t,\B(\mu_0,\epsilon))$ such that 
\begin{align*}
\dsf_{cc} \Big( \tfrac{1}{h_i}\mathsmaller{\INTSeg{v_i(s)}{s}{0}{h_i}} , w \Big) ~\underset{i \to +\infty}{\longrightarrow}~ 0.
\end{align*}
Since $\Mpazo_1(\nu)\le c_T+1$ for $\nu\in\B(\mu_0,\epsilon)$, the selected fields satisfy $|v_i(t,x)|\le A_T M(t)(1+|x|)$ and $\Lip(v_i(t))\le M(t)$. We extend them by zero outside $[0,h_i]$ when applying Lemma \ref{lemcc} with $A_T M$ in place of $M$ and the constant $c$ fixed above.
It follows then from Hypothesis \ref{hyp:USC}-$(ii)$ combined, e.g., with \cite[Lemma 2.15]{ViabPp} (see also the details given in the proof of Theorem \ref{Thm1}) that 
\begin{align*}
\Big\| \, w - \tfrac{1}{h_i} \mathsmaller{\INTSeg{v_i(s)}{s}{0}{h_i}} \, \Big\|_{L^1(\R^d,\R^d;\,\mu_0)} ~\underset{i \to +\infty}{\longrightarrow}~ 0, 
\end{align*}
wherefore we can find $i_1\in\mathbb N$ such that
\begin{align*}
W_1\bigg( \Big(\Id + \mathsmaller{\int_{0}^{h_i} v_i(s)\,\mathrm{d}s} \Big)_{\raisebox{2pt}{\hspace{-0.05cm}$\scriptstyle{\sharp}$}} \,\mu_0,  (\Id + h_i w)_{\sharp} \mu_0 \bigg) \le \frac{\varepsilon}{3} h_i
\end{align*}
for each $i \geq i_1$. Consider lastly $i_0\in\mathbb N$ with $i_0\ge i_1$ and $0<h_{i_0}<\min\{h_1,T,\varepsilon\}$, chosen so that 
\begin{align*}
W_1\Big( (\Id + h_{i_0} w)_{\sharp}\mu_0 \, ; \Qpazo(h_{i_0})\Big)\le \frac{\varepsilon}{3} h_{i_0}, 
\end{align*}
whose existence is ensured by \eqref{eq:ViabCondProofBis}. Then, upon combining the two previous estimates, we  get  
\begin{align}\label{l1}
W_1\bigg( \Big(\Id + \mathsmaller{\int_{0}^{h_{i_0}} v_{i_0}(s)\,\mathrm{d}s} \Big)_{\raisebox{2pt}{\hspace{-0.05cm}$\scriptstyle{\sharp}$}} \,\mu_0 \, ; \Qpazo(h_{i_0})\Big)\le \frac{2}{3}\varepsilon h_{i_0}.  
\end{align} 
Observe now that by Lemma \ref{lemcc}, the Cauchy problem
\begin{equation*}
\left\{ 
\begin{aligned}
& \partial_t \mu(t) + \Div_x(v_{i_0}(t) \mu(t)) = 0, \\
& \mu(0) = \mu_0, 
\end{aligned}
\right.
\end{equation*}
admits a unique solution $\mu(\cdot)\in \AC([0,T],\Pcal_1(\R^d))$ which satisfies
\begin{align}\label{l2}
W_1 \bigg( \mu(h_{i_0}) , \Big(\Id + \mathsmaller{\int_{0}^{h_{i_0}} v_{i_0}(s)\,\mathrm{d}s} \Big)_{\raisebox{2pt}{\hspace{-0.05cm}$\scriptstyle{\sharp}$}} \, \mu_0 \bigg) \le  c \bigg(\frac{1}{h_{i_0}}\int_{0}^{h_{i_0}}M(s)\,\mathrm{d}s\bigg)\bigg(\int_{0}^{h_{i_0}}M(s)\,\mathrm{d}s\bigg) h_{i_0}. 
\end{align}
Upon combining the estimates of (\ref{kem}), (\ref{l1}) and (\ref{l2}), we finally obtain that 
\begin{align}\label{l33}
W_1( \mu(h_{i_0}) \, ; \Qpazo(h_{i_0})) \le \varepsilon h_{i_0}. 
\end{align}
Set now $\tau:=h_{i_0}$, define $v: t \in [0,\tau] \mapsto v_{i_0}(t) \in V(t,\B(\mu_0,\epsilon))$, and let $\mu:[0,\tau]\to\mathcal P_{1}(\mathbb R^d)$ be as above. Note that owing to the basic moment estimate from \eqref{eq:BasicEstimates} combined with Hypothesis \ref{hyp:USC}-$(ii)$, one has that
\begin{equation*}
\begin{aligned}
\Mpazo_1(\mu(t)) & \leq \Mpazo_1(\mu_0) + \INTSeg{\NormL{v(s)}{1}{\R^d,\R^d;\, \mu(s)}}{s}{0}{t} \\
& \leq \Mpazo_1(\mu_0) + \INTSeg{M(s) \bigg( 1 + \Mpazo_1(\mu(s)) + \sup \Big\{ \Mpazo_1(\nu) ~\,\textnormal{s.t.}~ \nu \in \B(\mu_0,\epsilon) \Big\} \bigg) }{s}{0}{t} \\
& \leq \Mpazo_1(\mu_0) + \INTSeg{M(s) \Big( 2 + \Mpazo_1(\mu(s)) + \Mpazo_1(\mu_0) \Big)}{s}{0}{t}
\end{aligned}
\end{equation*}
where we used the fact that $\epsilon \leq 1/r_T < 1$. By Gr\"onwall's lemma, it follows from the previous inequality that
\begin{equation*}
\Mpazo_1(\mu(t)) \leq \Bigg( \Mpazo_1(\mu_0) + \Big( 2 + \Mpazo_1(\mu_0) \Big) \INTSeg{M(s)}{s}{0}{t} \Bigg)
\exp \bigg( \INTSeg{M(s)}{s}{0}{t} \bigg)
\end{equation*}
for all times $t \in [0,\tau]$. Recalling that $\Mpazo_1(\mu_0) \leq c_T$ and using the nonnegativity of $M(\cdot)$, we further obtain the crude bound
\begin{equation*}
\Mpazo_1(\mu(t)) \leq \Bigg( c_T \bigg( 1 + \INTSeg{M(s)}{s}{0}{T} \bigg) + 2 \INTSeg{M(s)}{s}{0}{T} \Bigg)
\exp \bigg( \INTSeg{M(s)}{s}{0}{T} \bigg)
\end{equation*}
for all times $t \in [0,\tau]$. Then, upon plugging the latter quantity in the absolute continuity estimate of \eqref{eq:BasicEstimates} while recalling our earlier choice of $\delta \in (0,1)$, we further obtain 
\begin{equation*}
\begin{aligned}
W_1(\mu_0,\mu(t))
& \leq \INTSeg{M(s) \Big( 2 + \Mpazo_1(\mu(s)) + \Mpazo_1(\mu_0) \Big)}{s}{0}{t} \\
& \leq
\Bigg(
2 + \Mpazo_1(\mu_0)
+ c_T\Bigg(1+\INTSeg{M(s)}{s}{0}{T}\Bigg)
   \exp\Bigg(\INTSeg{M(s)}{s}{0}{T}\Bigg) \\
& \qquad\qquad
+ 2\INTSeg{M(s)}{s}{0}{T}
   \exp\Bigg(\INTSeg{M(s)}{s}{0}{T}\Bigg)
\Bigg)
\INTSeg{M(s)}{s}{0}{\delta} \\
& \leq
(1+c_T)
\Bigg(1+\INTSeg{M(s)}{s}{0}{T}\Bigg)\epsilon \\
& \leq (r_T - 1)\epsilon,
\end{aligned}
\end{equation*}
where we used the choice of $\delta$, the fact that
$\Mpazo_1(\mu_0) \leq c_T$, and the definitions of $c_T$ and $r_T$,
from whence we may directly infer that
\begin{align}\label{l4}
v(t)\in V\big(t,\B(\mu_0,\epsilon)\big) \subseteq V\big(t,\mathbb B(\mu(t), r_T \epsilon)\big)
\end{align}
for $\Lcal^1$-almost every $t\in[0,\tau]$. At this stage, recalling that $r_T \epsilon \leq 1$, it follows from \eqref{l4} combined again with the moment estimate from \eqref{eq:BasicEstimates} that the following refined bound
\begin{equation}
\label{eq:MomentEstProofInit}
\begin{aligned}
\Mpazo_1(\mu(t)) & \leq \Mpazo_1(\mu_0) + \INTSeg{M(s) \bigg( 1 + \Mpazo_1(\mu(s)) + \sup \Big\{ \Mpazo_1(\nu) ~\,\textnormal{s.t.}~ \nu \in \B(\mu(s),r_T\epsilon) \Big\} \bigg) }{s}{0}{t} \\
& \leq \Mpazo_1(\mu_0) + 2 \INTSeg{M(s) \Big( 1 + \Mpazo_1(\mu(s)) \Big)}{s}{0}{t}
\end{aligned}
\end{equation}
holds for all times $t \in [0,\tau]$, so that
\begin{equation}
\label{eq:InitialisationMomentBound}
\Mpazo_1(\mu(t)) \leq 2 \bigg( \Mpazo_1(\mu_0) + \INTSeg{M(s)}{s}{0}{t} \bigg) \exp \bigg( 2 \INTSeg{M(s)}{s}{0}{t} \bigg)
\end{equation}
by Gr\"onwall's lemma. This combined with the absolute continuity estimate in \eqref{eq:BasicEstimates} and the definition \eqref{eq:cTDef} of the constant $c_T > 0$ further yields 
\begin{equation*}
\begin{aligned}
W_1(\mu(s),\mu(t)) \leq \INTSeg{\NormL{v(\theta)}{1}{\R^d,\R^d ; \, \mu(\theta)}}{\theta}{s}{t} \leq 2(1+c_T) \INTSeg{M(\theta)}{\theta}{s}{t}. 
\end{aligned}
\end{equation*}
As a result of which, we have shown that $(\tau,\{[0,\tau)\},\mu(\cdot))$ satisfies \ref{hyp:PT}-$(i)$, $(iii)$ with $[a_1,b_1)=[0,\tau)$ and $\Lambda=\{1\}$. Concerning \ref{hyp:PT}-$(ii)$, observe that the moment inequality has already been established in \eqref{eq:InitialisationMomentBound}, whereas 
\begin{equation*}
\begin{aligned}
W_1(\mu_0,\mu(t)) & \leq 2(1+c_T) \INTSeg{M(s)}{s}{0}{t} \\
& \leq 2(1+c_T) \bigg( 1 + \INTSeg{M(s)}{s}{0}{t} \bigg)
\end{aligned}
\end{equation*}
as a direct consequence of previous computations. Lastly, the fact that 
\begin{equation*}
W_1(\mu(\tau) \, ; \Qpazo(\tau)) \leq \tau \exp \bigg( \INTSeg{M(s)}{s}{0}{\tau} \bigg) \epsilon
\end{equation*}
follows straightforwardly from \eqref{l33} since $\tau = h_{i_0}$. As we assumed that $0 \in C_{\epsilon}$, we have indeed shown that the triple $(\tau,\{[0,\tau)\},\mu(\cdot))$ then belongs to $\Fpazo_\varepsilon$. 
\end{proof}

In the next proposition, we establish useful properties on the elements of the family $\{\Fpazo_\varepsilon\}_{\varepsilon\in(0,\varepsilon_0)}$. Therein and throughout what follows, we let
\begin{equation*}
E:=\bigcup_{a_i\in C_\varepsilon}[a_i,b_i) \qquad \text{and} \qquad F:=\bigcup_{a_i\notin C_\varepsilon}[a_i,b_i)    
\end{equation*}
and note that $E \cup F = [0,\tau)$.

\begin{proposition}[Regularity and feasibility estimates for admissible triples]
\label{prop:Estimates}
Let $\varepsilon\in(0,\varepsilon_0)$ and $(\tau, \{[a_i,b_i)\}_{i\in\Lambda}, \mu(\cdot))\in\Fpazo_\varepsilon$. Then, the following holds.  
\begin{itemize}
\item[$(a)$] {For all times $t \in [0,\tau]$, one has that
\begin{equation*}
\Mpazo_1(\mu(t)) \leq \Mpazo_1(\mu_0) + 2 \INTDom{M(s) \Big( 1 + \Mpazo_1(\mu(s)) \Big)}{[0,t] \cap E}{s} + \INTDom{M_{\Qpazo}(s)}{[0,t] \cap F}{s}.
\end{equation*}}
\item[$(b)$] For all times $0 \leq s \leq t \leq \tau$, one has that
\begin{equation*}
W_1(\mu(s),\mu(t)) \le {2(1+c_T)} \int_{[s,t] \cap E}M(\theta)\,\mathrm{d}\theta + \int_{[s,t] \cap F}M_{\Qpazo}(\theta)\,\mathrm{d}\theta.
\end{equation*}
\item[$(c)$] For each $i\in\Lambda$, one has that 
\begin{equation*}
W_1(\mu(a_i) \, ; \Qpazo(a_i)) \le \displaystyle a_i\exp\bigg(\int_{0}^{a_i}M(s)\,\mathrm{d}s \bigg) \varepsilon.
\end{equation*}
\end{itemize}
\end{proposition}

\begin{proof}
{We begin by proving item $(a)$. To this end, note that if $i \in \Lambda$ is such that $a_i \in C_{\epsilon}$, it then follows by repeating the computations of \eqref{eq:MomentEstProofInit} above that 
\begin{equation*}
\begin{aligned}
\Mpazo_1(\mu(t)) \leq \Mpazo_1(\mu(a_i)) + 2\INTSeg{M(\theta) \Big( 1 + \Mpazo_1(\mu(\theta))\Big)}{\theta}{a_i}{t}
\end{aligned}
\end{equation*}
for all times $t \in [a_i,b_i)$. In the situation in which $a_i \notin C_{\epsilon}$, we straightforwardly get that
\begin{equation*}
\begin{aligned}
\Mpazo_1(\mu(t)) & \leq \Mpazo_1(\mu(a_i)) + W_1(\mu(a_i),\mu(t)) \\
& \leq \Mpazo_1(\mu(a_i)) + \INTSeg{M_{\Qpazo}(\theta)}{\theta}{a_i}{t}.
\end{aligned}
\end{equation*}
Consequently, given any $t \in [0,\tau)$ and denoting by $i_t \in \Lambda$ the index for which $t \in [a_{i_t},b_{i_t})$, while letting $\Lambda_t \subseteq \Lambda$ be the subfamily for which  
\begin{equation*}
\bigcup_{i \in \Lambda_t} [a_i,b_i) = [0,a_{i_t}),
\end{equation*}
we infer from what precedes that
\begin{equation*}
\begin{aligned}
\Mpazo_1(\mu(t)) & \leq \Mpazo_1(\mu_0) + \sum_{i \in \Lambda_t} \Big( \Mpazo_1(\mu(b_i)) - \Mpazo_1(\mu(a_i)) \Big) + W_1(\mu(a_{i_t}),\mu(t)) \\
& \leq \Mpazo_1(\mu_0) + 2 \sum_{\substack{i\in\Lambda_t \\ a_i\in C_\varepsilon}} \hspace{-0.1cm} \INTSeg{ \hspace{-0.025cm} M(\theta) \Big( 1 + \Mpazo_1(\mu(\theta)) \Big)}{\theta}{a_i}{b_i} + \sum_{\substack{i\in\Lambda_t \\ a_i\notin C_\varepsilon}} \hspace{-0.1cm} \INTSeg{\hspace{-0.025cm} M_{\Qpazo}(s)}{s}{a_i}{b_i} + W_1(\mu(a_{i_t}),\mu(t)) \\
&  \leq \Mpazo_1(\mu_0) + 2 \INTDom{ \hspace{-0.05cm} M(s)\Big( 1+ \Mpazo_1(\mu(s)) \Big)}{[0,a_{i_t}] \cap E}{s} + \INTDom{\hspace{-0.05cm} M_{\Qpazo}(s)}{[0,a_{i_t}] \cap F}{s} + W_1(\mu(a_{i_t}),\mu(t)) \\
& \leq \Mpazo_1(\mu_0) + 2 \INTDom{M(s)\Big( 1+ \Mpazo_1(\mu(s)) \Big)}{[0,t] \cap E}{s} + \INTDom{M_{\Qpazo}(s)}{[0,t] \cap F}{s}
\end{aligned}
\end{equation*}
where we again used the basic estimates from \eqref{eq:BasicEstimates}}. The estimate at $t=\tau$ follows by continuity. 

We now move on to item $(b)$, and start by fixing $0 \leq s \leq t < \tau$ along with a pair of indices $i_{s},i_{t}\in \Lambda$ such that $s\in[a_{i_s},b_{i_s})$ and $t\in[a_{i_t},b_{i_t})$. If $i_s=i_t=i$, the desired estimate follows directly from \ref{hyp:PT}-$(iii)$ when $a_i\in C_\varepsilon$, and from \ref{hyp:PT}-$(iv)$ otherwise. We may thus suppose that $i_s\neq i_t$. Denoting by $\Lambda_{s,t}\subseteq \Lambda$ the subfamily for which 
\begin{equation*}
\bigcup_{i\in \Lambda_{s,t}}[a_{i},b_i) = [b_{i_s},a_{i_{t}}),
\end{equation*}
we obtain by repeatedly applying the triangle inequality that 
\begin{equation*}
\begin{aligned}
& W_1(\mu(s),\mu(t)) \\
& \hspace{0.6cm} \le W_1(\mu(s),\mu(b_{i_s})) + \sum_{i\in\Lambda_{s,t}}W_1(\mu(a_i),\mu(b_i)) + W_1(\mu(a_{i_t}),\mu(t))\\
& \hspace{0.6cm}= W_1(\mu(s),\mu(b_{i_s})) + \sum_{\substack{i\in\Lambda_{s,t} \\ a_i\in C_\varepsilon}}W_1(\mu(a_i),\mu(b_i)) + \sum_{\substack{i\in\Lambda_{s,t} \\ a_i\notin C_\varepsilon}}W_1(\mu(a_i),\mu(b_i)) + W_1(\mu(a_{i_t}),\mu(t))\\
& \hspace{0.6cm} \le W_1(\mu(s),\mu(b_{i_s})\big) + {2(1+c_T)} \sum_{\substack{i\in\Lambda_{s,t} \\ a_i\in C_\varepsilon}} \int_{a_i}^{b_i} M(\theta)\,\mathrm{d}\theta + \sum_{\substack{i\in\Lambda_{s,t} \\ a_i\notin C_\varepsilon}}\int_{a_i}^{b_i} M_{\Qpazo}(\theta)\,\mathrm{d}\theta + W_1(\mu(a_{i_t}),\mu(t))\\
& \hspace{0.6cm} \le W_1(\mu(s),\mu(b_{i_s})) + {2(1+c_T)} \int_{[b_{i_s},a_{i_t}]\cap E} \hspace{-0.1cm} M(\theta)\,\mathrm{d}\theta +  \int_{[b_{i_s},a_{i_t}]\cap F} \hspace{-0.1cm} M_{\Qpazo}(\theta)\,\mathrm{d}\theta+ W_1(\mu(a_{i_t}),\mu(t)) \\
& \hspace{0.6cm} \le {2(1+c_T)} \int_{[s,t]\cap E} \hspace{-0.1cm} M(\theta)\,\mathrm{d}\theta +  \int_{[s,t]\cap F} \hspace{-0.1cm} M_{\Qpazo}(\theta)\,\mathrm{d}\theta,
\end{aligned}
\end{equation*}
which is the desired estimate. The estimate for $s< t=\tau$ follows by continuity, while the case $s=t=\tau$ is immediate. 

We conclude by proving item $(c)$. Fix some $i\in\Lambda$, and observe that if $a_i=0$ or $a_i=b_j$ for some $j\in\Lambda$, the result follows directly from the property \ref{hyp:PT}-$(ii)$ satisfied by the triple $(\tau, \{[a_i,b_i)\}_{i\in\Lambda}, \mu(\cdot))$. Otherwise, by \ref{hyp:PT}-$(i)$ above, there exists a sequence $\{i_{k}\}_{k\in\mathbb N}\subseteq\Lambda$ such that $b_{i_k}\to a_i^-$ as $k\to+\infty$. For each $k\in\mathbb N$, choose $\nu_k\in \Qpazo(b_{i_k})$ such that $W_1(\mu(b_{i_k}), \nu_k) = W_1(\mu(b_{i_k}) \,; \Qpazo(b_{i_k}))$, which is possible since $\Qpazo(b_{i_k})$ is proximal, and note that 
\begin{equation*}
\begin{aligned}
& W_{1}\big(\mu(a_i) \,; \Qpazo(a_i)\big) \\
& \hspace{0.6cm} \le W_1(\mu(a_i),\mu(b_{i_k})) + W_1(\mu(b_{i_k}),\nu_{k}) + W_{1}(\nu_k \, ; \Qpazo(a_i)) \\
& \hspace{0.6cm} \le W_1(\mu(a_i),\mu(b_{i_k})) + W_1(\mu(b_{i_k}) \, ; \Qpazo(b_{i_k}))  + \Delta_{\mu_0,r_T}(\Qpazo(b_{i_k}), \Qpazo(a_i)) \\
&\hspace{0.6cm} \le \int_{b_{i_k}}^{a_i} \max\Big\{{2(1+c_T)}M(s),M_{\Qpazo}(s)\Big\}\,\mathrm{d}s + b_{i_{k}}\exp\bigg(\int_{0}^{b_{i_k}} M(s)\,\mathrm{d}s\bigg)\varepsilon + \int_{b_{i_k}}^{a_i} M_{\Qpazo}(s)\,\mathrm{d}s.
\end{aligned}
\end{equation*}
Taking the limit as $k\to +\infty$ then yields the desired result. 
\end{proof}

\begin{corollary}[Global feasibility estimate for admissible triples]
\label{corclsvia}
For any $\varepsilon\in(0,\varepsilon_0)$ and each $(\tau, \{[a_i,b_i)\}_{i\in\Lambda}, \mu(\cdot))\in\Fpazo_\varepsilon$, it holds that
\begin{equation*}
W_1(\mu(t) \, ; \Qpazo(t)) \le 2\int_{t-\varepsilon}^t\max\Big\{{2(1+c_T)}M(s),M_{\Qpazo}(s)\Big\}\,\mathrm{d}s + t\exp\bigg(\int_{0}^t M(s)\,\mathrm{d}s\bigg) \varepsilon 
\end{equation*}
for all times $t\in[\varepsilon,\tau]$.
\end{corollary}

\begin{proof}
Fix $t\in[\varepsilon,\tau]$. It is enough to consider $t<\tau$, since the case $t=\tau$ follows by approximation from the left, using the continuity of $\mu(\cdot)$ and the left absolute continuity of $\Qpazo$. We may therefore suppose that $t<\tau$. Consider some $i\in\Lambda$ such that $t\in [a_i,b_i)$, and let $\nu\in \Qpazo(a_i)$ be such that $W_1(\mu(a_i), \nu) = W_1(\mu(a_i) \, ; \Qpazo(a_i))$, which is possible since $\Qpazo(a_i)$ is proximal. Then, from Proposition \ref{prop:Estimates}-$(c)$ above, we get that
\begin{equation*}
\begin{aligned}
& W_1(\mu(t) \, ;\Qpazo(t)) \\
& \hspace{0.6cm} \le W_1(\mu(t), \mu(a_i)) + W_1(\mu(a_i), \nu\big) + 	W_1(\nu \,; \Qpazo(t))\\
&\hspace{0.6cm} \le  \int_{a_i}^t\max\Big\{2(1+c_T) M(s),M_{\Qpazo}(s)\Big\}\,\mathrm{d}s  +  a_i \exp\bigg(\int_{0}^{a_i} M(s)\,\mathrm{d}s\bigg)\varepsilon +  \Delta_{\mu_0,r_T}\big(\Qpazo(a_i) \, ; \Qpazo(t)\big)\\
&\hspace{0.6cm} \le \int_{t-\varepsilon}^t \max\Big\{ 2(1+c_T)M(s),M_{\Qpazo}(s)\Big\}\,\mathrm{d}s  +  t \exp\bigg(\int_{0}^{t}M(s)\,\mathrm{d}s\bigg)\varepsilon + \int_{t-\varepsilon}^t M_{\Qpazo}(s)\,\mathrm{d}s, 
\end{aligned}
\end{equation*}
from whence the thesis follows. 
\end{proof}


\subsubsection*{Maximal elements are defined over the whole interval}

In this second part, we show that maximal elements of the partially ordered set $(\Fpazo_\varepsilon, \preceq)$, whenever they exist, are defined over the whole interval $[0,T]$. To do so, we need to distinguish two cases depending on whether $\tau$ belongs to $C_{\epsilon} \subseteq [0,T]$ or not. In both scenarios, we show that a candidate maximal triple $(\tau,\{[a_i,b_i)\}_{i \in \Lambda},\mu(\cdot))$ for which $\tau < T$ can always be extended, thereby leading to a contradiction. 

\begin{proposition}[Maximal elements are globally defined when $\tau \notin C_{\epsilon}$]
\label{prop:GlobalNotCeps}
For every $\varepsilon\in(0,\varepsilon_0)$, if the triple $(\tau,\{[a_i,b_i)\}_{i\in\Lambda},\mu(\cdot))$ is a maximal element in $(\Fpazo_\varepsilon, \preceq)$ and $\tau\notin C_\varepsilon$, then $\tau = T$.
\end{proposition}

\begin{proof}
Suppose by contradiction that $\tau<T$ and choose $\sigma\in(\tau,T]$ such that $0<\sigma-\tau\leq\varepsilon$ and $[\tau, \sigma)\subseteq [0,T]\setminus C_\varepsilon$. Set $F^*:=F\cup[\tau,\sigma)$, so that $F^*\subseteq[0,T]\setminus C_\varepsilon$ and $\Lcal^1(F^*)\leq\varepsilon$. Define the new index set $\Lambda^*$ in such a way that $\{[a_i,b_i)\}_{i \in \Lambda^*} :=\{[a_i,b_i)\}_{i \in \Lambda} \cup\{[\tau,\sigma)\}$, and take two elements $\mu_\tau\in \Qpazo(\tau)$ and $\mu_{\sigma}\in \Qpazo(\sigma)$ for which 
\begin{equation*}
W_1(\mu(\tau), \mu_\tau) = W_1(\mu(\tau) \,; \Qpazo(\tau)) \qquad \text{and} \qquad W_1\big(\mu_\tau , \mu_\sigma\big) = W_1(\mu_\tau \,; \Qpazo(\sigma)).
\end{equation*}
Note that such a choice is possible since $\Qpazo(\tau)$ and $\Qpazo(\sigma)$ are proximal. Given a pair of optimal plans $\alpha\in \Gamma_o(\mu(\tau),\mu_\tau)$ and $\beta\in\Gamma_o(\mu_\tau, \mu_\sigma)$, there exists by virtue of the gluing lemma, see e.g. \cite[Lemma 5.3.2]{AGS}, a triple plan $\eta\in \Pcal_1(\R^{3d})$ such that $\pi^{1,2}_{\sharp}\eta = \alpha$ and $\pi^{2,3}_{\sharp}\eta=\beta$. This allows us in turn to define the measure
\begin{equation*}
\mu_\sigma^*:=\big(\pi^1 + \pi^3 - \pi^2\big)_{\sharp}\eta
\end{equation*}
where $\pi^1,\pi^2,\pi^3:\mathbb R^{3d}\to\mathbb R^d$ stand for the canonical projections onto the first, second and third components, respectively. At this stage, take any optimal plan $\gamma\in \Gamma_o\big(\mu(\tau), \mu_\sigma^*\big)$ and define the extended curve $\mu^*:[0,\sigma]\to\mathcal P_{1}(\mathbb R^d)$  by
\begin{equation*}
\mu^*(t):=\begin{dcases*}
\mu(t)
& if  $t\in[0,\tau]$, \\
\bigg(\pi^1 + \displaystyle\frac{\int_{\tau}^t M_{\Qpazo}(\theta)\,\mathrm{d}\theta}{\int_{\tau}^\sigma M_{\Qpazo}(\theta)\,\mathrm{d}\theta} \big( \pi^2 - \pi^1 \big)\bigg)_{\raisebox{4pt}{\hspace{-0.05cm}$\scriptstyle{\sharp}$}} \, \gamma
& if $t\in[\tau,\sigma]$.
\end{dcases*}
\end{equation*}
Notice in particular that $\mu^*(\sigma)=\mu_\sigma^*$ by construction. Then, it is straightforward to see that the triple $(\sigma, \{[a_i,b_i)\}_{i\in\Lambda^*}, \mu^*(\cdot))$ satisfies \ref{hyp:PT}-$(i)$. To verify that \ref{hyp:PT}-$(iv)$ holds, observe first that
\begin{equation*}
\begin{aligned}
W_{1}(\mu_0,\mu_\tau)&\le W_1(\mu_0,\mu(\tau)) + W_1(\mu(\tau),\mu_\tau) \\
& \le {2(1+c_T) \bigg( 1 + \int_{0}^\tau M(s)\,\mathrm{d}s \bigg)} +  \tau\exp\bigg(\int_{0}^\tau M(s)\,\mathrm{d}s\bigg)\varepsilon \\
& \le {2(1+c_T) \bigg( 1 + \int_{0}^T M(s)\,\mathrm{d}s \bigg)} +  T \exp\bigg(\int_{0}^T M(s)\,\mathrm{d}s\bigg) \\
& \leq r_T.
\end{aligned}
\end{equation*}
Hence,  $\mu_\tau \in \Qpazo(\tau)\cap \mathbb B\big(\mu_0,r_T\big)$ and $W_1(\mu_\tau,\mu_\sigma)\le \Delta_{\mu_0,r_T}(\Qpazo(\tau),\Qpazo(\sigma))$ by the left absolute continuity of $\Qpazo : [0,T] \rightrightarrows \Pcal_1(\R^d)$ and the definition \eqref{eq:HausdorffSemiDist} of one-sided Hausdorff semidistance. In particular
\begin{equation*}
\begin{aligned}
W_1(\mu^*(\tau),\mu^*(\sigma)) &= W_1\Big(\pi^1_{\sharp}\eta \, , (\pi^1 + \pi^3 -\pi^2)_{\sharp}\eta\Big) \\
& \le \| \pi^3-\pi^2\|_{L^1(\R^{3d},\R^d;\,\eta)} \\
& =  W_1(\mu_\tau,\mu_\sigma) \\
& \le \Delta_{\mu_0,r_T}(\Qpazo(\tau),\Qpazo(\sigma)),
\end{aligned}
\end{equation*}
which allows us to further estimate the distance between $\mu^*(s)$ and $\mu^*(t)$ for any $\tau \leq s \leq t \leq \sigma$ as 
\begin{equation}
\label{eq:DistEstGeod}
\begin{aligned}
W_{1}(\mu^*(s),\mu^*(t)) & \le \bigg( \frac{\int_{s}^tM_{\Qpazo}(\theta)\,\mathrm{d}\theta}{\int_{\tau}^\sigma M_{\Qpazo}(\theta)\,\mathrm{d}\theta} \bigg) \| \pi^2-\pi^1\|_{L^1(\R^{2d},\R^d;\,\gamma)} \\
& =\bigg( \frac{\int_{s}^tM_{\Qpazo}(\theta)\,\mathrm{d}\theta}{\int_{\tau}^\sigma M_{\Qpazo}(\theta)\,\mathrm{d}\theta} \bigg) W_{1}(\mu^*(\tau),\mu^*(\sigma))\\
& \le \bigg( \frac{\int_{s}^tM_{\Qpazo}(\theta)\,\mathrm{d}\theta}{\int_{\tau}^\sigma M_{\Qpazo}(\theta)\,\mathrm{d}\theta} \bigg) \Delta_{\mu_0,r_T}(\Qpazo(\tau),\Qpazo(\sigma)) \\
& \le  \int_{s}^tM_{\Qpazo}(\theta)\,\mathrm{d}\theta, 
\end{aligned}
\end{equation}
from whence we conclude that \ref{hyp:PT}-$(iv)$ holds. We now shift our focus to \ref{hyp:PT}-$(ii)$. To see that the augmented triple $(\sigma,\{[a_i,b_i)\}_{i \in \Lambda^*},\mu^*(\cdot))$ {satisfies the moment estimate therein, note first that by Proposition \ref{prop:Estimates}-$(a)$ combined with \eqref{eq:DistEstGeod} above, one has that
\begin{equation*}
\begin{aligned}
\Mpazo_1(\mu^*(t)) & \leq \Mpazo_1(\mu(\tau)) + W_1(\mu(\tau),\mu^*(t)) \\
& \leq \Mpazo_1(\mu_0) + 2 \INTDom{M(s) \Big( 1 + \Mpazo_1(\mu^*(s)) \Big)}{[0,\tau] \cap E}{s} + \INTDom{M_{\Qpazo}(s)}{[0,\tau] \cap F}{s}  + \INTSeg{M_{\Qpazo}(s)}{s}{\tau}{t} \\
& \leq \Mpazo_1(\mu_0) + 2 \INTSeg{M(s) \Big( 1 + \Mpazo_1(\mu^*(s)) \Big)}{s}{0}{t} + \INTDom{M_{\Qpazo}(s)}{[0,t] \cap F^*}{s} 
\end{aligned}
\end{equation*}
for all times $t \in [\tau,\sigma]$, where we recall that $E=\bigcup_{a_i\in C_\varepsilon}[a_i,b_i)$ and $F=\bigcup_{a_i\notin C_\varepsilon}[a_i,b_i)$ by definition. From Gr\"onwall's lemma combined with the fact that $\Lcal^1(F^*) \leq \epsilon$ and \eqref{absint}, we deduce that
\begin{equation*}
\Mpazo_1(\mu^*(t)) \leq 2 \bigg( 1 + \Mpazo_1(\mu_0) + \INTSeg{M(s)}{s}{0}{t} \bigg) \exp \bigg( 2 \INTSeg{M(s)}{s}{0}{t} \bigg)  
\end{equation*}
for all times $t \in [\tau,\sigma]$}. The derivation of the second estimate in \ref{hyp:PT}-$(ii)$ is based on a similar idea leveraging the decomposition from Proposition \ref{prop:Estimates}-$(b)$, with 
\begin{equation*}
\begin{aligned}
W_1(\mu_0,\mu^*(t)) & \le W_1(\mu_0,\mu(\tau)) + W_1(\mu(\tau),\mu^*(t)) \\
&\le {2(1+c_T)} \int_{E\cap\,[0,\tau]}M(s)\,\mathrm{d}s + \int_{F\cap\,[0,\tau]}M_{\Qpazo}(s)\,\mathrm{d}s +  \int_{\tau}^t M_{\Qpazo}(s)\,\mathrm{d}s \\
& \le {2(1+c_T)} \int_{0}^tM(s)\,\mathrm{d}s + \int_{F^*\cap\, [0,t]} M_{\Qpazo}(s)\,\mathrm{d}s \\
& \le {2(1 +c_T)} \bigg( 1 + \int_{0}^t M(s)\,\mathrm{d}s \bigg). 
\end{aligned}
\end{equation*}
Lastly, the third estimate in \ref{hyp:PT}-$(ii)$ stems from the observation that 
\begin{equation*}
\begin{aligned}
W_1(\mu^*(\sigma) \, ; \Qpazo(\sigma)) & \le W_1(\mu^*(\sigma),\mu_{\sigma}) \\
& = W_1\Big((\pi^1+\pi^3-\pi^2)_{\sharp}\eta, \pi^{3}_{\sharp}\eta\Big)\\
&\le \| \pi^2-\pi^1\|_{L^1(\R^{3d},\R^d;\,\eta)} \\
& =W_{1}(\mu(\tau), \mu_\tau) \\
& \le \tau\exp\bigg(\int_{0}^{\tau}M(s)\,\mathrm{d}s\bigg)\varepsilon \\
& \le \sigma\exp\bigg(\int_{0}^{\sigma}M(s)\,\mathrm{d}s\bigg)\varepsilon,
\end{aligned}
\end{equation*}
where we used the fact that the initial triple $(\tau,\{[a_i,b_i)\}_{i \in \Lambda},\mu(\cdot))$ itself satisfies \ref{hyp:PT}-$(ii)$, and Proposition \ref{prop:Estimates}-$(c)$. This, along with all that precedes allows us to conclude that $(\tau,  \{[a_i,b_i)\}_{i\in\Lambda}, \mu(\cdot)) \preceq (\sigma, \{[a_i,b_i)\}_{i\in\Lambda^*}, \mu^*(\cdot))$, thereby contradicting our assumption that the former was maximal.
\end{proof}

\begin{proposition}[Maximal elements are globally defined when $\tau \in C_{\epsilon}$]
\label{prop:GlobalCeps}
For every $\varepsilon\in(0,\varepsilon_0)$, if the triple $(\tau,\{[a_i,b_i)\}_{i\in\Lambda},\mu(\cdot))$ is a maximal element in $(\Fpazo_\varepsilon, \preceq)$ and $\tau\in C_\varepsilon$, then $\tau = T$.
\end{proposition}
\begin{proof}
Suppose by contradiction that $\tau<T$. To begin with, fix $\epsilon \in (0,\epsilon_0)$ and recall that since the elements of $C_\varepsilon \subseteq [0,T]$ are Lebesgue points of $M(\cdot) \in L^1([0,T],\R_+)$ by definition, there exists $\delta\in(0,1)$ with $\delta\leq T-\tau$, for which 
\begin{equation}
\label{eq:DeltaChoice}
\int_{\tau}^{\tau+\delta} M(s)\,\mathrm{d}s \leq \frac{\varepsilon}{2\exp \Big( 2\INTSeg{M(s)}{s}{0}{T} \Big)}.
\end{equation}  
We also fix $\mu_\tau\in \Qpazo(\tau)$ so that $W_1(\mu(\tau),\mu_\tau)=W_1(\mu(\tau)\,;\Qpazo(\tau))$, which is possible since $\Qpazo(\tau)$ is proximal. Let $c>0$ be as defined above, and choose $h_0\in(0,\min\{\delta,\varepsilon,T-\tau\})$ in such a way that
\begin{equation}\label{kemtau}
\bigg( \frac1h \int_\tau^{\tau+h} M(s)\,\mathrm{d} s\bigg) \bigg( \int_\tau^{\tau+h} M(s)\, \mathrm{d} s \bigg) \le \frac{\varepsilon}{2c}
\end{equation}
for every $h\in(0,h_0)$. As a consequence of the viability condition \eqref{eq:IntegralViabilityCond} applied at $(\tau,\mu_\tau) \in\Graph(\Qpazo)$, there exists $w_\tau \in C^0(\R^d,\R^d)$ such that its $L^1(\R^d,\R^d;\mu_{\tau})$-equivalence class belongs to $D\Qpazo(\tau | \mu_\tau)$ and  
\begin{equation}
\label{eq:ViabCondProofHard}
w_\tau \in 
\Liminf{h\to0^+}\;\frac1h\int_\tau^{\tau+h} V\big(s,\B(\mu_\tau,\varepsilon/2)\big)\, \mathrm{d} s.
\end{equation}
In particular, one may find a vanishing sequence $\{h_i\}_{i \in \N} \subseteq (0,h_0)$ such that 
\begin{equation}
\label{eq:ViabCondProofHardBis}
\lim_{i \to +\infty} \frac{1}{h_i} W_1\Big((\Id + h_i w_\tau)_{\sharp} \mu_\tau \, ; \Qpazo(\tau+h_i)\Big) = 0, 
\end{equation}
along with a family $\{v_i(\cdot)\}_{i \in \N}$ of integrable selections $t \in [\tau,\tau+h_i] \mapsto v_i(t) \in V\big(t,\B(\mu_\tau,\varepsilon/2)\big)$ satisfying
\begin{equation*}
\Big\| \, w_{\tau} - \tfrac1{h_i} \mathsmaller{\int_\tau^{\tau+h_i}} v_i(s)\, \mathrm{d} s \, \Big\|_{L^1(\R^d,\R^d ; \, \mu_{\tau})} ~\underset{i \to +\infty}{\longrightarrow}~ 0,
\end{equation*}
see e.g. the argument detailed in proof of the first implication of Theorem \ref{Thm1} above. By the feasibility estimate at time $\tau$, $r_T\ge Te^{B_T}$ and $r_T\varepsilon\le1$, one has $\Mpazo_1(\mu_\tau)\le c_T+\tau e^{B_T}\varepsilon\le c_T+1$. Hence $\Mpazo_1(\nu)\le c_T+1+\varepsilon/2$ for $\nu\in\B(\mu_\tau,\varepsilon/2)$, and each selected $v_i$ satisfies $|v_i(t,x)|\le A_T M(t)(1+|x|)$ and $\Lip(v_i(t))\le M(t)$. We extend them by zero outside $[\tau,\tau+h_i]$ when applying Lemma \ref{lemcc} with $A_T M$ in place of $M$ and the constant $c$ fixed above. Therefore, there exists an index $i_1\in\mathbb N$ such that
\begin{equation*}
W_1 \bigg( \Big(\Id+ \mathsmaller{\int_\tau^{\tau+h_i}} v_i(s)\, \mathrm{d}s \Big)_{\hspace{-0.025cm} \raisebox{2pt}{$\scriptstyle\sharp$}} \mu_\tau,(\Id+h_i w_\tau)_{\sharp} \mu_\tau \bigg) \le \frac{\varepsilon}{4}\,h_i,
\end{equation*}
for all $i \geq i_1$. Moreover, it also follows from \eqref{eq:ViabCondProofHardBis} that there is some $i_0 \geq i_1$ such that 
\begin{equation*}
W_1 \Big((\Id+h_i w_\tau)_{\sharp}\mu_\tau \, ; \Qpazo(\tau+h_i) \Big)\le \frac{\varepsilon}{4}\,h_i 
\end{equation*}
whenever $i \geq i_0$, at which point, one may merge the two latter bounds to deduce that
\begin{equation}\label{ll1}
W_1\bigg(\Big(\Id+\mathsmaller{\int_{\tau}^{\tau+h_{i_0}} v_{i_0}(s)\,\mathrm{d} s}\Big)_{\hspace{-0.05cm} \raisebox{2pt}{$\scriptstyle\sharp$}} \mu_\tau\;;\;\Qpazo(\tau+h_{i_0})\bigg) 
\le \frac{\varepsilon}{2}\,h_{i_0}.
\end{equation}
At this stage, observe that following Lemma~\ref{lemcc}, the Cauchy problem
\begin{equation*}
\left\{
\begin{aligned}
& \partial_t \nu(t) + \Div_x(v_{i_0}(t) \nu(t)) = 0, \\
& \nu(\tau)= \mu_{\tau},
\end{aligned}
\right.
\end{equation*}
admits a unique solution $\nu(\cdot) \in \AC([\tau,\tau+h_{i_0}],\Pcal_1(\R^d))$ given explicitly by $\nu(t):=(\Phi^{v_{i_0}}_{(\tau,t)})_\sharp\mu_\tau$ for all times $t \in [\tau,\tau+h_{i_0}]$, which additionally satisfies
\begin{equation}\label{ll2}
\begin{aligned}
W_1\bigg(\nu(\tau+h_{i_0}) \,, \Big(\Id+\mathsmaller{\int_\tau^{\tau+h_{i_0}} v_{i_0}(s)\, \mathrm{d}s}\Big)_{\hspace{-0.025cm} \raisebox{2pt}{\hspace{-0.125cm} $\scriptstyle\sharp$}} \,\mu_\tau\bigg)
& \le c \bigg(\frac1{h_{i_0}} \int_\tau^{\tau+h_{i_0}} M(s)\, \mathrm{d}s\bigg)\bigg(\int_\tau^{\tau+h_{i_0}} M(s)\, \mathrm{d}s\bigg) h_{i_0} \\
& \le \frac{\varepsilon}{2}\,h_{i_0},
\end{aligned} 
\end{equation}
where the last inequality follows from \eqref{kemtau}. Combining both
\eqref{ll1} and \eqref{ll2}, we obtain
\begin{equation}\label{l3}
W_1 \Big(\nu(\tau + h_{i_0}) \, ; \Qpazo(\tau + h_{i_0}) \Big) \le \varepsilon\,h_{i_0}.
\end{equation}
Set now $\sigma := \tau+h_{i_0}$, let $\Lambda^*$ be the new index set for which $\{[a_i,b_i)\}_{i \in \Lambda^*} := \{[a_i,b_i)\}_{i \in \Lambda} \cup\{[\tau,\sigma)\}$, and define the extended curve $\mu^*:[0,\sigma]\to\mathcal P_1(\R^d)$ by
\begin{equation*}
\mu^*(t):=\begin{cases}
\mu(t) & \text{if } t\in[0,\tau],\\[2pt]
(\Phi^{v_{i_0}}_{(\tau,t)})_\sharp\,\mu(\tau) & \text{if } t\in[\tau,\sigma].
\end{cases}
\end{equation*}
It is then straightforward to see that the triple $(\sigma,\big\{[a_i,b_i)\big\}_{i\in\Lambda^*},\mu^*(\cdot))$ satisfies \ref{hyp:PT}-$(i)$. Moreover, \ref{hyp:PT}-$(iii)$ is inherited from the original triple on every interval $[a_i,b_i)$ with $i\in\Lambda$, since $\mu^*(t)=\mu(t)$ for all $t\in[0,\tau]$. Because $\tau\in C_\varepsilon$, it therefore remains to verify this property on the newly added interval $[\tau,\sigma)$. To begin with, observe that
\begin{equation*}
\begin{aligned}
\derv{}{t} \INTDom{\varphi(t,x)}{\R^d}{\mu^*(t)(x)}
& = \INTDom{\derv{}{t} \varphi \Big( t, \Phi_{(\tau,t)}^{v_{i_0}}(x) \Big)}{\R^d}{\mu(\tau)(x)} \\
& = \INTDom{\bigg( \partial_t \varphi \Big( t, \Phi_{(\tau,t)}^{v_{i_0}}(x) \Big) + \Big\langle \nabla_x \varphi \Big( t, \Phi_{(\tau,t)}^{v_{i_0}}(x) \Big) , v_{i_0} \Big( t,\Phi_{(\tau,t)}^{v_{i_0}}(x) \Big) \Big\rangle \bigg)}{\R^d}{\mu(\tau)(x)} \\
& = \INTDom{\Big( \partial_t \varphi(t,x) + \big\langle \nabla_x \varphi(t,x) , v_{i_0}(t,x) \big\rangle \Big)}{\R^d}{\mu^*(t)(x)}
\end{aligned}
\end{equation*}
for $\Lcal^1$-almost every $t \in [\tau,\sigma)$ and each $\varphi \in C^{\infty}_c((\tau,\sigma) \times \R^d)$. Hence, $\mu^*(\cdot)$ solves the continuity equation driven by $v_{i_0}(\cdot)$ on $[\tau,\sigma)$. The remaining requirements of \ref{hyp:PT}-$(iii)$ on this interval will follow from \eqref{eq:ExtendedVelInc} and \eqref{eq:ACContIncEst} below. At this point, we can leverage the moment inequality from \eqref{eq:BasicEstimates} and apply Gr\"onwall's lemma while recalling that $\Mpazo_1(\mu^*(\tau)) \leq c_T$ and $\epsilon \leq 1/r_T < 1$ to get the crude bound 
\begin{equation*}
\Mpazo_1(\mu^*(t)) \leq 2(1+c_T) \bigg( 1 + \INTSeg{M(s)}{s}{0}{T} \bigg) \exp \bigg( \INTSeg{M(s)}{s}{0}{T} \bigg)
\end{equation*}
for all times $t \in [\tau,\sigma]$. This, combined with the absolute continuity estimate in \eqref{eq:BasicEstimates}, further implies 
\begin{equation}
\label{eq:DistanceExtCurve}
\begin{aligned}
W_1(\mu(\tau),\mu^*(t)) & \leq \INTSeg{\NormL{v_{i_0}(s)}{1}{\R^d,\R^d;\, \mu^*(s)}}{s}{\tau}{t} \\
& \leq \INTSeg{M(s) \Big( 2 + \Mpazo_1(\mu^*(s)) + \Mpazo_1(\mu^*(\tau)) \Big) }{s}{\tau}{t} \\
& \leq 4(1+c_T) \bigg( 1 + \INTSeg{M(s)}{s}{0}{T} \bigg) \exp \bigg( 2\INTSeg{M(s)}{s}{0}{T} \bigg) \INTSeg{M(s)}{s}{\tau}{t} \\
& \leq 2(1+c_T) \bigg( 1 + \INTSeg{M(s)}{s}{0}{T} \bigg) \epsilon
\end{aligned}
\end{equation}
by our choice of $\delta > 0$ made in \eqref{eq:DeltaChoice}. Lastly, upon remarking that 
\begin{equation*}
\begin{aligned}
W_1(\mu_{\tau},\mu^*(t)) & \leq W_1(\mu_{\tau},\mu(\tau)) + W_1(\mu(\tau),\mu^*(t)) \\
& \leq \tau \exp \bigg( \INTSeg{M(s)}{s}{0}{\tau} \bigg) \epsilon + 2(1+c_T) \bigg( 1 + \INTSeg{M(s)}{s}{0}{T} \bigg) \epsilon \\
& \leq (r_T-1) \epsilon
\end{aligned}
\end{equation*}
as a consequence of \eqref{eq:DeltaChoice} and \eqref{eq:DistanceExtCurve}, we finally obtain
\begin{equation}
\label{eq:ExtendedVelInc}
v_{i_0}(t)\in V \big(t,\B( \mu_\tau, \epsilon)\big) \subseteq V\big(t \, , \mathbb B(\mu^*(t), r_T \varepsilon)\big).
\end{equation} 
To establish \ref{hyp:PT}-$(ii)$, note to begin with that $\mu^* : [0,\sigma] \to \Pcal_1(\R^d)$ is clearly continuous by construction. By repeating the moment estimates from the proof of Proposition \ref{prop:Estimates}-$(a)$ while using \eqref{eq:ExtendedVelInc} along with the fact that $ r_T\epsilon < 1$, one gets the refined bound
\begin{equation*}
\begin{aligned}
& \Mpazo_1(\mu^*(t)) \\
& \hspace{0.3cm} \leq \bigg( \Mpazo_1(\mu^*(\tau)) + 2 \INTSeg{M(s)}{s}{\tau}{t} \bigg) \exp \bigg( 2 \INTSeg{M(s)}{s}{\tau}{t} \bigg) \\
& \hspace{0.3cm} \leq \Bigg( 2 \bigg( 1 + \Mpazo_1(\mu_0) + \INTSeg{M(s)}{s}{0}{\tau} \bigg) \exp \bigg( 2 \INTSeg{M(s)}{s}{0}{\tau} \bigg) + 2 \INTSeg{M(s)}{s}{\tau}{t} \Bigg) \exp \bigg( 2 \INTSeg{M(s)}{s}{\tau}{t} \bigg) \\
& \hspace{0.3cm} \leq 2 \bigg( 1 + \Mpazo_1(\mu_0) + \INTSeg{M(s)}{s}{0}{t} \bigg) \exp \bigg( 2 \INTSeg{M(s)}{s}{0}{t} \bigg) \\
& \hspace{0.3cm} \leq c_T
\end{aligned}
\end{equation*}
for all times $t \in [\tau,\sigma]$. Regarding now the second estimate in \ref{hyp:PT}-$(ii)$, and the remaining estimate in \ref{hyp:PT}-$(iii)$ on the newly added interval, it follows from \eqref{eq:BasicEstimates} that
\begin{equation}
\label{eq:ACContIncEst}
W_1(\mu^*(s),\mu^*(t)) \leq \INTSeg{\NormL{v_{i_0}(\theta)}{1}{\R^d,\R^d ; \, \mu^*(\theta)}}{\theta}{s}{t} \leq 2(1+c_T) \INTSeg{M(\theta)}{\theta}{s}{t}
\end{equation}
for all times $\tau \leq s \leq t \leq \sigma$, where we used the fact that $\Mpazo_1(\mu^*(t)) \leq c_T$. Together with the continuity equation established above and \eqref{eq:ExtendedVelInc}, estimate \eqref{eq:ACContIncEst} proves that the augmented triple satisfies \ref{hyp:PT}-$(iii)$ on the newly added interval $[\tau,\sigma)$, and hence on every interval indexed by $\Lambda^*$. Taking $s=\tau$ in \eqref{eq:ACContIncEst} also allows us to deduce 
\begin{equation*}
\begin{aligned}
W_1(\mu_0,\mu^*(t)) & \leq W_1(\mu_0,\mu(\tau)) + W_1(\mu(\tau),\mu^*(t)) \\
& \leq {2(1+ c_T)} \bigg( 1+ \INTSeg{M(s)}{s}{0}{\tau} \bigg) + {2(1+c_T)} \INTSeg{M(s)}{s}{\tau}{t} \\
& \leq {2(1+ c_T)} \bigg( 1+ \INTSeg{M(s)}{s}{0}{t} \bigg) 
\end{aligned}
\end{equation*}
for all times $t \in [\tau,\sigma]$. Concerning the third estimate in \ref{hyp:PT}-$(ii)$, one may note that
\begin{equation*}
\begin{aligned}
W_1(\mu^*(\sigma) \, ; \Qpazo(\sigma)) & = W_1\Big(\mu^*(\tau + h_{i_0}) \, ; \Qpazo(\tau + h_{i_0}) \Big) \\
& \le	W_1\Big(\mu^*(\tau + h_{i_0}), \nu(\tau + h_{i_0})\Big) + W_1 \Big(\nu(\tau + h_{i_0}) \, ; \Qpazo(\tau + h_{i_0}) \Big) \\
& \le \exp\bigg(\int_{\tau}^{\tau + h_{i_0}} M(s)\,\mathrm{d}s\bigg) W_{1}(\mu(\tau),\mu_\tau) + \varepsilon h_{i_0} \\
& \le \tau \exp\bigg(\int_{0}^{\tau+h_{i_0}} M(s)\,\mathrm{d}s\bigg) \varepsilon +  h_{i_0}\exp \bigg(\int_{0}^{\tau+h_{i_0}} M (s)\,\mathrm{d}s\bigg) \varepsilon \\
& = \sigma \exp\bigg(\int_{0}^{\sigma} M(s) \,\mathrm{d}s\bigg) \varepsilon
\end{aligned}
\end{equation*}
where we used (\ref{l3}) and the flow estimate $\Lip(\Phi_{(\tau,t)}^{v_{i_0}}) \leq \exp\big( \INTSeg{M(s)}{s}{\tau}{t} \big)$ which holds for all times $t \in [\tau,\sigma]$ (see e.g. \cite[Chapter 16]{AmbrosioBS2021}), along with Proposition \ref{prop:Estimates}-$(c)$. In conclusion, we have shown that $(\tau, \{[a_i,b_i)\}_{i\in \Lambda}, \mu(\cdot))\preceq(\sigma,\{[a_i,b_i)\}_{i\in \Lambda^*}, \mu^*(\cdot))$, which contradicts our initial assumption that the former was maximal. 
\end{proof}

Upon combining the results of Proposition \ref{prop:GlobalNotCeps} and Proposition \ref{prop:GlobalCeps}, we may deduce the following. 

\begin{corollary}[Maximal elements are globally defined for every $\tau$]
\label{cor:max}
For every $\varepsilon\in(0,\varepsilon_0)$, if the triple $(\tau,\{[a_i,b_i)\}_{i\in\Lambda},\mu(\cdot))$ is a maximal element in $(\Fpazo_\varepsilon, \preceq)$, then $\tau = T$.
\end{corollary}


\subsubsection*{Existence of maximal elements}

In what precedes, we have shown that maximal elements in $\Fpazo_\varepsilon$ are of special importance. Indeed, they represent globally defined curves which approximately solve the continuity inclusion, while lying quantitatively close to the images of the constraints mapping $\Qpazo : [0,T] \rightrightarrows \Pcal_1(\R^d)$ at all times. We will now prove that such maximal elements do exist, by employing the celebrated Zorn's lemma. 

\begin{theorem}[Zorn's lemma]
Let $(\Fpazo,\preceq)$ be a partially ordered set. Suppose that every chain, that is every totally ordered subset, has an upper bound in $\Fpazo$. Then $(\Fpazo,\preceq)$ contains at least one maximal element.
\end{theorem}

\begin{proposition}[Existence of a maximal triple]
\label{exismax}
For every $\varepsilon \in (0,\epsilon_0)$, the family $\Fpazo_\varepsilon$ has at least one maximal element for the partial order $\preceq$. 
\end{proposition}

\begin{proof}
Let $\lbrace (\tau_\alpha, \big\{[a^\alpha_i,b^\alpha_i)\}_{i\in \Lambda_\alpha}, \mu_\alpha)\rbrace_{\alpha\in \mathcal I}$ be an arbitrary chain in $\Fpazo_\varepsilon$, and let us prove that it has an upper bound. To this end, set $\tau:= \sup_{\alpha\in\mathcal I}\tau_\alpha$ and note that, since the set  $\bigcup_{\alpha\in\mathcal I}\{[a_i^\alpha,b_{i}^\alpha)\}_{i\in{\Lambda_\alpha}}$ is equal to the union of pairwise disjoint intervals with nonempty interior, we can find a countable set $\Lambda$ and a family of intervals $\{[a_i,b_i)\}_{i\in\Lambda}$ such that 
\begin{equation*}
\bigcup_{\alpha\in\mathcal{I}}\{[a_i^\alpha,b_{i}^\alpha)\}_{i\in\Lambda_\alpha}=\{[a_i,b_i)\}_{i\in\Lambda}.
\end{equation*}
The family on the left-hand side is at most countable, since each of its intervals has nonempty interior and therefore contains a rational number, while pairwise disjointness allows these rational numbers to be chosen distinctly.
Let now $\mu:[0,\tau)\to\mathcal P_1(\mathbb R^d)$ be the curve given by $\mu(t) = \mu_\alpha(t)$ if $t\in[0,\tau_\alpha]$, which we recall is licit since
\begin{equation*}
\mu_{\alpha}(t) = \mu_{\beta}(t)  \qquad
\end{equation*}
whenever $t \in [0,\min\{\tau_{\alpha},\tau_{\beta}\}]$ for any given $\alpha,\beta \in \mathcal{I}$, by definition \eqref{eq:PartialOrder1}-\eqref{eq:PartialOrder2} of the partial order $\preceq$. To prove that $(\tau,\{[a_i,b_i)\}_{i\in\Lambda}, \mu(\cdot))$ is an upper bound of the chain, it is enough to show that $\lim_{t\to \tau^-}\mu(t)$ exists, as one may then extend $\mu : [0,\tau] \to \Pcal_1(\R^d)$ continuously, and from there infer all the requirements of \ref{hyp:PT}. Let $\{t_n\}_{n\in\mathbb N}\subseteq[0,\tau)$ be an arbitrary sequence such that $t_n\to \tau^-$, and note that by Proposition \ref{prop:Estimates}-$(b)$, one has that  
\begin{equation*}
W_1(\mu(t_n),\mu(t_{n+p})) \leq \bigg| \INTSeg{\Big( 2(1+c_T)M(\theta) + M_{\Qpazo}(\theta) \Big)}{\theta}{t_n}{t_{n+p}} \, \bigg| ~\underset{n,p \to +\infty}{\longrightarrow}~ 0.  
\end{equation*}
Whence, the sequence $\{\mu(t_n)\}_{n\in\mathbb N}$ is Cauchy and therefore admits a limit by the completeness of $(\mathcal P_1(\mathbb R^d),W_1(\cdot,\cdot))$. It is possible to verify upon using the very same estimate that such a limit is independent of the choice of the sequence $\{t_n\}_{n\in\mathbb N}$, and we may thus extend $\mu : [0,\tau] \to \Pcal_1(\R^d)$ continuously, so that 
\begin{equation*}
(\tau_\alpha, \big\{[a^\alpha_i,b^\alpha_i)\}_{i\in \Lambda_\alpha}, \mu_\alpha) \preceq (\tau,\{[a_i,b_i)\}_{i\in\Lambda}, \mu(\cdot))
\end{equation*}
for every $\alpha \in \mathcal{I}$. The existence of a maximal element in $(\Fpazo_{\epsilon},\preceq)$ follows then from Zorn's lemma.
\end{proof}

Synthesising all the above results, we have the following fact that will play a crucial role in the ensuing proof of Theorem \ref{Thm4}.  

\begin{corollary}[On maximal elements in $\Fpazo_{\epsilon}$]
\label{corapp}
For every $\varepsilon\in(0,\varepsilon_0)$, there exists a family of intervals $\{[a_i,b_i)\}_{i\in\Lambda}$ and a curve $\mu(\cdot) \in \AC([0,T],\Pcal_1(\mathbb R^d))$ such that $(T,\{[a_i,b_i)\}_{i\in\Lambda},\mu(\cdot)) \in \Fpazo_\varepsilon$.
\end{corollary}

\begin{proof}
By Proposition \ref{prop:nonemptiness}, the family $\Fpazo_\varepsilon$ is nonempty and due to  Proposition \ref{exismax} there exists a maximal element $(\tau,\{[a_i,b_i)\}_{i\in\Lambda},\mu(\cdot))$ in $(\Fpazo_\varepsilon,\preceq)$. By Corollary \ref{cor:max}, it must be so that $\tau = T$, whereas the absolute continuity of $\mu(\cdot)$ follows directly from \ref{hyp:PT}-$(ii)$. 
\end{proof}


\subsubsection*{A technical lemma on approximate solutions}	

In this last preparatory subsection, we prove the existence of a sequence of curves satisfying a list of relevant technical properties, which stem from them being maximal elements in $(\Fpazo_{\epsilon},\preceq)$ for adequate choices of $\epsilon \in (0,\epsilon_0)$. 

\begin{lemma}[A good sequence of approximate solutions]
\label{techlem}
Suppose that the assumptions of Theorem \ref{Thm4} hold. Then for any $\mu_0\in \Qpazo(0)$, there exists a sequence of absolutely continuous curves $\{\mu_n(\cdot)\}_{n\in\mathbb N} \subseteq \AC([0,T], \Pcal_1(\mathbb R^d))$ satisfying $\mu_n(0) = \mu_0$ for each $n \in \N$, along with the following. 
\begin{itemize}
\item[$(a)$] For each $n\in\mathbb N$, there exist $\epsilon_n \in(0,1)$ and an open set $\Opazo_n \subseteq [0,T]$ {of the form
\begin{equation*}
\Opazo_n = \bigcup_{a_i^n \in C_{\epsilon_n}} (a_i^n,b_i^n)
\end{equation*}}
such that $\Lcal^1([0,T]\setminus\Opazo_n)\le 1/n$. Moreover, the curve $\mu_n(\cdot) \in \AC([0,T],\Pcal_1(\R^d))$ solves the continuity inclusion
\begin{equation*}
\partial_t \mu_n(t) \in -\Div_x \Big( V \big(t,\mathbb B(\mu_n(t), \tfrac{1}{n})\big)\, \mu_n(t)\Big)
\end{equation*}
over $\Opazo_n$, namely there exists an integrable selection $t \in [0,T] \mapsto v_n(t) \in V\big(t,\B(\mu_n(t),1/n)\big)$ such that
\begin{equation*}
\int_0^T \int_{\mathbb R^d}\Big( \partial_t\varphi(t,x) + \langle \nabla_x \varphi(t,x) , v_n(t,x) \rangle\Big) \mathrm{d}\mu_n(t)(x)\,\mathrm{d}t = 0
\end{equation*}
for every $\varphi\in C_c^\infty(\Opazo_n \times\mathbb R^d)$.
\item[$(b)$] The family $\{\mu_n(\cdot)\}_{n\in\mathbb N}\subseteq \AC([0,T],\mathcal P_1(\mathbb R^d))$ complies with the uniform moment and absolute continuity estimates
\begin{equation*}
\quad \quad \quad \sup_{n \in \N} \Mpazo_1(\mu_n(t)) \leq c_T \quad \text{and} \quad \sup_{n\in\mathbb N}W_1(\mu_n(s),\mu_n(t)) \le \int_{s}^t\max\Big\{2(1+c_T) M(\theta),M_{\Qpazo}(\theta)\Big\}\,\mathrm{d}\theta
\end{equation*}
for all times $0\le s\le t\le T$, where $c_T > 0$ is given as in \eqref{eq:cTDef}. 
\item[$(c)$] It holds that 
\begin{equation*}
W_1(\mu_n(t) \,; \Qpazo(t)) ~\underset{n \to +\infty}{\longrightarrow} 0
\end{equation*}
for all times $t\in[0, T]$. 	
\end{itemize}
\end{lemma}
\begin{proof}
Let $E_1 \subseteq [0,T]$ be the set of Lebesgue points of $M(\cdot) \in L^1([0,T],\mathbb R_+)$, let further $E_2 \subseteq [0,T]$ be the set of points where \eqref{eq:IntegralViabilityCond} is satisfied, and set $E:= E_1\cap E_2$. Fix also $\varepsilon_0>0$ for which (\ref{absint}) holds, and choose $N\in\mathbb N$ large enough so that $1/(n r_T) < \varepsilon_0$ for all $n\ge N$,  where $r_T >0$ is given in \eqref{r}. Moreover, let $\delta_n > 0$ be chosen in such a way that 
\begin{equation*}
\INTDom{\max\Big\{2(1+c_T) M(\theta) , M_{\Qpazo}(\theta) \Big\}}{A}{\theta} \leq \frac{1}{n} \qquad \text{whenever} \qquad \Lcal^1(A) \leq \delta_n,  
\end{equation*}
and set then $\varepsilon_n := \min\{1/(r_T n) , \delta_n\}$ for $n\ge N$. By the inner regularity of the Lebesgue measure, there exists an increasing sequence of compact sets $\{C_{\epsilon_n}\}_{n\ge N}$ such that $C_{\epsilon_n} \subseteq E$ and 
\begin{equation*}
\Lcal^1([0,T] \setminus C_{\epsilon_n}) \le \epsilon_n \leq \frac{1}{n}
\end{equation*}
for each $n \ge N$. Then, by Corollary \ref{corapp} above there exists for each $n\ge N$ a countable family of intervals $\{[a^n_i,b^n_i)\}_{i\in\Lambda_n}$ and a curve $\mu_n(\cdot) \in \AC([0,T],\Pcal_1(\R^d))$ such that $(T,\{[a^n_i,b^n_i)\}_{i\in\Lambda_n},\mu_n(\cdot)) \in \Fpazo_{\varepsilon_n}$. Consider now the open set $\Opazo_n \subseteq [0,T]$ given by
\begin{equation*}
\Opazo_n:= \bigcup_{a_i^n\in C_{\varepsilon_n}} (a^n_i,b^n_i), 
\end{equation*}
which by \ref{hyp:PT}-$(iv)$ can be shown to satisfy
\begin{equation*}
\begin{aligned}
\Lcal^1([0,T]\setminus\Opazo_n) & = \Lcal^1 \bigg( \bigcup_{a_i^n\notin C_{\varepsilon_n}} [a^n_i,b^n_i] \bigg) \\
& \leq \Lcal^1([0,T]\setminus  C_{\varepsilon_n}) \le \frac{1}{n}
\end{aligned}
\end{equation*}
for all $n\ge N$. At this stage, it can be checked that the claims made in $(b)$ and $(c)$ follow directly from Proposition \ref{prop:Estimates} and Corollary \ref{corclsvia}, respectively. Regarding $(a)$, we know from the properties \ref{hyp:PT} satisfied by each of the triples $(T,\{[a^n_i,b^n_i)\}_{i\in\Lambda_n},\mu_n(\cdot))$ that for every $i \in \Lambda_n$ with 
$a_i^n \in C_{\epsilon_n}$, there exists an integrable selection $t \in [a_i^n,b_i^n] \mapsto v_n^i(t) \in V\big(t,\B(\mu_n(t),1/n)\big)$ such that 
\begin{equation*}
\int_{a_i^n}^{b_n^i}\int_{\mathbb R^d}\Big( \partial_t\varphi(t,x) + \langle \nabla_x \varphi(t,x) , v_n^i(t,x) \rangle\Big) \mathrm{d}\mu_n(t)(x)\,\mathrm{d}t = 0
\end{equation*}
holds for each $\varphi\in C^{\infty}_c((a_i^n,b_i^n) \times \R^d)$. For $i\in \Lambda_n$ with $a_{i}^n\notin C_{\varepsilon_n}$, we may leverage Corollary \ref{cors} to get the existence of an integrable selection $t \in [a_i^n, b_i^n] \mapsto v^i_n(t) \in V(t,\mu_n(t))$, and it can then be verified that the velocity maps defined by 
\begin{equation*}
v_n(t) := \sum_{i \in \Lambda_n} \mathds{1}_{[a_i^n,b_i^n)}(t) \, v_n^i(t)
\end{equation*}
for $\Lcal^1$-almost every $t \in [0,T]$ and each $n \geq N$ satisfy the requirements of item $(a)$. For $n<N$, we let $\varepsilon_n:=\varepsilon_N$, $\{[a_i^n,b_i^n)\}_{i\in\Lambda_n}:=\{[a_i^N,b_i^N)\}_{i\in\Lambda_N}$, $\mu_n(\cdot):=\mu_N(\cdot)$, $\Opazo_n:=\Opazo_N$ and $v_n:=v_N$, and note that the sequence $\{\mu_n(\cdot)\}_{n \in \N} \subseteq\AC([0,T],\Pcal_1(\R^d))$ clearly verifies items $(a)$, $(b)$ and $(c)$ by construction.
\end{proof}


\subsection{Proof of Theorem \ref{Thm4}}\label{ssusc}

Let $\{\mu_n(\cdot)\}_{n\in\mathbb N} \subseteq \AC([0,T],\Pcal_1(\R^d))$ be a sequence of curves satisfying the conditions of Lemma \ref{techlem} above. We are going to show that the latter admits a uniformly converging subsequence, whose limit solves the continuity inclusion \eqref{CIUC} and is viable for $\Qpazo : [0,T] \rightrightarrows \Pcal_1(\R^d)$. 

\smallskip

\noindent \textbf{Step 1.} (Existence of a uniformly converging subsequence). To begin with, observe that by Lemma \ref{techlem}-$(a)$, there exists for every $n \geq 1$ a function $\tilde{\mu}_n : [0,T] \to \Pcal_1(\R^d)$, which need not be measurable a priori, and an $\Lcal^1$-measurable map $v_n : [0,T] \to C^0(\R^d,\R^d)$ such that
\begin{equation*}
W_1(\tilde{\mu}_n(t),\mu_n(t)) \leq \frac{1}{n} \qquad\text{and}\qquad v_n(t)\in V(t,\tilde{\mu}_n(t)) 
\end{equation*}
for $\Lcal^1$-almost every $t \in [0,T]$.
In what follows, we let $E_n := [0,T] \setminus \Opazo_n$ and consider the truncated velocities given by
\begin{equation*}
w_n(t,x) := \mathds{1}_{\Opazo_n}(t) \, v_n(t,x)
\end{equation*}
for $\Lcal^1$-almost every $t \in [0,T]$ and all $x \in \R^d$. Then, let $\hat{\mu}_n(\cdot) \in \AC([0,T],\Pcal_1(\R^d))$ be the unique solution of the Cauchy problem
\begin{equation*}
\left\{
\begin{aligned}
& \partial_t \hat{\mu}_n(t) + \Div_x(w_n(t) \hat{\mu}_n(t)) = 0, \\
& \hat{\mu}_n(0) = \mu_0. 
\end{aligned}
\right.
\end{equation*}
Our goal in what ensues is to show that the sequence $(\hat{\mu}_n(\cdot))_{n \in \N} \subset \AC([0,T],\Pcal_1(\R^d))$ is relatively compact. To this end, note first that by Lemma \ref{techlem}-$(b)$ together with the fact that $W_1(\tilde{\mu}_n(t),\mu_n(t)) \leq 1/n \leq 1$, one has $\Mpazo_1(\tilde{\mu}_n(t)) \leq 1 + c_T$ for $\Lcal^1$-almost every $t \in [0,T]$, so that Hypothesis \ref{hyp:USC}-$(ii)$ yields 
\begin{equation*}
|w_n(t,x)| \leq M(t) \Big( 2 + c_T + |x| \Big) \qquad \text{and} \qquad \Lip(w_n(t)) \leq M(t)
\end{equation*}
for $\Lcal^1$-almost every $t \in [0,T]$, all $x \in \R^d$ and each $n \in \N$. 
Repeating the arguments detailed in the proof of Lemma \ref{lemcc}, we infer that each $w_n(\cdot)$ generates a characteristic flow $(\Phi^{w_n}_{(0,t)})_{t \in [0,T]} \subseteq C^0(\R^d,\R^d)$ such that $\hat{\mu}_n(t) = (\Phi^{w_n}_{(0,t)})_{\sharp}\mu_0$ by \eqref{eq:FlowRepresentation}, and which complies with the estimate
\begin{equation*}
\big| \Phi^{w_n}_{(0,t)}(x) \big| \leq \bar{M}(1 + |x|) \qquad \text{where} \qquad \bar{M} := \bigg(1+(2+c_T) \INTSeg{M(s)}{s}{0}{T}\bigg) \exp \bigg(  \INTSeg{M(s)}{s}{0}{T} \bigg) 
\end{equation*}
for all $(t,x) \in [0,T] \times \R^d$ and every $n \in \N$. This implies in particular that 
\begin{equation*}
\Mpazo_1(\hat{\mu}_n(t)) \leq \bar{M}(1 + \Mpazo_1(\mu_0))    
\end{equation*}
for all times $t \in [0,T]$, whereas for each $R > \bar{M}$ the following inclusion 
\begin{equation*}
\Big\lbrace x \in \R^d ~\,\textnormal{s.t.}~ \big| \Phi^{w_n}_{(0,t)}(x) \big| \geq R \Big\rbrace \subseteq \Big\lbrace x \in \R^d ~\,\textnormal{s.t.}~ |x| \geq \tfrac{R}{\bar{M}} - 1 \Big\rbrace
\end{equation*}
holds. Whence we deduce that 
\begin{equation*}
\INTDom{|y|}{\{y \,\textnormal{s.t.}\, |y| \geq R\}}{\hat{\mu}_n(t)(y)} \leq \INTDom{\bar{M} \big( 1 + |x| \big)}{\{x \,\textnormal{s.t.}\, |x| \geq R/\bar{M} - 1 \}}{\mu_0(x)} ~\underset{R \to +\infty}{\longrightarrow}~ 0, 
\end{equation*}
uniformly with respect to $n \in \N$ and $t \in [0,T]$ since $\mu_0 \in \Pcal_1(\R^d)$. Hence, the set $\{\hat{\mu}_n(t) ~\,\textnormal{s.t.}~ n \in \N ~\text{and}~ t \in [0,T]\}$ is relatively compact in $(\Pcal_1(\R^d),W_1(\cdot,\cdot))$ by the characterisation \eqref{eq:RelativeCompactness}. Recalling that $w_n(t) = 0$ for $\Lcal^1$-almost every $t \in E_n$, the absolute continuity estimate of \eqref{eq:BasicEstimates} further gives
\begin{equation*}
W_1(\hat{\mu}_n(s),\hat{\mu}_n(t)) \leq \INTDom{\NormL{w_n(\theta)}{1}{\R^d,\R^d ; \, \hat{\mu}_n(\theta)}}{[s,t] \cap \Opazo_n}{\theta} \leq \kappa_T \INTDom{M(\theta)}{[s,t] \cap \Opazo_n}{\theta}
\end{equation*}
for all times $0 \leq s \leq t \leq T$ and each $n \in \N$, where we set 
$\kappa_T := 2 + c_T + \bar{M}(1+\Mpazo_1(\mu_0))$. The curves $\{\hat{\mu}_n(\cdot)\}_{n \in \N}$ being thus uniformly equicontinuous, it follows from the metric version of the Ascoli-Arzel\`a theorem (see e.g. \cite[Chapter 7 -- Theorem 17]{Kelley1975}) that there exist an increasing map $\bar{\sigma} : \N \to \N$ and a curve $\mu(\cdot) \in C^0([0,T],\Pcal_1(\R^d))$ such that $\{\hat{\mu}_{\bar{\sigma}_n}(\cdot)\}_{n \in \N}$ converges uniformly towards $\mu(\cdot)$, where we write $\bar{\sigma}_n := \bar{\sigma}(n)$.

We now prove that the sequence $(\mu_{\bar{\sigma}_n}(\cdot))_{n \in \N}$ also converges uniformly to the same limit. To this end, we consider the function 
\begin{equation*}
d_n(t) := W_1(\mu_n(t),\hat{\mu}_n(t))
\end{equation*}
defined for all times $t \in [0,T]$, and observe that $d_n(\cdot)$ is absolutely continuous with $d_n(0) = 0$, being the distance between two absolutely continuous curves. If $(a^n_i,b^n_i)$ is a connected component of $\Opazo_n$, both $\mu_n(\cdot)$ and $\hat{\mu}_n(\cdot)$ solve therein the continuity equation driven by $v_n(\cdot)$, so that by the flow representation \eqref{eq:FlowRepresentation} and the second estimate in \eqref{eq:WassEst}, one gets
\begin{equation*}
d_n(t) \leq \Lip \big( \Phi^{v_n}_{(s,t)} \big) \, d_n(s) \leq \exp \bigg( \INTSeg{M(\theta)}{\theta}{s}{t} \bigg) d_n(s)
\end{equation*}
for all times $a^n_i \leq s \leq t \leq b^n_i$. Multiplying the latter inequality by $\exp\big(-\int_0^t M(\theta)\,\mathrm{d}\theta\big)$ yields
\begin{equation*}
\exp\bigg(-\int_0^t M(\theta)\,\mathrm{d}\theta\bigg)d_n(t)
\leq
\exp\bigg(-\int_0^s M(\theta)\,\mathrm{d}\theta\bigg)d_n(s),
\end{equation*}
so that the map $t\mapsto \exp\big(-\int_0^t M(\theta)\,\mathrm{d}\theta\big)d_n(t)$ is nonincreasing on $(a_i^n,b_i^n)$. Since $d_n(\cdot)$ is absolutely continuous and $M(\cdot)\in L^1([0,T],\R_+)$, this map is also absolutely continuous. By differentiating almost everywhere, we therefore obtain
\begin{equation*}
\exp\bigg(-\int_0^t M(\theta)\,\mathrm{d}\theta\bigg)
\Big(d_n'(t)-M(t)d_n(t)\Big)\leq 0,
\end{equation*}
hence $d_n'(t) \leq M(t)d_n(t)$ for $\Lcal^1$-almost every $t \in \Opazo_n$.
On the other hand, it stems from iterated triangle inequalities combined with Lemma \ref{techlem}-$(b)$ and our previous estimates on $\hat{\mu}_n(\cdot)$ that
\begin{equation*}
\begin{aligned}
d_n(t) - d_n(s) & \leq W_1(\mu_n(s),\mu_n(t)) + W_1(\hat{\mu}_n(s),\hat{\mu}_n(t)) \\
& \leq \INTSeg{\max\Big\{2(1+c_T)M(\theta) , M_{\Qpazo}(\theta)\Big\}}{\theta}{s}{t} + \kappa_T \INTDom{M(\theta)}{[s,t] \cap \Opazo_n}{\theta}
\end{aligned}
\end{equation*}
for all times $0 \leq s \leq t \leq T$, 
so that $d_n'(t) \leq \max\{2(1+c_T)M(t) , M_{\Qpazo}(t)\}$ for $\Lcal^1$-almost every $t \in E_n$. Merging both inequalities, we finally obtain 
\begin{equation*}
d_n'(t) \leq M(t) d_n(t) + \mathds{1}_{E_n}(t) \max\Big\{2(1+c_T)M(t) , M_{\Qpazo}(t) \Big\}
\end{equation*}
for $\Lcal^1$-almost every $t \in [0,T]$, and Gr\"onwall's lemma along with $d_n(0)=0$ then yields
\begin{equation*}
\sup_{t \in [0,T]} W_1(\mu_n(t),\hat{\mu}_n(t)) \leq \bigg( \INTDom{\max\Big\{2(1+c_T)M(\theta) , M_{\Qpazo}(\theta)\Big\}}{E_n}{\theta} \bigg) \exp \bigg( \INTSeg{M(t)}{t}{0}{T} \bigg) ~\underset{n \to +\infty}{\longrightarrow}~ 0, 
\end{equation*}
the convergence being a consequence of the absolute continuity of the Lebesgue integral, since $\Lcal^1(E_n) \leq 1/n$ by Lemma \ref{techlem}-$(a)$. Combining this with the uniform convergence of $\{\hat{\mu}_{\bar{\sigma}_n}(\cdot)\}_{n \in \N}$ towards $\mu(\cdot)$, we conclude that
\begin{equation}
\label{eq:Unieq:IntegralViabilityCondvCurve}
\sup_{t \in [0,T]} W_1(\mu_{\bar{\sigma}_n}(t),\mu(t)) ~\underset{n \to +\infty}{\longrightarrow}~ 0.
\end{equation}
In what follows to lighten the exposition, we will drop all explicit dependence on this subsequence, and simply write $n$ in place of $\bar{\sigma}_n$. 

\medskip

\noindent \textbf{Step 2.} (The limit curve  solves a continuity equation). To begin with, consider the maps $\tilde{\mu}_n(\cdot)$ and $v_n(\cdot)$ introduced in Step 1. Recall in addition that the distributional identity 
\begin{equation}
\label{eq:LimitCE0}
\INTSeg{\INTDom{\Big( \partial_t \varphi_n(t,x) + \langle \nabla_x \varphi_n(t,x) , v_n(t,x) \rangle \Big)}{\R^d}{\mu_n(t)(x)}}{t}{0}{T} = 0
\end{equation}
holds for each $n \geq 1$ and every $\varphi_n \in C^{\infty}_c(\Opazo_n \times \R^d)$. To prove that the limit curve $\mu(\cdot) \in \AC([0,T],\Pcal_1(\R^d))$ solves a continuity equation, let us consider test functions $\varphi\in C_c^\infty((0,T)\times\mathbb R^d)$ of the specific form
\begin{equation}
\label{eq:TestFunction}   
\varphi(t,x) := \xi(t)\psi(x)
\end{equation}
for some $(\xi,\psi) \in C^\infty_c((0,T)) \times C_c^\infty(\mathbb R^d)$ and all $(t,x) \in [0,T] \times \R^d$. Then, observe that under Hypothesis \ref{hyp:USC}-$(ii)$ and the uniform moment bound from Lemma \ref{techlem}-$(b)$, together with the fact that $\Mpazo_1(\tilde{\mu}_n(t)) \leq \Mpazo_1(\mu_n(t)) + W_1(\tilde{\mu}_n(t),\mu_n(t)) \leq c_T + 1/n$, we have 
\begin{equation}
\label{eq:LimitCE1}
\begin{aligned}
& \bigg| \INTDom{\INTDom{\Big( \xi'(t) \psi(x) + \langle \xi(t) \nabla \psi(x) , v_n(t,x) \rangle \Big)}{\R^d}{\mu_n(t)(x)}}{[0,T] \setminus \Opazo_n}{t} \, \bigg| \\
& \hspace{1cm} \leq \NormC{\xi}{1}{(0,T)} \NormC{\psi}{1}{\R^d} \bigg( \Lcal^1([0,T] \setminus \Opazo_n) + \Big(1+2c_T+\tfrac{1}{n}\Big) \INTDom{M(t)}{[0,T] \setminus \Opazo_n}{t} \bigg) ~\underset{n \to +\infty}{\longrightarrow}~ 0. 
\end{aligned}
\end{equation}
Recalling now that the curve $\mu_n(\cdot) \in \AC([0,T],\Pcal_1(\R^d))$ takes the specific form 
\begin{equation*}
\mu_n(t) = \big( \Phi_{(a_i^n,t)}^{v_n} \big)_{\sharp} \mu_n(a_i^n)
\end{equation*}
for all times $t \in [a_i^n,b_i^n)$ with $a_i^n \in C_{\epsilon_n}$, one can show through a simple integration by parts that
\begin{equation*}
\begin{aligned}
\xi(b_i^n) \INTDom{ \psi(x)}{\R^d}{\mu_n(b_i^n)(x)} & = \xi(a_i^n) \INTDom{\psi(x)}{\R^d}{\mu_n(a_i^n)(x)} \\
& \hspace{0.45cm} + \INTSeg{\int_{\mathbb R^d}\Big( \xi'(t) \psi(x) + \big\langle \xi(t) \nabla \psi(x) , v_n(t,x) \big\rangle \Big) \mathrm{d} \mu_n(t)(x)}{t}{a_i^n}{b_i^n}
\end{aligned}
\end{equation*}
for every $i \in \Lambda_n$ such that $a_i^n \in C_{\epsilon_n}$. Since $\Opazo_n = \bigcup_{a_i^n \in C_{\epsilon_n}} (a_i^n,b_i^n)$ is made of pairwise disjoint intervals, the latter identity allows us to further infer that 
\begin{equation}
\label{eq:LimitCE21}
\begin{aligned}
& \bigg| \INTDom{\INTDom{\Big( \xi'(t) \psi(x) + \langle \xi(t) \nabla \psi(x) , v_n(t,x) \rangle \Big)}{\R^d}{\mu_n(t)(x)}}{\Opazo_n}{t} \, \bigg| \\
& \hspace{0.8cm} = \bigg| \sum_{a_i^n \in C_{\epsilon_n}} \INTSeg{\INTDom{\Big( \xi'(t) \psi(x) + \big\langle \xi(t) \nabla \psi(x) , v_n(t,x) \big\rangle \Big)}{\R^d}{\mu_n(t)(x)}}{t}{a_i^n}{b_i^n} \, \bigg| \\
& \hspace{0.8cm}  = \bigg| \sum_{a_i^n \in C_{\epsilon_n}} \bigg( \xi(b_i^n) \INTDom{ \psi(x)}{\R^d}{\mu_n(b_i^n)(x)} - \xi(a_i^n) \INTDom{\psi(x)}{\R^d}{\mu_n(a_i^n)(x)} \bigg) \bigg| \\
& \hspace{0.8cm} \leq \bigg| \sum_{i \in \Lambda_n} \bigg( \xi(b_i^n) \INTDom{ \psi(x)}{\R^d}{\mu_n(b_i^n)(x)} - \xi(a_i^n) \INTDom{\psi(x)}{\R^d}{\mu_n(a_i^n)(x)} \bigg) \bigg|  \\
& \hspace{1.3cm} + \bigg| \sum_{a_i^n \notin C_{\epsilon_n}} \bigg( \xi(b_i^n) \INTDom{ \psi(x)}{\R^d}{\mu_n(b_i^n)(x)} - \xi(a_i^n) \INTDom{\psi(x)}{\R^d}{\mu_n(a_i^n)(x)} \bigg) \bigg|
\end{aligned}
\end{equation}
where, in the last inequality, we added and subtracted the sum for $a_i^n \notin C_{\epsilon_n}$ and used the triangle inequality. Upon recalling that the intervals $\{[a_i^n,b_i^n)\}_{i \in \Lambda_n}$ are pairwise disjoint and such that $\bigcup_{i \in \Lambda_n} [a_i^n,b_i^n) = [0,T)$, one may leverage estimate \eqref{eq:WassEstBis} together with Lemma \ref{techlem}-$(b)$ to show that
\begin{equation*}
\begin{aligned}
& \sum_{i \in \Lambda_n} \bigg| \, \xi(b_i^n) \INTDom{ \psi(x)}{\R^d}{\mu_n(b_i^n)(x)} - \xi(a_i^n) \INTDom{\psi(x)}{\R^d}{\mu_n(a_i^n)(x)} \bigg| \\
& \hspace{0.8cm} \leq \NormC{\xi}{1}{(0,T)} \NormC{\psi}{1}{\R^d} \sum_{i \in \Lambda_n} \bigg( (b_i^n-a_i^n) + \INTSeg{\max\Big\{ 2(1+c_T) M(t) , M_{\Qpazo}(t) \Big\}}{t}{a_i^n}{b_i^n} \bigg) \\
& \hspace{0.8cm} = \NormC{\xi}{1}{(0,T)} \NormC{\psi}{1}{\R^d} \bigg( T + \INTSeg{\max\Big\{ 2(1+c_T) M(t) , M_{\Qpazo}(t) \Big\}}{t}{0}{T} \bigg) < +\infty,  
\end{aligned}
\end{equation*}
from whence we deduce that the above series indexed by $\Lambda_n$ is absolutely convergent. Define
\begin{equation*}
f_n(t):=\xi(t)\INTDom{\psi(x)}{\R^d}{\mu_n(t)(x)}
\end{equation*}
for all times $t\in[0,T]$. It follows from Lemma \ref{techlem}-$(b)$ combined with \eqref{eq:WassEstBis} that $f_n\in\AC([0,T],\R)$. Hence, using the absolute continuity of $f_n(\cdot)$, the fact that the intervals $\{[a_i^n,b_i^n)\}_{i\in\Lambda_n}$ partition $[0,T)$ up to their endpoints and the countable additivity of the Lebesgue integral, we obtain
\begin{equation}
\label{eq:LimitCE22}
\begin{aligned}
& \sum_{i \in \Lambda_n} \bigg( \xi(b_i^n) \INTDom{ \psi(x)}{\R^d}{\mu_n(b_i^n)(x)} - \xi(a_i^n) \INTDom{\psi(x)}{\R^d}{\mu_n(a_i^n)(x)} \bigg) \\
& \hspace{0.5cm} = \sum_{i\in\Lambda_n}\big(f_n(b_i^n)-f_n(a_i^n)\big)
=\sum_{i\in\Lambda_n}\int_{a_i^n}^{b_i^n}f_n'(t)\,\mathrm{d}t
=\int_0^T f_n'(t)\,\mathrm{d}t \\
& \hspace{0.5cm} = f_n(T)-f_n(0)
=\xi(T) \INTDom{ \psi(x)}{\R^d}{\mu_n(T)(x)} - \xi(0) \INTDom{\psi(x)}{\R^d}{\mu_n(0)(x)} = 0
\end{aligned}
\end{equation}
since $\xi \in C^{\infty}_c((0,T))$ by assumption. Thence, upon merging \eqref{eq:LimitCE21} and \eqref{eq:LimitCE22}, we further obtain 
\begin{equation}
\label{eq:LimitCE23}
\begin{aligned}
& \bigg| \INTDom{\INTDom{\Big( \xi'(t) \psi(x) + \big\langle \xi(t) \nabla \psi(x) , v_n(t,x) \big\rangle \Big)}{\R^d}{\mu_n(t)(x)}}{\Opazo_n}{t} \, \bigg| \\
& \hspace{1cm} \leq \bigg| \sum_{a_i^n \notin C_{\epsilon_n}} \bigg( \xi(b_i^n) \INTDom{ \psi(x)}{\R^d}{\mu_n(b_i^n)(x)} - \xi(a_i^n) \INTDom{\psi(x)}{\R^d}{\mu_n(a_i^n)(x)} \bigg) \bigg| \\
& \hspace{1cm} \leq \NormC{\xi}{1}{(0,T)} \NormC{\psi}{1}{\R^d} \sum_{a_i^n \notin C_{\epsilon_n}} \INTSeg{\bigg( 1 + \max\Big\{ 2(1+c_T) M(t) , M_{\Qpazo}(t) \Big\} \bigg)}{t}{a_i^n}{b_i^n} \\
& \hspace{1cm} \leq \NormC{\xi}{1}{(0,T)} \NormC{\psi}{1}{\R^d} \INTDom{\bigg( 1 + \max\Big\{ 2(1+c_T) M(t) , M_{\Qpazo}(t) \Big\} \bigg)}{[0,T] \setminus \Opazo_n}{t} ~\underset{n \to +\infty}{\longrightarrow}~ 0,
\end{aligned}
\end{equation}
where we used again the estimates from Lemma \ref{techlem}-$(b)$ together with the fact that 
\begin{equation*}
\bigcup_{a_i^n \notin C_{\epsilon_n}} [a_i^n,b_i^n) \subseteq ([0,T] \setminus \Opazo_n)
\end{equation*}
for each $n \geq 1$. By combining \eqref{eq:LimitCE1} and \eqref{eq:LimitCE23}, we finally get that 
\begin{equation}
\label{eq:LimitCE0Bis}
\INTSeg{\INTDom{\Big( \partial_t \varphi(t,x) + \langle \nabla_x \varphi(t,x) , v_n(t,x) \rangle \Big)}{\R^d}{\mu_n(t)(x)}}{t}{0}{T} ~\underset{n \to +\infty}{\longrightarrow}~ 0
\end{equation}
for each $\varphi \in C^{\infty}_c((0,T) \times \R^d)$ of the form \eqref{eq:TestFunction}. Besides, it also follows from the Kantorovich--Rubinstein estimate \eqref{eq:WassEstBis} combined with Hypothesis \ref{hyp:USC}-$(ii)$ and Lemma \ref{techlem}-$(b)$ above that 
\begin{equation}
\label{eq:LimitCE3}
\begin{aligned}
& \bigg| \INTSeg{\INTDom{\Big( \partial_t \varphi(t,x) + \langle \nabla_x \varphi(t,x) , v_n(t,x) \rangle \Big)}{\R^d}{(\mu(t) - \mu_n(t))(x)}}{t}{0}{T} \, \bigg| \\
& \hspace{0cm} \leq \bigg( \NormL{\xi'}{1}{0,T} \Lip(\psi) + \NormC{\xi}{0}{(0,T)} \INTSeg{\Lip \big(\langle \nabla \psi , v_n(t) \rangle \big)}{t}{0}{T} \bigg) \sup_{t \in [0,T]} W_1(\mu_n(t),\mu(t)) ~\underset{n \to +\infty}{\longrightarrow}~ 0.
\end{aligned}
\end{equation}
At this stage, note that $\{v_n(\cdot)\}_{n \in \N}$ satisfies the assumptions of Lemma \ref{lem:comvf} as a consequence of Hypothesis \ref{hyp:USC}-$(ii)$. Hence, upon remarking that the map $t \in [0,T] \mapsto \Bnu_{\varphi}(t) \in \Mcal_c(\R^d,\R^d)$ defined through the duality action
\begin{equation*}
\langle \Bnu_{\varphi}(t) , v \rangle_{C^0(\R^d,\R^d)}
:= \INTDom{\langle \nabla_x \varphi(t,x) , v(x) \rangle}{\R^d}{\mu(t)(x)}
\end{equation*}
for every $v \in C^0(\R^d,\R^d)$ and all times $t \in [0,T]$ is scalarly-$^*$ measurable, one may find a further subsequence that we do not relabel along with a measurable map $v : [0,T] \to C^0(\R^d,\R^d)$ such that
\begin{equation}
\label{eq:LimitCE4}
\INTSeg{\INTDom{\big\langle \nabla_x \varphi(t,x) , v(t,x) - v_n(t,x) \big\rangle}{\R^d}{\mu(t)(x)}}{t}{0}{T} ~\underset{n \to +\infty}{\longrightarrow}~ 0.
\end{equation}
Therefore, upon combining \eqref{eq:LimitCE0Bis}, \eqref{eq:LimitCE3} and \eqref{eq:LimitCE4}, we finally obtain 
\begin{equation*}
\INTSeg{\INTDom{\Big( \partial_t \varphi(t,x) + \langle \nabla_x \varphi(t,x) , v(t,x) \rangle \Big)}{\R^d}{\mu(t)(x)}}{t}{0}{T} = 0
\end{equation*}
for all $\varphi \in C^{\infty}_c((0,T) \times \R^d)$ of the form \eqref{eq:TestFunction}. Lastly, because the linear span of all such test functions is dense in $C^{\infty}_c((0,T) \times \R^d)$ (see e.g. \cite[Chapter 8]{AGS}), we may conclude that the curve $\mu(\cdot) \in \AC([0,T],\Pcal_1(\R^d))$ satisfies
\begin{equation*}
\left\{
\begin{aligned}
& \partial_t \mu(t) + \Div_x(v(t)\mu(t)) = 0, \\
& \mu(0) = \mu_0,
\end{aligned}
\right.
\end{equation*}
which is the desired claim. 

\medskip

\noindent \textbf{Step 3.} (The limit curve satisfies the viability constraint and continuity inclusion). To begin with, note that $\mu(\cdot) \in \AC([0,T],\Pcal_1(\R^d))$ satisfies the viability constraint as a simple consequence of Lemma \ref{techlem}-$(c)$ together with \eqref{eq:Unieq:IntegralViabilityCondvCurve}, since 
\begin{equation*}
W_1(\mu(t) \, ; \Qpazo(t)) \leq W_1(\mu(t),\mu_n(t)) + W_1(\mu_n(t) \, ; \Qpazo(t)) ~\underset{n \to +\infty}{\longrightarrow}~ 0 
\end{equation*}
for all times $t \in [0,T]$. Since $\Qpazo(t)$ is proximal and thus closed, this implies that $\mu(t) \in \Qpazo(t)$ for all times $t \in [0,T]$. Regarding the continuity inclusion, what we need to show is that $v(t) \in V(t,\mu(t))$ for $\Lcal^1$-almost every $t \in [0,T]$. We start by observing that the integrable selections $t \in [0,T] \mapsto v_n(t) \in V(t,\tilde{\mu}_n(t))$ are such that
\begin{equation*}
\INTSeg{\xi(t) \langle \Bnu , v_n(t) \rangle_{C^0(\R^d,\R^d)}}{t}{0}{T} ~\underset{n \to +\infty}{\longrightarrow}~ \INTSeg{\xi(t) \langle \Bnu , v(t) \rangle_{C^0(\R^d,\R^d)}}{t}{0}{T}
\end{equation*}
for every $\Bnu \in \Mcal_c(\R^d,\R^d)$ and each $\xi(\cdot) \in L^{\infty}([0,T],\R)$, as a simple consequence of Lemma \ref{lem:comvf}. Moreover, it holds for $\Lcal^1$-almost every $t \in [0,T]$ that 
\begin{equation*}
W_{1}(\tilde{\mu}_n(t),\mu(t)) \leq W_1(\tilde{\mu}_n(t),\mu_n(t)) + W_1(\mu_n(t),\mu(t)) ~\underset{n \to +\infty}{\longrightarrow} 0, 
\end{equation*} 
and it thus follows from the closure principle of Proposition \ref{prop:CastaingValadier} above that $v(t) \in V(t,\mu(t))$ for $\Lcal^1$-almost every $t \in [0,T]$, thereby closing the proof. \hfill $\square$

\begin{remark}[Concerning the regularity assumptions in Hypothesis \ref{hyp:USC}]
\label{rmk:Hyp}
As things stand, we do not expect the working assumptions of this section to be optimal. Indeed, similarly e.g. to the Peano existence result from \cite[Section 3]{ContIncPp}, one would expect that a sharp regularity framework on the driving fields should entail roughly the same smoothness in the space and measure variables. Here, that would mean working with velocity fields which are continuous (or ideally less regular) with respect to $x \in \R^d$, instead of Lipschitz. The difficulty in doing so lies in the fact that, contrary to the setting of $\R^d$ in which evolving an ODE with a velocity depending only on time leaves constant the relative distance between two trajectories, the evolution of two measures under the same vector field starting from different initial conditions does not yield any nice distance estimate in general, except when the former is Lipschitz. This limitation stems directly from the fact that, although continuity equations behave essentially like ODEs, the Wasserstein spaces are not flat. Hence, to obtain general viability results under the expected minimal assumptions mimicking those of \cite{Frankowska1995}, one would likely need to translate and readapt the strategy detailed in Section \ref{section:USC} in terms of Wasserstein geometry \cite{AGS}, to have access to other manifold-like constructs such as exponential mappings or parallel transport (see e.g. \cite{Ambrosio2008,Gigli2012}), which might allow to recover sharp regularity assumptions. 
\end{remark}


\appendix


\section{On the (lack of) superdifferentiability of the 1-Wasserstein distance}
\label{App1}

In this section, we illustrate the lack of superdifferentiability of the 1-Wasserstein distance. It is well-known (see e.g. \cite[Section 10.2]{AGS}) that the \(p\)-Wasserstein distance with \(p > 1\) supports nice joint superdifferentiability estimates, which play a pivotal role in the viability proofs of \cite{ViabPp}. In the following proposition, we show that these superdifferentiability estimates fail for the \(1\)-Wasserstein distance. Recall that the $1$-duality mapping $j_{1}:\mathbb R^d\rightrightarrows\mathbb R^d$ is defined by 
\begin{equation*}
j_1(x):= \begin{cases*}
\displaystyle\frac{x}{|x|} & if $x\neq 0$, \\
\big\{s\in\mathbb R^d ~\,\textnormal{s.t.}~ |s|\le1\big\}      & if $x=0$.
\end{cases*}
\end{equation*}
for all $x \in \R^d$, and that any selection of this mapping is automatically Borel. 

\begin{proposition}[Lack of superdifferentiability for the 1-Wasserstein distance]
Let $\xi:\mathbb R^d\to \mathbb R^d$ and $\zeta:\mathbb R^d\to\mathbb R^d$ be Borel maps for which there exists some $\bar{x}\in \mathbb R^d$ such that $\xi(\bar{x}) \neq \zeta(\bar{x})$. Then, there exist $\mu,\nu\in\mathcal P_1(\mathbb R^d)$ such that 
\begin{equation*}
\liminf_{h\to0^+}\frac{1}{h} \bigg( W_1\Big((\Id + h \xi)_\sharp\mu, (\Id + h \zeta)_\sharp\nu\Big) - W_1(\mu,\nu) - h \displaystyle\int_{\mathbb R^{2d}}\langle \xi(x) - \zeta(y), s(x-y)\rangle\mathrm{d}\gamma(x,y) \bigg) > 0
\end{equation*}
for every Borel selection $z \in \R^d \mapsto s(z) \in j_1(z)$ satisfying $|s(0)| < 1$ and any $\gamma \in \Gamma_{o}(\mu,\nu)$.
\end{proposition}

\begin{proof}
Take $\bar{x} \in \R^d$ such that $\xi(\bar{x}) \neq \zeta(\bar{x})$, set $a:=\xi(\bar{x})-\zeta(\bar{x})\neq0$, choose simply $\mu = \nu = \delta_{\bar{x}}$, and observe that 
\begin{equation*}
\begin{aligned}
\liminf_{h\to0^+}\frac{1}{h} \bigg( W_1\Big((\Id + h \xi)_\sharp\mu, (\Id + h \zeta)_\sharp\nu\Big) - W_1(\mu,\nu) & - h \displaystyle\int_{\mathbb R^{2d}}\langle \xi(x) - \zeta(y), s(x-y)\rangle\mathrm{d}\gamma(x,y) \bigg) \\
&  = |a|-\langle a,s(0)\rangle \geq |a|\big(1-|s(0)|\big) > 0
\end{aligned}
\end{equation*}
provided $|s(0)| < 1$, hence the conclusion. 
\end{proof}


\section{On the necessity of absolute left or right continuity for the viability of time-dependent constraints}
\label{App2}

In this second appendix, we show that if an $\Lcal^1$-measurable set-valued mapping $\Qpazo:[0,T]\rightrightarrows\mathcal P_1(\mathbb R^d)$ with proximal values is forward or backward viable for the continuity inclusion
\begin{equation}
\label{eq:ContIncApp}
\partial_t \mu(t) \in - \Div_x \Big( V(t,\mu(t)) \, \mu(t) \Big), 
\end{equation}
then it is necessarily left or right absolutely continuous. 

\begin{proposition}[Necessary regularity conditions for viable maps]	Let $V:[0,T]\times\mathcal P_1(\mathbb R^d)\rightrightarrows C^0\big(\mathbb R^d,\mathbb R^d\big)$ be a set-valued mapping such that $t \in [0,T] \mapsto V(t,\mu)$ is $\Lcal^1$-measurable for all $\mu \in \Pcal_1(\R^d)$, and suppose that there exists a function $M(\cdot)\in L^1([0,T],\R_+)$ such that for $\Lcal^1$-almost every $t \in [0,T]$, one has that 
\begin{equation*}
|v(x)| \le M(t) \Big( 1 + |x| + \Mpazo_1(\mu) \Big) 
\end{equation*}
for all $x \in \R^d$ and every $(\mu,v) \in \Graph(V(t))$. Then, the following hold. 
\begin{itemize}
\item[$(a)$] If for every $\tau\in[0,T]$ and any $\mu_{\tau}\in \Qpazo(\tau)$ there exists a solution of \eqref{eq:ContIncApp} satisfying 
\begin{equation*}
\mu(\tau) = \mu_{\tau} \qquad \text{and} \qquad \mu(t)\in\Qpazo(t) \quad \text{for all times $t \in [\tau,T]$}, 
\end{equation*}
then $\Qpazo:[0,T]\rightrightarrows\mathcal P_1(\mathbb R^d)$ is left absolutely continuous.
\item[$(b)$] If for  every $\tau\in[0,T]$ and any $\mu_{\tau}\in \Qpazo(\tau)$ there exists a solution of \eqref{eq:ContIncApp} satisfying 
\begin{equation*}
\mu(\tau) = \mu_{\tau} \qquad \text{and} \qquad \mu(t)\in\Qpazo(t) \quad \text{for all times $t \in [0,\tau]$}, 
\end{equation*}
then $\Qpazo:[0,T]\rightrightarrows\mathcal P_1(\mathbb R^d)$ is right absolutely continuous.
\item[$(c)$] If for  every $\tau\in[0,T]$ and any $\mu_{\tau}\in \Qpazo(\tau)$ there exists a solution of \eqref{eq:ContIncApp} satisfying 
\begin{equation*}
\mu(\tau) = \mu_{\tau} \qquad \text{and} \qquad \mu(t)\in\Qpazo(t) \quad \text{for all times $t \in [0,T]$},  
\end{equation*}
then $\Qpazo:[0,T]\rightrightarrows\mathcal P_1(\mathbb R^d)$ is absolutely continuous.
\end{itemize}
\end{proposition}

\begin{proof}
We only prove the first statement as the others are analogous. To this end, fix some $\mu_0\in\mathcal P_1(\mathbb R^d)$ and $r>0$, and let $0 \leq s \leq t \leq T$. Then, for any $\mu_s\in\Qpazo(s)\cap \mathbb B(\mu_0,r)$, denote by $\mu(\cdot) \in \AC([s,T],\Pcal_1(\R^d))$ a solution of \eqref{eq:ContIncApp} satisfying the requirements of item $(a)$ with $\tau=s$, and by $\theta \in [s,T] \mapsto v(\theta) \in V(\theta,\mu(\theta))$ the corresponding velocity selection.  Then, the growth assumption made on $V : [0,T] \times \Pcal_1(\R^d) \rightrightarrows C^0(\R^d,\R^d)$ combined with the moment estimate of \eqref{eq:BasicEstimates} gives
\begin{equation*}
\begin{aligned}
\Mpazo_1(\mu(t)) & \leq \Mpazo_1(\mu_s) + \INTSeg{\NormL{v(\theta)}{1}{\R^d,\R^d ; \, \mu(\theta)}}{\theta}{s}{t} \\
& \leq \Mpazo_1(\mu_s) + \INTSeg{M(\theta) \Big( 1 + 2 \Mpazo_1(\mu(\theta)) \Big)}{\theta}{s}{t},
\end{aligned}
\end{equation*}
so that, since $\Mpazo_1(\mu_s) \leq \Mpazo_1(\mu_0) + r$, Gr\"onwall's lemma yields
\begin{equation*}
\sup_{\theta \in [s,T]} \Mpazo_1(\mu(\theta)) \leq c_{\mu_0,r,T} := \bigg( \Mpazo_1(\mu_0) + r + \INTSeg{M(t)}{t}{0}{T} \bigg) \exp \bigg( 2 \INTSeg{M(t)}{t}{0}{T} \bigg).
\end{equation*}
Observe in particular that the constant $c_{\mu_0,r,T} > 0$ does not depend on the particular choice of $\mu_s \in \Qpazo(s) \cap \B(\mu_0,r)$. Plugging the latter bound into the distance estimate of \eqref{eq:BasicEstimates}, we obtain
\begin{equation*}
W_1(\mu_s \, ; \Qpazo(t)) \leq W_1(\mu_s,\mu(t)) \leq \big( 1 + 2 c_{\mu_0,r,T} \big) \INTSeg{M(\theta)}{\theta}{s}{t}.
\end{equation*}
Since $\mu_s \in \Qpazo(s) \cap \B(\mu_0,r)$ is arbitrary, we conclude that $\Delta_{\mu_0,r}(\Qpazo(s),\Qpazo(t)) \le (1+2c_{\mu_0,r,T}) \hspace{-0.05cm}\INTSeg{M(\theta)}{\theta}{s}{t}$, that being true for any choice of $\mu_0 \in \Pcal_1(\R^d)$ and $r > 0$.
\end{proof}


\section{Proof of Proposition \ref{prop:CastaingValadier}}
\label{AppCastaingValadier}

In this Appendix, we collect the abstract closure principle \cite[Theorem VI-4]{Castaing1977} and explain how it particularizes to Proposition \ref{prop:CastaingValadier} above. We refer to \cite[Chapter III]{Castaing1977} or \cite[Part I]{Rudin1991} for the definitions of the main objects appearing in the following statement.

\begin{theorem}[The Castaing-Valadier closure principle]
\label{thm:CastaingValadier}
Consider a topological space $(U,\tau_U)$ along with a locally convex Hausdorff topological vector space $(E,\tau_E)$ whose continuous dual $E'$ admits a countable set $\{e_n'\}_{n \in \N}$ that separates the points of $E$. Let furthermore $\Gpazo : I \times U \rightrightarrows E$ be a set-valued map with compact convex images such that $t \in I \rightrightarrows \Gpazo(t,x)$ is $\Lcal^1$-measurable for every $x\in U$, and $x \in U \rightrightarrows \Gpazo(t,x)$ is upper semicontinuous for $\Lcal^1$-almost every $t\in I$, and fix two sequences of maps $\{x_n(\cdot)\}_{n \in \N}$ and $\{y_n(\cdot)\}_{n \in \N}$ satisfying the following.   
\begin{enumerate}
\item[$(i)$] The functions $x_n : I \to U$ are such that 
\begin{equation*}
x_n(t) ~\underset{n \to +\infty}{\longrightarrow}~ x(t)
\end{equation*}
for $\Lcal^1$-almost every $t \in I$ and some map $x : I \to U$. 
\item[$(ii)$] The functions $y_n : I \to E$ are $\Lcal^1$-scalarly integrable and such that 
\begin{equation*}
\INTDom{\xi(t) \langle e' , y(t) - y_n(t) \rangle_E \,}{I}{t} ~\underset{n \to +\infty}{\longrightarrow}~ 0
\end{equation*}
for every $e' \in E'$, each $\xi(\cdot) \in L^{\infty}(I,\R)$, and some $\Lcal^1$-scalarly measurable function $y : I \to E$.  
\item[$(iii)$] It holds that $y_n(t) \in \Gpazo(t,x_n(t))$ for $\Lcal^1$-almost every $t \in I$ and each $n \geq 1$.
\end{enumerate}
Then, one has that $y(t) \in \Gpazo(t,x(t))$ for $\Lcal^1$-almost every $t \in I$.
\end{theorem}

\begin{proof}[Proof of Proposition \ref{prop:CastaingValadier}]
In our context, we let $(U,\tau_U)$ be the complete separable metric space $(\Pcal_1(\R^d),W_1(\cdot,\cdot))$ and $(E,\tau_E)$ be the separable Fréchet space $(C^0(\R^d,\R^d),\dsf_{cc}(\cdot,\cdot))$, whose continuous dual is isomorphic to the space of compactly supported vector-valued measures $\Mcal_c(\R^d,\R^d)$ through the pairing
\begin{equation*}
\langle \Bnu , \varphi \rangle_{C^0(\R^d,\R^d)} = \sum_{i=1}^d \INTDom{\varphi_i(x)}{\R^d}{\Bnu_i(x)}
\end{equation*}
for each $\Bnu \in \Mcal_c(\R^d,\R^d)$ and every $\varphi \in C^0(\R^d,\R^d)$, see e.g. \cite[Page 155 -- Proposition 14]{Bourbaki2007}. Furthermore, recall that the dual of every separable Fréchet space automatically supports a countable dense set which separates points, see e.g. \cite[p. 259]{K_1969}. By \cite[Page 83 -- Theorem III.36]{Castaing1977}, in separable Fréchet spaces, scalar measurability coincides with the standard measurability with respect to the native Borel $\sigma$-algebra induced by the metric. Hence, a mapping $v:I\to C^0\big(\mathbb R^d,\mathbb R^d\big)$ is measurable if and only if it is scalarly measurable, whereas the measurability of a map $\mu : I \to \Pcal_1(\R^d)$ is understood in the usual sense.

The integrals in item $(ii)$ of Proposition \ref{prop:CastaingValadier} being understood as finite Lebesgue integrals, each map $t \in I \mapsto (v(t)-v_n(t)) \in C^0(\R^d,\R^d)$ is scalarly integrable. In particular, $v(\cdot)$ is scalarly measurable and thus measurable by the preceding discussion. Consider now the set-valued map
\begin{equation*}
\Hpazo(t,\mu):=V(t,\mu)-v(t),
\end{equation*}
and note that it has compact convex images, is measurable with respect to $t \in I$ for every fixed $\mu\in\Pcal_1(\R^d)$, and is upper semicontinuous with respect to $\mu \in \Pcal_1(\R^d)$ for $\Lcal^1$-almost every $t\in I$. We may therefore apply Theorem \ref{thm:CastaingValadier} with $x_n:=\mu_n$, $x=\mu$, $y_n=v_n-v$ and $y=0$. Indeed, item $(ii)$ of Proposition \ref{prop:CastaingValadier} yields the required scalar weak convergence, while item $(iii)$ gives
\begin{equation*}
(v_n(t)-v(t))\in\Hpazo(t,\mu_n(t))
\end{equation*}
for $\Lcal^1$-almost every $t\in I$. It follows that $0\in\Hpazo(t,\mu(t))$, or equivalently $v(t)\in V(t,\mu(t))$, for $\Lcal^1$-almost every $t\in I$, which concludes the proof.
\end{proof}


\bibliographystyle{amsplain}
\bibliography{references}




\end{document}